\documentclass{amsart}
\usepackage{dsfont}
\usepackage[dvipsnames,svgnames,table]{xcolor}

\usepackage{tikz}
\usepackage{caption}

\usepackage[utf8]{inputenc}
\usepackage[T1]{fontenc}
\usepackage{ tipa }
\usepackage{helvet}
\usepackage{mathrsfs}  
\usepackage{stmaryrd}  
\usepackage{amsmath,amsxtra,amsthm,amssymb,xr,enumerate,fullpage,comment, graphicx, mathtools}
\usepackage[all]{xy}
\usepackage[bbgreekl]{mathbbol}
\usepackage{bbm}
\usepackage{hyperref}

\newcommand{\refsymbolA}{{\ensuremath(A)}}
\newcommand{\refsymbolB}{{\ensuremath(B)}}

\makeatletter
\newcommand{\mylabel}[2]{#2\def\@currentlabel{#2}\label{#1}}
\makeatother

 \definecolor{pAlgae}{RGB}{87,115,135}
\definecolor{airforceblue}{rgb}{0.36, 0.54, 0.66}
	\definecolor{bondiblue}{rgb}{0.0, 0.58, 0.71}
\definecolor{britishracinggreen}{rgb}{0.0, 0.26, 0.15}
\definecolor{camouflagegreen}{rgb}{0.47, 0.53, 0.42}
\definecolor{darkcyan}{rgb}{0.0, 0.55, 0.55}

\DeclareMathSymbol{\invques}{\mathord}{operators}{`>}
\DeclareUnicodeCharacter{00BF}{\tmquestiondown}
\DeclareRobustCommand{\tmquestiondown}{%
  \ifmmode\invques\else\textquestiondown\fi
}

\DeclareRobustCommand\iff{\;\Longleftrightarrow\;}

\usepackage{verbatim}

\hypersetup{
    colorlinks = true,
    linkcolor=blue,
     filecolor=blue,
     citecolor = darkcyan,      
     urlcolor=cyan,
    linkbordercolor = {white},
}

\makeatletter
\newcommand{\pushright}[1]{\ifmeasuring@#1\else\omit\hfill$\displaystyle{#1}$\fi\ignorespaces}
\newcommand{\pushleft}[1]{\ifmeasuring@#1\else\omit$\displaystyle{#1}$\hfill\fi\ignorespaces}
\makeatother

\theoremstyle{plain}

\newtheorem{theorem}{Theorem}[section]
\newtheorem{corollary}[theorem]{Corollary}
\newtheorem{lemma}[theorem]{Lemma}
\newtheorem{proposition}[theorem]{Proposition}

\newtheorem{conj}[theorem]{Conjecture}

\theoremstyle{definition}
\newtheorem{definition}[theorem]{Definition}

\newtheorem{examplewr}[theorem]{Example}
\newtheorem{ass}[theorem]{Assumption}

\theoremstyle{remark}
\newtheorem{obswr}[theorem]{Observation}
\newtheorem{remarkwr}[theorem]{Remark}

\newenvironment{remark}{\begin{remarkwr}\begin{upshape}}{\end{upshape}\end{remarkwr}}

\newenvironment{smatrix}{\left(\begin{smallmatrix}}{\end{smallmatrix}\right)}

\DeclareMathOperator{\ad}{ad}

\DeclareMathOperator{\BF}{BF}

\DeclareMathOperator{\Char}{Char}
\DeclareMathOperator{\cl}{cl}
\DeclareMathOperator{\coker}{coker}
\DeclareMathOperator{\cone}{cone}

\DeclareMathOperator{\cris}{cris}

\DeclareMathOperator{\cyc}{cyc}

\DeclareMathOperator{\ES}{ES}

\DeclareMathOperator{\fin}{f}
\DeclareMathOperator{\Frac}{Frac}
\DeclareMathOperator{\fs}{fs}
\DeclareMathOperator{\Gr}{Gr}
\DeclareMathOperator{\gr}{\mathscr Gr}

\DeclareMathOperator{\GSp}{GSp}

\DeclareMathOperator{\im}{im}

\DeclareMathOperator{\Iw}{Iw}
\DeclareMathOperator{\Kl}{Kl}

\DeclareMathOperator{\Log}{Log}

\DeclareMathOperator{\Si}{Si}

\DeclareMathOperator{\Spf}{Spf}

\DeclareMathOperator{\wt}{wt}

\newcommand{\Q}{\mathbb{Q}}
\newcommand{\Z}{\mathbb{Z}}

\newcommand{\GL}{\mathrm{GL}}

\newcommand{\Fil}{\mathrm{Fil}}
\newcommand{\rank}{\mathrm{rank}}

\newcommand{\Fr}{\mathrm{Fr}}
\newcommand{\diag}{\mathrm{diag}}

\newfont{\gotip}{eufb10 at 12pt}

\newcommand{\cO}{{\mathcal O}}

\newcommand{\cR}{{\mathcal{R}}}

\newcommand{\GG}{{\mathcal G}}

\newcommand{\ra}{\rightarrow}
\newcommand{\lra}{\longrightarrow}

\newcommand{\tr}{{\mathrm{tr}}}

\newcommand{\cX}{\mathcal X}

\newcommand{\hk}{{\mathbf{k}}}
\newcommand{\hr}{{\mathbf r}}

\newcommand{\DD}{{\mathbf D}}

\newcommand{\uf}{{\underline{\mathbf f}}}
\newcommand{\ug}{{\underline{\mathbf g}}}
\newcommand{\uh}{{\underline{\mathbf h}}}

\newcommand{\upi}{{\underline{\Pi}}}

\DeclareMathOperator{\Hom}{Hom}

\newcommand{\QQ}{\mathbb{Q}}
\newcommand{\Qp}{{\mathbb{Q}_p}}

\newcommand{\ZZ}{\mathbb{Z}}
\newcommand{\CC}{\mathbb C}

\newcommand{\res}{\mathrm{res}}

\newcommand{\fp}{\mathfrak{p}}
\newcommand{\fq}{\mathfrak{q}}

\numberwithin{equation}{section}

\begin{document}
\include{thebibliography}

\title{W\lowercase{all-crossing and $p$-adic} A\lowercase{rtin formalism for} $\GSp_4 \times \GL_2 \times \GL_2$}

\author{K\^azim B\"uy\"ukboduk, \'Oscar Rivero, Ryotaro Sakamoto}

\begin{abstract}
The goal of this article is to develop a $p$-adic Artin formalism in the context of $p$-adic families of automorphic forms on $\GSp_4 \times \GL_2 \times \GL_2$. Our treatment is guided by the (double) \emph{wall-crossing principle}, emphasising an interplay between arithmetic GGP and $p$-adic explicit GGP formulae. Although the picture we present remains largely conjectural, we provide evidence in favour of our conjectures (a) in terms of algebraic $p$-adic $L$-functions, and (b) in endoscopic scenarios.
\end{abstract}

\date{\today}

\address{K. B.: UCD School of Mathematics and Statistics, University College Dublin, Ireland}
\email{kazim.buyukboduk@ucd.ie}

\address{R.S.: Department of Mathematics\\University of Tsukuba\\1-1-1 Tennodai\\Tsukuba\\Ibaraki 305-8571\\Japan}
\email{rsakamoto@math.tsukuba.ac.jp}

\address{O. R.: CITMAga and Departamento de Matem\'aticas (Universidade de Santiago de Compostela), Santiago de Compostela, Spain}
\email{oscar.rivero@usc.es}


\maketitle

\setcounter{tocdepth}{1}
\tableofcontents

\section{Introduction}
\label{sec_Intro}
Let $\uf$ and $\ug$ be primitive Hida families of cuspidal eigenforms. In their seminal work~\cite{BDP}, Bertolini--Darmon--Prasanna construct a $p$-adic $L$-function $L_p^{\rm BDP}(\ug\times \uf)$ that interpolates the central critical Rankin--Selberg $L$-values $L(g\times f, c)$ (suitably normalised by a period) when $\ug$ is a CM family and the specialisation $g$ of $\ug$ has higher weight than the speecialisation $f$ of $\uf$. The construction of this $p$-adic $L$-function relies on another fundamental result, which is an extension of the celebrated work of Waldspurger, that expresses the indicated $L$-value as the square of a toric period (which is non-zero only when $\varepsilon_v(g\times f,c)=+1$ for all local epsilon factors), which can be reduced to a finite sum involving CM points. This allowed them to prove their celebrated ``BDP'' ($p$-adic Waldspurger) formula\footnote{This has been generalised in \cite{LZZ2018} to cover the case of forms on non-split quaternion algebras.}, which roughly reads
$$L_p^{\rm BDP}(\ug\times\uf)_{\vert_{g\times f}}\,=\,\lim_{\substack{(\uf_\kappa,\ug_\lambda)\to (f,g)\\
\wt(\kappa)<\wt(\lambda)}} \mathcal{P}^{\rm tor}(\uf_\kappa\times \ug_\lambda)\, \dot{=} \, \log_{\rm BK}({\rm GHC}_{f\times g})\,\qquad \wt(f)>\wt(g)\,.$$ 
This formula expresses the value of the BDP $p$-adic $L$-function $L_p^{\rm BDP}(\ug\times\uf)_{\vert_{g\times f}}$ at a pair $(g,f)$ with $\wt(g)<\wt(f)$ (that lies outside its range of interpolation), which can be thought of as the $p$-adic limit of toric periods $\mathcal{P}^{\rm tor}(\uf_\kappa\times \ug_\lambda)$ in the $\ug$-dominant range, in terms of the Bloch--Kato logarithm of the generalised Heegner cycle ${\rm GHC}_{f\times g}$ of Bertolini--Darmon--Prasanna. This is the first instance of what we call (following S. Lai) \emph{wall-crossing principle}, and it is summarised in Figure~\ref{Figure_2025_08_18_1114} below.

\begin{center}
          \resizebox{0.4\linewidth}{!}{
 \tikzset{every picture/.style={line width=0.75pt}} 
 \begin{tikzpicture}[x=0.75pt,y=0.75pt,yscale=-1,xscale=1]

 \draw    (100,230) -- (100,42) ;
 \draw [shift={(100,40)}, rotate = 450] [color={rgb, 255:red, 0; green, 0; blue, 0 }  ][line width=0.75]    (10.93,-3.29) .. controls (6.95,-1.4) and (3.31,-0.3) .. (0,0) .. controls (3.31,0.3) and (6.95,1.4) .. (10.93,3.29)   ;
 \draw    (100,230) -- (420,230) ;
 \draw [shift={(425,230)}, rotate = 180] [color={rgb, 255:red, 0; green, 0; blue, 0 }  ][line width=0.75]    (10.93,-3.29) .. controls (6.95,-1.4) and (3.31,-0.3) .. (0,0) .. controls (3.31,0.3) and (6.95,1.4) .. (10.93,3.29)   ;
   \fill[gray!20] (100,230) -- (280,50)  -- (100,48) -- cycle;
       \fill[lightgray!20]  (100,230)  -- (280,50)  --  (400,230)  -- cycle;
  \draw    (100,230) -- (280,50) ;

 \draw (72,32.4) node [anchor=north west][inner sep=-5pt]    {${\rm wt}(f)$};
 \draw (400,240.4) node [anchor=north west][inner sep=0.75pt]   {${\rm wt}(g)$};
\draw (190,78.4) node [anchor=north west][inner sep=0.75pt]    {$\varepsilon=-1$ };

 \draw (290,138.4) node [anchor=north west][inner sep=0.75pt]    {$\varepsilon=+1$ };
  \draw (221,158.4) node [anchor=north west][inner sep=0.75pt]    {$ L_p^{\rm BDP}$  };
  \draw (221,208.4) node [anchor=north west][inner sep=0.75pt]    {{${\rm wt}(g) > {\rm wt}(f)$}};
   \draw  (120,57.4) node [anchor=north west][inner sep=0.75pt]    {${\rm wt}(f)>{\rm wt}(g)$};

\draw plot [domain=0:1, samples=50, variable=\t] 
    ({187.88 - \t*(187.88 - 152.12)},  
     {167.88 - \t*(167.88 - 132.12) + 7.5*sin(22.5*\t r)});  
\draw [shift={(146.12, 126.12)}, rotate=45] 
    [fill={rgb, 255:red, 0; green, 0; blue, 0}] 
    [line width=0.08] 
    [draw opacity=0] 
    (11.72,-6.15) -- (0,0) -- (11.72,6.15) -- (8.12,0) -- cycle;
 \end{tikzpicture}   
 }
      \captionsetup{width=0.75\linewidth, font=scriptsize}
      
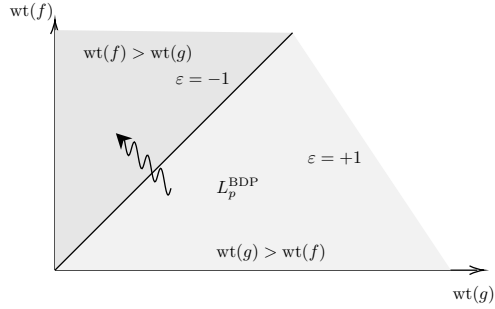
\captionof{figure}{Wall-crossing for $\GL2\times \GL2$}
       \label{Figure_2025_08_18_1114}
    \end{center}

A second example of wall-crossing has been studied in \cite{DR3, BSV} in the context of triple products, which is depicted in Figure~\ref{figure_2025_08_18_1113}: We let $\uh$ denote a third primitive Hida family (and we no longer assume that $\ug$ has CM). The local epsilon factors at the central critical point $\varepsilon_v(f\times g\times h)$ at a non-archimedean prime $v$ is constant as the classical specialisations $f\times g\times h$ of $\uf\times\ug\times\uh$ vary, and let us assume that they are equal to $+1$ for all such $v$. It turns out that, when the triple $(f,g,h)$ is balanced (namely, we have $\wt(x)+\wt(y)>\wt(z)$ for all permutations $(x,y,z)$ of $(f,g,h)$), the archimedean sign $\varepsilon_\infty(f\times g\times h)$ equals $-1$, and therefore also the global epsilon factor: $\varepsilon^{\rm bal}(f\times g\times h)=-1$. As a result, one expects cycles to explain the leading term of the associated $L$-series. These are the diagonal cycles of Gross--Kudla--Schoen, which Darmon--Rotger in \cite{DR3} and Bertolini--Seveso--Venerucci in \cite{BSV} interpolate as $f\times g\times h$ vary among the balanced classical specialisations of $\uf\times\ug\times\uh$. On the other hand, we have 
$$\varepsilon_\infty(f\times g\times h)=+1=\varepsilon^{(\ug)}(f\times g\times h)\,, \quad \hbox{ when } \wt(g)\geq \wt(f)+\wt(h)$$
in the $\ug$-dominant scenario. In this case, Hsieh in \cite{Hsieh2018} has constructed a $p$-adic $L$-function $L_p^{(\ug)}(\uf\times\ug\times\uh)$ whose square interpolates the central critical $L$-values $L(f\times g \times h, c)$ (suitably normalised by a period). This construction relies on an explicit Gan--Gross--Prasad (GGP) formula\footnote{We note that this is a generalisation of the work of Waldspurger and its extension by Bertolini--Darmon--Prasanna to the setting of the present paragraph.}, expressing the indicated $L$-value as a global trilinear period, and is one of the key ingredients of the reciprocity law of \cite{DR3, BSV}. This reciprocity law roughly reads
$$L_p^{(\ug)}(\uf\times\ug\times\uh)_{\vert_{f\times g\times h}}\,=\,\lim_{\substack{(\uf_\kappa,\ug_\lambda,\uh_\mu)\to (f,g,h)\\
\wt(\lambda)\geq \wt(\kappa)+\wt(\mu)}} \mathcal{P}_\Delta(\uf_\kappa, \ug_\lambda,\uh_\mu)\, \dot{=} \, \log_{\rm BK}({\rm \Delta}_{f\times g \times h})\,,\qquad (f,g,h) \hbox{ is balanced}\,.$$
This \emph{$p$-adic GGP limit formula} expresses the value of Hsieh's $p$-adic $L$-function $L_p^{(\ug)}(\uf\times\ug\times\uh)_{\vert_{f\times g\times h}}$, which can be thought of as the $p$-adic limit of trilinear periods $\mathcal{P}_\Delta(\uf_\kappa, \ug_\lambda,\uh_\mu)$ in the $\ug$-dominant range, at a balanced triple $(f,g,h)$ (that lies outside its range of interpolation) in terms of the Bloch--Kato logarithm of the diagonal cycle ${\rm \Delta}_{f\times g \times h}$.
\begin{center}
 \resizebox{0.45\linewidth}{!}{
 \tikzset{every picture/.style={line width=0.75pt}} 
 \begin{tikzpicture}[x=0.75pt,y=0.75pt,yscale=-1,xscale=1]

 \draw    (100,230) -- (100,42) ;
 \draw [shift={(100,40)}, rotate = 450] [color={rgb, 255:red, 0; green, 0; blue, 0 }  ][line width=0.75]    (10.93,-3.29) .. controls (6.95,-1.4) and (3.31,-0.3) .. (0,0) .. controls (3.31,0.3) and (6.95,1.4) .. (10.93,3.29)   ;
 \draw    (100,230) -- (420,230) ;
 \draw [shift={(425,230)}, rotate = 180] [color={rgb, 255:red, 0; green, 0; blue, 0 }  ][line width=0.75]    (10.93,-3.29) .. controls (6.95,-1.4) and (3.31,-0.3) .. (0,0) .. controls (3.31,0.3) and (6.95,1.4) .. (10.93,3.29)   ;
 \draw    (100,150) -- (180,230) ;
   \fill[gray!20] (100,150) -- (180,230)  -- (290,120) -- (210,40) -- cycle;
       \fill[lightgray!20]  (180,230)  -- (290,120)  --  (400,230)  -- cycle;
 \draw    (100,150) -- (210,40) ;
 \draw    (290,120) -- (180,230) ;
\draw plot [domain=0:1, samples=50, variable=\t] 
    ({227.88 - \t*(227.88 - 192.12)},  
     {207.88 - \t*(207.88 - 172.12) + 7.5*sin(22.5*\t r)});  

\draw [shift={(186.12, 166.12)}, rotate=45] 
    [fill={rgb, 255:red, 0; green, 0; blue, 0}] 
    [line width=0.08] 
    [draw opacity=0] 
    (11.72,-6.15) -- (0,0) -- (11.72,6.15) -- (8.12,0) -- cycle;
 \draw (72,32.4) node [anchor=north west][inner sep=-5pt]    {${\rm wt}(h)$};
 \draw (400,240.4) node [anchor=north west][inner sep=0.75pt]   {${\rm wt}(g)$};
 \draw  (165,142.4) node [anchor=north west][inner sep=0.75pt]    {${\rm Balanced}$};
 \draw (115,62.4) node [anchor=north west][inner sep=0.75pt]    {$L_p^{(\uh)}$};
 \draw (106,178.4) node [anchor=north west][inner sep=0.75pt]    {${L}^{(\underline{\mathbf f})}_{p}$};
  \draw (106,208.4) node [anchor=north west][inner sep=0.75pt]    {$\varepsilon=+1$};
 \draw (236,202.4) node [anchor=north west][inner sep=0.75pt]    {${L}^{(\underline{\mathbf g})}_{p}$};
  \draw (200,78.4) node [anchor=north west][inner sep=0.75pt]    {$\varepsilon=-1$ };

 \draw (280,158.4) node [anchor=north west][inner sep=0.75pt]    {$\varepsilon=+1$ };
  \draw (281,218.4) node [anchor=north west][inner sep=0.75pt]    {\tiny{${\rm wt}(g) \geq {\rm wt}(f)+{\rm wt}(h)$}};
 \end{tikzpicture}
 }
  \captionsetup{width=\linewidth, font=scriptsize}
      
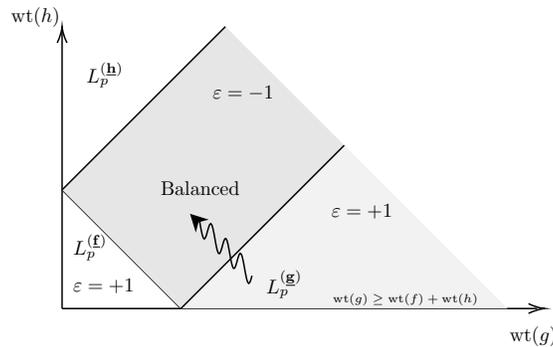
\captionof{figure}{Wall-crossing from $\ug$-dominant region to balanced region}
       \label{figure_2025_08_18_1113}
\end{center}

\subsection{Applications to Artin formalism} 
\label{subsec_2025_08_18_1530}
The key insight of this paper, as well as its precursors \cite{BS_triple_factor, bcpv, BC}, is that the Artin formalism of $p$-adic $L$-functions (that come about interpolating GGP periods) at points away from their range of interpolation is governed by the wall-crossing principle. In op. cit., the authors consider the case when $\uh=\ug^c$ is the family dual to $\ug$, so that we have a decomposition 
\begin{equation}
    \label{eqn_2025_09_23_0641}
    \uf \times \ug \times \ug^c \,\dot{=}\, \uf \times \ad^0(\ug) \boxplus \uf
\end{equation}
of families of automorphic representations, and formulate a conjecture that factors the $\ug$-dominant $p$-adic $L$-function $L_p^{(\ug)}(\uf\times\ug\times\ug^c)$ to reflect the decomposition \eqref{eqn_2025_09_23_0641}. We remark that such a factorisation is \emph{not} a soft consequence of Artin formalism, as the range of interpolation of $L_p^{(\ug)}(\uf\times\ug\times\ug^c)$ does not intersect the locus range where the decomposition \eqref{eqn_2025_09_23_0641} is valid. In \cite{bcpv}, the factorisation conjecture is reduced to the proof of arithmetic GGP conjectures\footnote{These are extensions of Gross--Kudla conjectures to allow non-semistable cases.} in the balanced range. As a result, one may treat this factorisation problem as a recipe to encode the interplay between $p$-adic GGP formulae (of Bertolini--Seveso--Venerucci and Darmon--Rotger) and arithmetic GGP formulae (cf. \cite{YZZ, HangXue}) in this degenerate scenario (where $\uh=\ug^c$).

We remark that the factorisation conjecture of \cite{BS_triple_factor} has been proved unconditionally in \cite{BC} in the special case when $\uh=\ug^c$ has CM by an imaginary quadratic field $K$, whose non-trivial character is denoted by $\epsilon_K$. This result relies on the factorisation of the BDP $p$-adic $L$-function $L_p^{\rm BDP}(\ug_K\times\uf)_{\vert_{g_{\rm Eis}}}$ when the canonical Hida family $\ug_K$ (as in \cite{BDV}, \S4.2) is specialised to $g_{\rm Eis}$ (a weight-one Eisenstein series), which reflects the decomposition $g_{\rm Eis}\times \uf=\uf \boxplus \uf\otimes\epsilon_K$. The proof of the factorisation of $L_p^{\rm BDP}(\ug_K\times\uf)_{\vert_{g_{\rm Eis}}}$ is deduced from the BDP formula in this setting, paralleling the approach outlined in the previous paragraph.

\subsection{This paper} Our goal in the present work is to develop a $p$-adic Artin formalism in the context of $p$-adic families of automorphic forms on $\GSp_4 \times \GL_2 \times \GL_2$, extending the work of the first and third named authors outlined in \S\ref{subsec_2025_08_18_1530}. 

To make this precise, let us fix an automorphic representation  $\Pi$ of $\GSp_4$, and assume that the $\GL_2 \times \GL_2$-factor is the anti-symmetric product $\sigma\times \sigma^c$. In this scenario, the degree--16 $L$-series 
$$L(s,\Pi\times \sigma\times  \sigma^c)=L(s,\Pi\times {\rm ad}^0\,\sigma)\,L(s,\Pi)$$ 
factors into a product of a degree-12 and degree-4 $L$-series.
As $\Pi$ and $\sigma$ vary in $p$-adic families, one expects (``$p$-adic Artin formalism'') a corresponding factorisation of the degree--16 $p$-adic $L$-functions. Our primary focus will be those scenarios where the sought-after factorisation is \emph{not} a direct consequence of the interpolation formulae characterising the $p$-adic $L$-functions. 

The first objective of this paper is to formulate a precise factorisation conjecture for $p$-adic $L$-functions attached to families of forms on $\GSp_4\times \GL_2\times\GL_2$, and the second is to prove its variant for algebraic $p$-adic $L$-functions. As in \cite{BS_triple_factor}, our guide is the wall-crossing principle that ties together the ($p$-adic limits of) explicit GGP formulae and arithmetic GGP formulae (both of which appear to be out of reach for the time being), together with the Equivariant Tamagawa Number Conjecture (ETNC).

Before we move ahead to describe our work in detail, we remark a fundamental difference between the consideration here concerning the $p$-adic $L$-function whose interpolation range is ``Region (a)'' (in the terminology of \cite{LZ20}), and that in \cite{BS_triple_factor} outlined above: One needs to wall-cross twice. This point is detailed in \S\ref{subsubsec_2025_08_18_1601} below.

\subsection{The setting of this paper}
We fix a family of cuspidal automorphic representations $\upi$ of $\GSp_4$, together with a family of cuspidal automorphic representations $\underline{\sigma}$ of $\GL_2$. We denote by $\underline{\sigma}^c$ the conjugate Hida family, namely the twist of $\underline{\sigma}$ by the inverse of its central character. The weights of a classical specialisation of $\upi$ will be denoted by a pair $(k_1,k_2)$, where $k_1 \geq k_2 \geq 3$, and we write $c_1$ and $c_2$ for the weights of classical specialisations of $\underline{\sigma}$ and $\underline{\sigma}^c$, respectively. In \cite[\S2.3]{LZ20}, the authors explicate the ``GGP regions'', determined in terms of the interlacing relations of the quadruple of weights $(k_1,k_2,c_1,c_2)$; cf. \S\ref{subsubsec_2025_08_19_0801}. 

\subsubsection{} One expects a $p$-adic $L$-function in each of these regions, characterised by an interpolation property involving the square roots of the central critical values. When the global epsilon factor equals $-1$ in a region $(\diamond)$ (as is the case when $\diamond\in \{b, e\}$), then the said $L$-values vanish, and as a result, so does the corresponding $p$-adic $L$-function. We call these regions ``geometric'', as one expects the existence of cycles to explain this systematic order of vanishing (\`a la arithmetic GGP conjectures). In \S\ref{subsubsec_2025_08_26} we review the conjectural descriptions of these $p$-adic $L$-functions, which we denote by $L_p^{(\diamond)}(\underline{\Pi} \times \underline{\sigma} \times \underline{\sigma}^c)$, where $\diamond \in \{a,c,d,f\}$ (those GGP regions where the global sign equals $+1$). 

\subsubsection{}  The main objective of this article is to study a $p$-adic Artin formalism for these degree--16 $p$-adic $L$-functions (cf. our Factorisation Conjecture~\ref{conj_2025_08_21_1552} and Conjecture~\ref{conj_2025_08_22_0031}), with an emphasis on regions (a) and (d) where the intersection of the range of interpolation of $L_p^{(\diamond)}(\underline{\Pi} \times \underline{\sigma} \times \underline{\sigma}^c)$ does not intersect the locus where Artin formalism is relevant (i.e., the locus determined by the diagonal of the family $\sigma\times \sigma^c$, over which have a decomposition $\upi\times \ad\underline\sigma=\upi\times \ad^0\underline\sigma \boxplus \upi$ of families of automorphic representations). As a shorthand for the restriction of $L_p^{(\diamond)}(\underline{\Pi} \times \underline{\sigma} \times \underline{\sigma}^c)$ to this locus, we sometimes use the notation $L_p(16 \diamond)$ (where $16$ is the degree of the $L$-functions whose central values are being interpolated).

\subsubsection{}  Our Factorisation Conjecture~\ref{conj_2025_08_21_1552} for $L_p(16d)$ involves another (also conjectural) degree--12 $p$-adic $L$-function $L_p(12e)$, reflecting the decomposition of the underlying degree-16 motives into a direct sum of degree-12 and degree-4 motives alluded to above.

\subsubsection{} As notation for later use, we write $\mathcal R_{\underline{\Pi}}$ the Hecke algebra of $\upi$ (see \S\ref{subsubsec_2025_08_26_1220} for its precise definition), and we define $\mathcal R_{\underline{\sigma}}$ in a similar way (cf. \S\ref{subsec_2025_08_25_1050}). We set
\[
  \cR_4 := \cR_{\upi} \,\widehat \otimes_{\mathbb Z_p}\, \cR_{\underline{\sigma}}  \,\widehat \otimes_{\mathbb Z_p}\, \cR_{\underline{\sigma}},
  \qquad
  \cR_3 := \cR_{\upi}  \,\widehat \otimes_{\mathbb Z_p}\, \cR_{\underline{\sigma}}.
\]

\subsection{An overview} We briefly outline the main results of the paper.

\subsubsection{Conjectural factorisation formulae}
As we have remarked above, the main purpose of this work is to puzzle out the form of Artin formalism for the $p$-adic $L$-functions $L_p(16d)$ and $L_p(16a)$: these are formulated as Conjecture~\ref{conj_2025_08_21_1552} and Conjecture~\ref{conj_2025_08_22_0031} below. We refer the reader to \S\ref{remark_intro_BDP_principle}, where we discuss the motivation behind these conjectures, from the perspective offered by the (double) \emph{wall-crossing} principle, which rests upon conjectural \emph{$p$-adic GGP formulae} in this context. 

The factorisation problem for $L_p(16d)$ seems to parallel the analogous problem studied in \cite{BS_triple_factor} in the context of forms on $\GL_2\times\GL_2\times\GL_2$. As a matter of fact, in \S\ref{subsec_2025_08_06_1504_bis} (see especially Theorem~\ref{thm_2025_08_21_1108}), we explain that a degenerate case of  Conjecture~\ref{conj_2025_08_21_1552} is equivalent to the factorisation problem considered in op. cit.

On the other hand, Factorisation Conjecture~\ref{conj_2025_08_22_0031} concerning $L_p(16a)$ seems far more subtle, requiring a novel input. We discuss the underlying key principles in~\S\ref{subsubsec_2025_08_18_1601}.

Note that even the $p$-adic $L$-functions $L_p(16d)$ and $L_p(16a)$ are currently conjectural (as well as some of the factors that appear in their conjectural factorisation formulae) in the level generality they are introduced in our paper. Our evidence in support of our Factorisation Conjecture~\ref{conj_2025_08_21_1552} and Conjecture~\ref{conj_2025_08_22_0031} are in two forms: Firstly, we formulate and prove their algebraic counterparts; cf. \S\ref{subsubsec_2025_08_25_1156} for an outline of our results in this vein.  Secondly, we obtain unconditional results in endoscopic cases; cf. \S\ref{subsubsec_2025_08_25_1157} for an overview of our work in this direction.

\subsubsection{Algebraic counterpart}
\label{subsubsec_2025_08_25_1156}
Sections~\ref{sec_Selmer}--\ref{sec_factorisation_main} are dedicated to prove algebraic counterparts of our factorisation conjectures. The algebraic counterparts of the conjectural $p$-adic $L$-functions $L_p(16\diamond)$ and $L_p(12e)$ are characteristic ideals of appropriately defined Selmer groups (which arise as the cohomology of a Selmer complex in degree $2$; cf. \S\ref{subsec_2025_08_25_1436}), which one calls ``algebraic $p$-adic $L$-functions''. In this introduction, we shall denote them by $L_p^{\rm alg}(16\diamond)$ and $L_p^{\rm alg}(12e)$.

However, as the propagations of the $p$-refinements of the degree-16 family of motives to degree-12 and degree-4 summands no longer satisfy the Panchiskin condition, the factorisation of the algebraic $p$-adic $L$-functions does not simply involve (as in the case of \cite{palvannan18}) algebraic $p$-adic $L$-functions associated to the degree-12 and degree-4 summands, but rather modules of leading terms that we introduce and systematically study in \S\ref{sec_Koly_Sys}.

Our Theorem~\ref{thm_5_1_2025_02_14_1424} indeed factors $L_p^{\rm alg}(16d)$ as a product of $L_p^{\rm alg}(12e)$ and the trivialisation of the module of leading terms associated to the Galois representation $T_\upi^\dagger$ of rank $4$. The proof of this result closely follows the arguments of \cite{BS_triple_factor}; we give a second proof of this theorem in Remark~\ref{remark_2025_03_10_1606}.

The factorisation of $L_p^{\rm alg}(16a)$ is far more subtle, as based on our double wall-crossing principle (cf. \S\ref{subsubsec_2025_08_18_1601}), it is expected to be a $p$-adic avatar of second-order central derivatives of $L$-functions. We study this problem in \S\ref{subsec_5_2_2025_02_11}. Our main result, Theorem~\ref{thm_main_double_wall_crossing}, in this subsection factors $L_p^{\rm alg}(16a)$ into a product of trivialisations of $2$ modules of leading terms (associated with each family of rank-$12$ and rank-$4$ sub-representations).

Our algebraic factorisation results in \S\ref{sec_factorisation_main} are supplemented in \S\ref{subsec_2025_08_25_1428}, where we treat the case of Yoshida lifts in the region (a). Our main result in this subsection is Corollary~\ref{cor_2025_08_14_1551}. We also refer the reader to Corollaries~\ref{cor_2025_08_06_1325}--\ref{cor_2025_08_07_0919} and Remark~\ref{rem:yoshida_log}, where we explain why the modification in \S\ref{subsec_2025_08_25_1428} is required (in view of the properties of Eichler--Shimura isomorphisms in this context).

\subsubsection{Endoscopic scenarios} 
\label{subsubsec_2025_08_25_1157}
When $\underline{\Pi}$ is a family of Yoshida lifts, one may give an ad hoc (but unconditional) definition of $L_p(16a)$ and $L_p(16d)$ in terms of Hsieh's triple product $p$-adic $L$-functions (cf. \S\ref{subsubsec_2025_08_25_1205} and \S\ref{subsec_2025_08_06_1504_bis}), and our Factorisation Conjecture~\ref{conj_2025_08_21_1552} for $L_p(16d)$ can be reduced to one considered in \cite{BS_triple_factor,bcpv} (cf. Theorem \ref{thm_2025_08_21_1108}).

The Factorisation Conjecture~\ref{conj_2025_08_22_0031} concerning $L_p(16a)$ can be formulated as Conjecture~\ref{conj_2025_08_06_1401}. As we note at the start of \S\ref{sec_2025_02_06_0708}, a reformulation is indeed required due to the interpolative properties of Eichler--Shimura isomorphisms (cf. \S\ref{subsec_2025_02_04_1413}). The proof of Conjecture~\ref{conj_2025_08_06_1401} is then reduced to a pair factorisation problems for families on $\GL_2\times \GL_2\times\GL_2$, one of which is the one considered in \cite{BS_triple_factor,bcpv}, and the other (cf. Equation~\eqref{eqn_2025_08_25_1222} below), that supplements these works, is work in progress by the first and third named authors with A. Cauchi.

Finally, in \S\ref{subsec_2025_08_25_1211}, we consider the case when $\upi$ is a (one-parameter) family of Saito--Kurokawa lifts of a Hida family $\uf$ of ordinary modular forms, and $T_\upi^\dagger=T_{\underline{\mathbf f}}^\dagger \oplus \mathbb Z_p \oplus \mathbb Z_p(1)$. In this case, the factorisation problem once again reduces to a combination of results established in \cite{BS_triple_factor} and \cite{Das}, and this is summarised as Proposition \ref{prop_2025_08_14_1443}.

\subsection{Wall-crossing principle}
\label{remark_intro_BDP_principle}
We briefly discuss the (BDP) ``wall-crossing'' principle\footnote{This nomenclature is due to Shilin Lai.}, which is an extension of the BDP principle in \cite[\S2.2.5]{BS_triple_factor}, that governs our factorisation conjectures. This will also highlight the differences with the scenario considered in \cite{BS_triple_factor} in the scenario ``(a) to (c)'' (cf. \S\ref{subsec_5_2_2025_02_11}) which requires ``double wall-crossing'': first from (a) to (b), then from (b) to (c). 

To facilitate this rather philosophical discussion, for $D\in \{12,16\}$, and $?\in \{a,b,c,d,e\}$, let us denote by $L_p(D?)$ the restriction of the (conjectural) $p$-adic $L$-function interpolating the central critical values of the degree-D $L$-series $L(s, \Pi\times \sigma \times \sigma^c)$ when the weights belong to (?), to $\cX:= {\rm im}({\rm Spec}\,\cR_3 \xrightarrow{\iota_{3,4}} {\rm Spec}\,\cR_4)$. Let us denote the classical specialisations in $\cX$ by $\cX^{\rm cl}$ (that correspond to classical forms on ${\rm GSp}_4\times \GL_2\times\GL_2$), and by $\cX^{\rm cl}_?$ those classical specialisations with weights in the region (?). 

Throughout this discussion, we retain our running assumption that the global root number $\varepsilon(\upi)$ of all members of the Hida family $\upi$ at their central critical points is $-1$.

\subsubsection{Crossing from (d) to (e)} 
\label{subsubsec_2025_02_11_1659}
Notice that the interpolation range of the conjectural $p$-adic $L$-function $L_p(16d)$ is $\cX^{\rm cl}_d$, which is empty. On the other hand, $\cX^{\rm cl}_e$ is dense in $\cX$. Recall also that, by assumption, the global root number of $\upi_\kappa\times {\rm ad}^0 \underline{\sigma}_\lambda$ at its central critical point equals $-1$ for all $(\kappa,\lambda)\in \cX_e^{\rm cl}$. 

Furthermore,  whenever $(\kappa,\lambda)\in \cX^{\rm cl}_e$, $s=\frac{1}{2}$ is critical\footnote{See \S\ref{subsec_2025_09_24_1234} for a further discussion on this point.} (in the sense of Deligne) for both degree-12 $L$-series $L(s,\upi_\kappa\times {\rm ad}^0 \underline{\sigma}_\lambda)$ and the degree--4 $L$-series $L(s,\upi_\kappa)$. Note that we have 
$$\varepsilon(\upi_\kappa)=-1\,, \quad \varepsilon(\upi_\kappa\times {\rm ad}^0 \underline{\sigma}_\lambda)=+1,\qquad\qquad   \forall\,(\kappa,\lambda) \in \cX^{\rm cl}_e\,,$$
and the Artin formalism for such $(\kappa,\lambda)$ yields the factorisation
\begin{equation}
\label{eqn_2025_02_11_1627}
    L'(\tfrac{1}{2},\upi_\kappa\times {\rm ad}\, \underline{\sigma}_\lambda)\,=\,L(\tfrac{1}{2},\upi_\kappa\times {\rm ad}^0 \underline{\sigma}_\lambda)\,\cdot\, L'(\tfrac{1}{2},\upi_\kappa)\,.
\end{equation}
Inspired by the fundamental results in \cite{BDP}, we adopt the guiding principle that 
\begin{itemize}
    \item[\mylabel{item_BDP1}{($\mathbf{BDP}_{e}^{d}$)}] the $p$-adic $L$-function $L_p(16d)$ should be thought of as a $p$-adic avatar of the family of derivatives 
    $$\left\{L'(\tfrac{1}{2},\upi_\kappa\times {\rm ad}\, \underline{\sigma}_\lambda)\,:(\kappa,\lambda)\in \cX^{\rm cl}_e\right\}\,.$$ 
\end{itemize}
This, combined with \eqref{eqn_2025_02_11_1627}, in turn suggests that 
$$L_p (16d)_{\vert_{(e)}}\, \dot{=} \, L_p(12e)\,\cdot ``\hbox{a $p$-adic avatar of }\left\{L'(\tfrac{1}{2},\upi_\kappa)\,:\kappa\in {\rm Spec}(\cR_\upi)^{\rm cl}\right\}"\,.$$
This is the form of our factorisation conjecture for the $p$-adic $L$-function $L_p(16d)$, as well as its algebraic variant concerning modules of leading terms. We invite the reader to compare this discussion to \cite[\S2.2.5]{BS_triple_factor}, where the similarity of the problem at hand to that considered in op. cit. will be evident.

\begin{remark}
    The wall-crossing principle can be recast as follows. The $p$-adic limits of the explicit GGP formulae on the region (d) approximating the points in the region (e) (which give rise to $L_p (16d)_{\vert_{(e)}}$) are $p$-adic avatars of the arithmetic GGP formulae in the region (e). This is an extension of the principles underlying \cite{BDP}: Indeed, in the setting of op. cit., the BDP $p$-adic $L$-function is constructed building on the Waldspurger formula, and the arithmetic GGP formulae boils down to Gross--Zagier formulae for generalised Heegner cycles. \hfill$\blacksquare$
\end{remark}

\subsubsection{Crossing from (a) to (c) through (b)}
\label{subsubsec_2025_08_18_1601}
We now discuss the case that concerns the factorisation of the $p$-adic $L$-function $L_p(16a)$, which involves a new phenomenon (that we call ``double wall-crossing'').

As in the previous case, $\cX^{\rm cl}_a$ is empty, but so is $\cX^{\rm cl}_b$.  This is the key difference with \S\ref{subsubsec_2025_02_11_1659}, and that is why \ref{item_BDP1} will not apply. That also means, the search for an arithmetic meaning of $L_p(16a)$ leads us to the region (c) that region (b) ``neighbours'' (cf. \cite{LZ20}, Figure 2). Note that $\cX^{\rm cl}_c$ is dense in $\cX$, and by assumption, the global root number of $\upi_\kappa\times {\rm ad}^0 \underline{\sigma}_\lambda$ at its central critical point $c(\kappa)$ equals $+1$ for all $(\kappa,\lambda)\in \cX_c^{\rm cl}$.

The value $s=\frac{1}{2}$ is critical (in the sense of Deligne) for both $L$-series $L(s,\upi_\kappa\times {\rm ad}^0 \underline{\sigma}_\lambda)$ and $L(s,\upi_\kappa)$ whenever $(\kappa,\lambda)\in \cX^{\rm cl}_c$. Furthermore, we have 
$$\varepsilon(\upi_\kappa)=-1=\varepsilon(\upi_\kappa\times {\rm ad}^0 \underline{\sigma}_\lambda),\qquad\qquad  \forall\, (\kappa,\lambda) \in \cX^{\rm cl}_c\,,$$
and the Artin formalism for such $(\kappa,\lambda)$ yields the factorisation
\begin{equation}
\label{eqn_2025_02_11_1710}
    L''(\tfrac{1}{2},\upi_\kappa\times {\rm ad}\, \underline{\sigma}_\lambda)\,=\,L'(\tfrac{1}{2},\upi_\kappa\times {\rm ad}^0 \underline{\sigma}_\lambda)\,\cdot\, L'(\tfrac{1}{2},\upi_\kappa)\,.
\end{equation}
of second-order derivatives. The key insight that led us to our factorisation conjecture is the following:
\begin{itemize}
    \item[\mylabel{item_BDP2}{($\mathbf{BDP}_{c}^{a}$)}] the $p$-adic $L$-function $L_p(16a)$ should be thought of as a $p$-adic avatar of the family of derivatives 
    $$\left\{L''(\tfrac{1}{2},\upi_\kappa\times {\rm ad}\, \underline{\sigma}_\lambda)\,:(\kappa,\lambda)\in \cX^{\rm cl}_c\right\}\,,$$
\end{itemize}
which, combined with \eqref{eqn_2025_02_11_1710}, suggests\footnote{This insight can be made explicit in the context of endoscopic families; cf. \S\ref{sec_2025_02_06_0708}.} 
\begin{align*}
L_p (16a)\, &\dot{=} \, ``\hbox{a $p$-adic avatar of }\left\{L'(\tfrac{1}{2},\upi_\kappa\times {\rm ad}^0 \underline{\sigma}_\lambda)\,:(\kappa,\lambda)\in \cX^{\rm cl}_c\right\}"\\
&\qquad \times ``\hbox{a $p$-adic avatar of }\left\{L'(\tfrac{1}{2},\upi_\kappa)\,:\,\kappa\in {\rm Spec}(\cR_\upi)^{\rm cl}\right\}"\,.
\end{align*}
This is the form of our factorisation conjecture for the $p$-adic $L$-function $L_p(16a)$, as well as its algebraic variant concerning modules of leading terms.

\begin{remark}
    The double-wall-crossing principle can be thought of as follows. The $p$-adic limits of the explicit GGP formulae on the region (a) approximating the points in the region (c) (which give rise to $L_p (16a)_{\vert_{(c)}}$) are $p$-adic avatars of degree--2 arithmetic GGP formulae\footnote{This presently remains a fantasy, even in a conjectural form. However, in the degenerate picture we have placed ourselves in, one may cast this as a pair of arithmetic GGP formulae. Our Theorem~\ref{prop:delta_a-formula2} is the algebraic counterpart of a BDP (``$p$-adic GGP'') formula of degree $2$ for $\upi\times \underline{\sigma}\times\underline{\sigma}^c$.} in the region (c). 
    \hfill$\blacksquare$
\end{remark}

\subsection{Acknowledgements}
KB’s research in this publication was conducted with the financial support of Taighde \'{E}ireann -- Research Ireland under Grant number IRCLA/2023/849 (HighCritical). OR was supported by PID2023-148398NA-I00, funded by MCIU/AEI/10.13039/501100011033/FEDER, UE. RS was supported by JSPS KAKENHI Grant Number 24K16886.

\section{Hida families and $p$-adic $L$-functions}

In this section, we recall the notations and properties we will be using around Hida families, both in the $\GL_2$ and in the Siegel cases, and also around $p$-adic $L$-functions, where most of the results are still conjectural. We also fix the general assumptions that will be imposed along the work.

\subsection{Hida families}
\label{subsec_2025_08_25_1050}
Let $p$ be an odd prime and let $\mathcal O$ be the ring of integers of a finite extension $E$ of $\mathbb Q_p$. Let us put $[ \,\cdot\, ] \colon \mathbb Z_p^{\times} \hookrightarrow \Lambda(\mathbb Z_p^{\times})^{\times}$ to denote the natural injection. The universal weight character $\chi$ is defined as the composite map $$  \chi \colon G_{\mathbb Q} \xrightarrow{\chi_{\cyc}} \mathbb Z_p^{\times} \hookrightarrow \Lambda(\mathbb Z_p^{\times})^{\times},  $$  where $\chi_{\cyc}$ is the $p$-adic cyclotomic character. We take our normalisations in such a way that the Hodge--Tate weight of the cyclotomic character is $-1$.

\subsubsection{}A ring homomorphism $\nu \colon \Lambda_{\wt} := \Lambda(\mathbb Z_p^{\times}) \rightarrow \mathcal O  $ is called an arithmetic specialisation of weight $k \in \mathbb Z$ if the compositum $G_{\mathbb Q} \xrightarrow{\chi} \Lambda_{\wt}^{\times} \xrightarrow{\nu} \mathcal O  $  agrees with $\chi_{\cyc}^k$ on an open subgroup of $G_{\mathbb Q}$. We also regard integers as elements of the weight space via 
$$  \mathbb Z \rightarrow \Hom_{\mathcal O}(\Lambda_{\wt}, \mathcal O)\,, \qquad n \mapsto (\nu_n \colon [x] \mapsto x^n).  $$  
For an integer $k$, we define 
$\Lambda(\mathbb Z_p^{\times})^{(k)} \simeq \Lambda(1+p\mathbb Z_p)$ as the component determined by the weight $k$.

\subsubsection{} We let $\underline{\sigma} = \sum_{n=1}^{\infty} a_n(\underline{\sigma})q^n \in \mathcal R_{\underline{\sigma}}[[q]]$ denote the branch of the primitive Hida family of tame conductor $N_{\underline{\sigma}}$ and tame nebentype $\varepsilon_{\underline{\sigma}}$, which admits a crystalline specialisation $f_{\circ}$ of weight $k$. The universal weight character $\chi$ give rise to a character $$  \chi_{\underline{\sigma}} \colon G_{\mathbb Q} \xrightarrow{\chi} \Lambda_{\wt}^{\times} \twoheadrightarrow \Lambda(\mathbb Z_p^{\times})^{(k), \times} \rightarrow \mathcal R_{\underline{\sigma}}^{\times}.  $$ 

\subsubsection{} By the fundamental work of Hida, there is a big Galois representation $$  \rho_{\underline{\sigma}} \colon G_{\mathbb Q, \Sigma} \rightarrow \GL_2(\Frac(\mathcal R_{\underline{\sigma}}))  $$  attached to $\underline{\sigma}$, where $\Sigma$ is a finite set of primes containing all those dividing $pN_{\underline{\sigma}} \infty$ and $\Frac(\mathcal R_{\underline{\sigma}})$ stands for the field of fractions of $R_{\underline{\sigma}}$. We denote by $T_{\underline{\sigma}} \subset \text{Frac}(\mathcal R_{\underline{\sigma}})^{\oplus 2}$ the Ohta lattice, where $T_{\underline{\sigma}}$ corresponds to $M_{\underline{\sigma}}^*$ in the notations of \cite{KLZ}, which realize the Galois representation $\rho_{\underline{\sigma}}$ in the \'etale cohomology groups of a tower of modular curves. 

\subsubsection{} We make the following assumptions.
\begin{ass}\label{ass-gl2}
The Galois representation $T_{\underline{\sigma}}$ satisfies the following conditions:
\begin{enumerate}
\item[(a)] The residual Galois representation of $T_{\underline{\sigma}}$ is irreducible.
\item[(b)] The residual Galois representation is $p$-distinguished (that is, its semisimplification is not scalar).
\end{enumerate}
\end{ass}
We remark that Condition (a) assures that any $G_{\mathbb Q}$-stable lattice in the field of fractions is homothetic to $T_{\underline{\sigma}}$, whereas Condition (b) supplies one with an integral $p$-ordinary filtration. We assume, unless explicitly stated otherwise, that these assumptions hold for all Hida families on $\GL_2$ that appear in this work. We also assume the analogous conditions for all modular forms (i.e. not only for the families).

\subsection{Preliminaries on Siegel modular forms}\label{subsec:Siegel}

We denote by $G$ the group scheme $\GSp_4$ (over $\mathbb Z$), defined with respect to the anti-diagonal matrix $J = \begin{smatrix}&&&1\\ &&1 \\ &-1 \\ -1 \end{smatrix}$; and we let $\nu$ be the multiplier map $G \to \GG_m$, where $\GG_m$ is the multiplicative group. We define $H = \GL_2 \times_{\GL_1} \GL_2$, which we embed into $G$ via 
$$  \iota \, : \, \Big[
   \begin{pmatrix}a&b\\c&d\end{pmatrix}, \begin{pmatrix}a'&b'\\c'&d'\end{pmatrix} \Big]
   \longmapsto \begin{pmatrix}a&&&b\\&a'&b'\\&c'&d'\\c&&&d\end{pmatrix}.
 $$ 
Let $B_G$ for the upper-triangular Borel subgroup of $G$, and $P_{\Si}$ and $P_{\Kl}$ for the standard Siegel and Klingen parabolics containing $B_G$, so that
$$  P_{\Si} =
   \begin{smatrix}\star & \star & \star & \star \\ \star & \star & \star & \star\\
   && \star & \star\\ && \star & \star \end{smatrix}, \quad
   P_{\Kl} =
   \begin{smatrix}\star & \star & \star & \star \\  & \star & \star & \star\\
   &\star & \star & \star\\ &&& \star \end{smatrix}.
 $$ 
We write $B_H = \iota^{-1}(B_G) = \iota^{-1}(P_{\Si})$ for the upper-triangular Borel of $H$. 

\subsubsection{}
Let $\Pi$ be a cuspidal automorphic representation of $G(\mathbb A_{\fin})$, with $\Pi_{\infty}$ of weights $(k_1,k_2)$, with $k_1 \geq k_2 \geq 3$. Let us put $ r_i := k_i - 3$, for $i=1,2$. Let $\{\alpha,\beta,\gamma,\delta\}$ be the Hecke parameters of $\Pi_p$ in the sense of \cite[\S5.1.5]{Pil20} or \cite[Theorem 10.1.3]{LSZ22}, where we borrow our conventions. More precisely, they are the reciprocal of the roots of a precise Hecke polynomial at $p$, $P_p(X)$, defined in Theorem 10.1.3(3) of loc.\,cit.
This means that we can normalize $(\alpha,\beta,\gamma,\delta)$ in such a way that all four are algebraic integers of complex absolute value $p^{(k_1+k_2-3)/2}$, and ordered such that $\alpha \delta = \beta \gamma = p^{(k_1+k_2-3)} \chi_{\Pi}(p)$ and $0 \leq v_p(\alpha) \leq \ldots \leq v_p(\delta)$. 
Here $\chi_\Pi$ denotes the Dirichlet character associated with the central character of $\Pi$.

\subsubsection{}
Let $\Iw_G(p)$ denote the Iwahori subgroup of $G$ associated with $B_G$. We consider the following operators in the Hecke algebra of level $\Iw_G(p)$, acting on the cohomology of any of the sheaves introduced above:

\begin{itemize}
\item The Siegel operator $\mathcal U_{\Si} = [\diag(p,p,1,1)]$, as well as its dual $\mathcal U_{\Si}' = [\diag(1,1,p,p)]$.
\item The Klingen operator $\mathcal U_{\Kl} = p^{-r_2} \cdot [\diag(p^2,p,p,1)]$, as well as its dual $\mathcal U_{\Kl}' = p^{-r_2} \cdot [\diag(1,p,p,p^2)]$.
\item The Borel operator $\mathcal U_B = \mathcal U_{\Si} \cdot \mathcal U_{\Kl}$, as well as its dual $\mathcal U_B' = \mathcal U_{\Si}' \cdot \mathcal U_{\Kl}'$.
\end{itemize}
A Hecke eigenform is called \emph{Siegel} (resp.~\emph{Klingen}) ordinary if the corresponding eigenvalue is a $p$-adic unit. It is called \emph{Borel} ordinary if it is both Siegel and Klingen ordinary.

Throughout this article, we assume that $\Pi_p$ is Borel-ordinary, which implies that it is both Klingen-ordinary and Siegel-ordinary. In this case, we may and will assume that $(\alpha,\beta,\gamma,\delta)$ have $p$-adic valuations $(0,\,r_2+1,\,r_1+2,\,r_1+r_2+3)$, respectively.

\subsection{Families of (Borel-ordinary) Siegel modular forms}\label{subsec:families}

We recall the definition of a Hida family $\underline{\Pi}$, following the notation of Loeffler--Zerbes \cite{LZ20-bk} and \cite[\S10.4]{LZ21-erl}. Note that in the former reference, the authors consider only one parameter (what they call a Siegel-type Hida family), allowing parallel variation of the weights. Moreover, since we are concerned here solely with Borel-ordinary families, we may and shall work in the setting of \cite{TU99}.  

\subsubsection{}
We begin by a discussion of the notion of a \emph{family of Siegel automorphic forms}, following the conventions of Tilouine--Urban \cite{TU99} and Loeffler--Zerbes \cite{LZ20-bk}, \cite{LZ21-erl}. We also refer the reader to \cite{AIP15} for a more detailed account of the geometry of Siegel eigenvarieties.

\subsubsection{} Let $K$ be a finite extension of $\mathbb Q_p$, and let $\mathcal O$ denote its ring of integers. Following \cite{TU99}, let us set $\mathcal R = \mathcal O[[T_1,T_2]]$. Let $\Omega = \Spf(\mathcal R)(\mathcal O)$ denote the weight space. For any pair $(a,b)$ with $a > b > 0$, define the arithmetic prime $\mathcal P_{a,b}$ of $\Lambda$ as the kernel of the homomorphism $\Lambda \to \mathcal O$ given by $T_1 \mapsto (1+p)^a - 1$ and $T_2 \mapsto (1+p)^b - 1$. Such a morphism is called a \emph{classical point} of $\Omega$.

\subsubsection{} 
\label{subsubsec_2025_08_26_1220}
Following \cite{boxerpilloni20} and, in particular, \cite[\S10.4]{LZ21-erl}, there exist graded coherent sheaves $H^k(\mathcal M_{\text{cusps},w_j}^{\bullet,-,\text{fs}})$ on $\Omega_{\Pi}$ for $0 \leq j,k \leq 3$, whose pushforward to $\Omega$ is the corresponding $H_{w_j,\mathrm{an}}^k(K^p, \nu_U, \text{cusp})^{(-,\text{fs})}$. Here we adopt the conventions of \cite[\S9.5]{LZ21-erl} for cuspidal, locally analytic, overconvergent cohomology; in particular, $w_j$ denotes the Kostant representative defined in \S2.1 of loc.\,cit., $K^p$ denotes the level structure, and $\nu_U$ is a suitable character of the torus. 

\subsubsection{} 
Given any finite flat extension $\widetilde {\mathcal R}$ of $\mathcal R$, let $\widetilde \Omega := \Spf(\tilde{\mathcal R})(\mathcal O)$. This space is equipped with its natural $p$-adic topology and with a natural projection ${\rm wt}\, \colon\, \widetilde \Omega \to \Omega$ (the weight map) induced by the inclusion of $\mathcal O$-algebras $\mathcal R \subset \widetilde {\mathcal R}$. A point $x \in \widetilde \Omega$ such that $\wt(x)$ is a classical point of $\Omega$ is called a \emph{classical point} of $\widetilde \Omega$, and the set of all such points is denoted by $\widetilde \Omega_{\cl}$. Recall that we denote by $\chi$ the cyclotomic character, and by $\hr_1$,  $\hr_2$ the universal characters associated to the two factors (associated with the coordinates $T_1$ and $T_2$) of the weight space.

\subsubsection{}
We are now in a position to introduce the notion of a \emph{family of automorphic representations}, which will be used in our subsequent applications.

\begin{definition}
A {\it family of automorphic representations} $\underline{\Pi}$ of tame level $N_{\upi}$ and tame central character $\chi_{\upi}$ is the data of a finite flat extension $\mathcal R_{\underline{\Pi}}$ of $\mathcal R$ satisfying the following two conditions:
\begin{enumerate}
\item[(a)] The restriction of $H^k(\mathcal M_{\text{cusp},w_j}^{\bullet,-,\fs})$ to $\Omega_{\upi}$ is zero if $j+k \neq 3$, and the sheaves $S^k(\upi) = H^k(\mathcal M_{\text{cusp},w_{3-k}}^{\bullet,-,\fs})$ are either free over $\mathcal O(\Omega_{\upi})$ of rank 1 for all $k$, or free of rank 1 for $k=1,2$ and zero for $k=0,3$; 
\item[(b)] For any classical specialisation $\Pi$ of $\upi$ of weights $(r_1+3,r_2+3)$, the center of $G(\mathbb A_{\text{fin}}^p)$ acts on the specialisation $S^k(\Pi)$ of $S_k(\upi)$ by $|\cdot|^{-(r_1+r_2)} \chi_{\upi}$.  \hfill $\blacksquare$
\end{enumerate}
\end{definition}
We assume until the end of our paper that $\upi$ is Borel ordinary. This means that all its classical specialisations are (equivalently, some specialisation of regular weight is, cf. \cite[Theorem 2.5 (3)]{LO14}) so.

\subsubsection{}
According to \cite[\S7]{TU99}, there exists a big Galois representation 
$$ \rho_{\upi} \colon G_{\mathbb Q, \Sigma} \rightarrow \GL_4(\Frac(\mathcal R_{\upi})) $$  
attached to $\upi$, where $\Sigma$ is a finite set of primes containing all those dividing $pN_{\upi} \infty$ and $\Frac(\mathcal R_{\upi})$ stands for the field of fractions of $R_{\upi}$. 

\subsubsection{}
\label{subsubsec_2025_08_25_1049}
Thanks to our running Borel-ordinarity assumption, there exist unramified characters $\widetilde \alpha:=\alpha_{\upi}$, $\widetilde \beta:=\beta_{\upi} \chi^{\hr_2+1}$, $\widetilde \gamma:=\gamma_{\upi} \chi^{\hr_1+2}$, and $\widetilde \delta:=\delta_{\upi} \chi^{\hr_1+\hr_2+3}$ of $G_{\Qp}$ such that
$$
(\widetilde \alpha({\rm Frob}_p),\widetilde \beta({\rm Frob}_p),\widetilde \gamma({\rm Frob}_p),\widetilde \delta({\rm Frob}_p))_{\vert_\Pi} = (\alpha', p^{r_2+1}\beta', p^{r_1+1}\gamma', p^{r_1+r_2+3}\delta')
$$
for any specialisation $\Pi$ of $\upi$, where $(\alpha', \beta', \gamma', \delta')$ are the Satake parameters of $\Pi$ at $p$.

\subsubsection{Assumptions}
\label{ass-gsp4}
We shall work under the following assumptions (a)--(c) on the Galois representation $\rho_{\underline{\Pi}}$.

\begin{enumerate}
\item[(a)] The residual representation (given as in \cite[Definition 2.3]{LO14}) of $\rho_{\upi}$ is irreducible.
\end{enumerate}
As explained in \cite[Theorem 2.5]{LO14}, this condition implies (thanks to \cite{TU99,Urban2005}) that there exists a Galois stable lattice $T_{\upi} \subset \text{Frac}(\mathcal R_{\upi})^{\oplus 4}$, that we fix once and for all.
\begin{enumerate}
\item[(b)] The characters $\widetilde \alpha$, $\widetilde \beta$, $\widetilde \gamma$ and $\widetilde \delta$ are non-trivial modulo the maximal ideal of $\cR_{\upi}$ (cf. the hypothesis \eqref{item_non_anom} below) and pairwise distinct. 
\end{enumerate}
Conditions (a), (b), our Borel-ordinarity condition applied with Nakayama's lemma imply that there exists a basis of $T_\upi$ with respect to which the action of $G_{\mathbb Q_p}$ on the Galois representation is upper triangular.
\begin{enumerate}
\item[(c)] The local ring $\cR_{\upi}$ is regular.
\end{enumerate}

\subsubsection{} 
The condition that $\cR_{\upi}$ be regular can be achieved on passing to a sufficiently small wide-open neighbourhood of about a classical point with sufficiently regular weights. In what follows, we shall do so without further pointers to this arrangement.

We assume, unless we explicitly state otherwise, that these assumptions hold true for all Borel-ordinary families that appear in this article. We also assume the conditions analogous to (a)--(c) for all the automorphic forms (and not just for the families) that are relevant to our discussion.

\subsubsection{} For $\cR_{\upi}$ and $\cR_{\underline{\sigma}}$ as above, let us put $\cR_4 := \cR_{\upi} \widehat \otimes_{\mathbb Z_p} \cR_{\underline{\sigma}}  \widehat \otimes_{\mathbb Z_p} \cR_{\underline{\sigma}}$ and $\cR_3 = \cR_{\upi}  \widehat \otimes_{\mathbb Z_p} \cR_{\underline{\sigma}}$.

\subsection{Degree--16 triple products}
\label{subsec_2025_09_24_1234}
Let $\Pi \times \sigma_1 \times \sigma_2$ be an automorphic representation of $\GSp_4 \times \GL_2 \times \GL_2$. We have a degree--16 $L$-function $L(s,\Pi \times \sigma_1 \times \sigma_2)$
associated with the tensor product of the degree-4 (spin) and degree-4 (standard) representations of the $L$-groups of $\GSp_4$ and $\GL_2 \times \GL_2$. Our main results will concern products of the form $\Pi \times \sigma \times \sigma^c$, where $\sigma^c$ denotes the contragradient automorphic representation.

If $\Pi$, $\sigma_1$, and $\sigma_2$ are algebraic, then this $L$-function is expected to be motivic, and in particular, we can ask whether it has critical values. As in the previous section, suppose that the $L$-packet of $\Pi$ corresponds to a holomorphic Siegel modular form of weight $(k_1, k_2)$, with $k_1 \ge k_2 \ge 3$, and that $\sigma_1$ (resp.~$\sigma_2$) corresponds to a modular form of weight $c_1$ (resp.~$c_2$). Then we expect that there exist motives $M(\Pi)$ (of motivic weight $k_1 + k_2 - 3$), $M(\sigma_1)$ (of motivic weight $c_1 -  1$), and $M(\sigma_2)$ (of motivic weight $c_2 - 1$) such that
\[
  L(s,\Pi \times \sigma_1 \times \sigma_2) = L\left( s + w - \tfrac{1}{2}, M(\Pi) \otimes M(\sigma_1) \otimes M(\sigma_2) \right), \qquad w = \dfrac{k_1 + k_2 + c_1 + c_2}{2}-2\,.
\]

\subsection{$p$-adic $L$-functions}
Let $\underline{\Pi}$ (resp. $\underline{\Sigma}:=\underline{\sigma}_1 \otimes \underline{\sigma}_2$) be a cuspidal Borel-ordinary family of forms on $\GSp_4$ (resp. on $\GL_2 \times \GL_2$). We shall denote by $(P,Q_1,Q_2)$, where $P$ corresponds to a specialisation of $\underline{\Pi}$ and $Q_i$ to a specialisation of $\underline{\sigma}_i$, specialisations of $\cR_4$. We denote its weights by $(k_1,k_2,c_1,c_2)$ whenever $P$ is of weight $(k_1,k_2)$, and $\underline{\sigma}_i$ is of weight $c_i$. 
\subsubsection{Regions for $\GSp_4 \times \GL_2 \times \GL_2$}
\label{subsubsec_2025_08_19_0801}
In \cite[\S2.3]{LZ20}, the authors explicate the ``GGP regions'', determined in terms of the interlacing relations of the quadruple of weights $(k_1,k_2,c_1,c_2)$, given by the following inequalities.

\begin{itemize}
\item {\bf Region (a).} $c_1,c_2 \geq 1$, \,$k_1+k_2-2 \leq c_2-c_1$.
\item {\bf Region (b).} $c_1,c_2 \geq 1$,\, $k_1-k_2+2 \leq c_2-c_1 \leq k_1+k_2-4$, \,$k_1+k_2 \leq c_1+c_2$.
\item {\bf Region (c).} $c_1,c_2 \geq 1$,\,  $|c_2-c_1| \leq k_1-k_2$,\, $k_1+k_2 \leq c_1+c_2$.

\item {\bf Region (d).} $c_1,c_2 \geq 1$,\, $k_1-k_2+2 \leq c_2-c_1$\,, 
$c_1+c_2 \leq k_1+k_2-2$.
\item {\bf Region (e).} $c_1,c_2 \geq 1$,\, $|c_2-c_1| \leq k_1-k_2$,\, $k_1-k_2+4 \leq c_1+c_2 \leq k_1+k_2-2$.
\item {\bf Region (f).} $c_1,c_2 \geq 2$,\, $c_1+c_2 \leq k_1-k_2+2$.
\end{itemize}

One also has the GGP regions (a'), (b'), and (d'), in which the roles of $c_1$ and $c_2$ are reversed.

\subsubsection{Conjectures for $\GSp_4 \times \GL_2 \times \GL_2$: $p$-adic $L$-functions}
\label{subsubsec_2025_08_26}
With the notations and conventions above, Loeffler and Zerbes in \cite{LZ20} conjecture the existence of $p$-adic $L$-functions (one for each GGP region). In this subsection, we recall their formulation, with a slight strengthening that allows variation in the $\GSp_4$-factor. 
We recall from \cite[\S10]{LZ21-erl} the notion of {\it good points} (cf. Definition 10.4.2 in op. cit.).

\begin{conj}\label{conj:l-function}
For each $\diamond \in \{a,d\}$, there exists an element
$$ L_p^{(\diamond)}(\underline{\Pi} \times \underline{\Sigma}) \in \Frac(\mathbb I)  $$  
with the following interpolation property: 
$$ L_p^{(\diamond)}(\underline{\Pi} \times \underline{\Sigma})^2(P,Q_1,Q_2) = Z_S^{\diamond}(P,Q_1,Q_2) \cdot \mathcal E_p^{\diamond}(\underline{\Pi} \otimes \underline{\Sigma})(P,Q_1,Q_2) \cdot \frac{\Lambda({\tfrac{1}{2}}, \underline{\Pi}_{P} \otimes \underline{\Sigma}_{Q_1,Q_2})}{\mathcal C_1^{(\diamond)}(P) \cdot \Omega_{\infty}^{\diamond}(\underline{\Pi} \times \underline{\Sigma}(P,Q_1,Q_2))}, $$  
for all good points $(P,Q_1,Q_2)$ whose weights $(k_1,k_2,c_1,c_2)$ belong to the region $\diamond$. Here, 
\begin{itemize}
    \item $Z_S^{\diamond}(P,Q_1,Q_2)$ is a factor depending on the choice of tame data (that will not be made precise);
    \item $\mathcal E_p^{\diamond}(\underline{\Pi} \times \underline{\Sigma})(P,Q_1,Q_2))$ is an Euler-like factor at $p$ (given according to the recipe of \cite{CPR89}, see also \cite{LZ20}, \S4.3);
    \item $\Omega_{\infty}^{\diamond}(-)$ is an archimedean period that is described according to the recipe of Coates and Perrin-Riou \cite{CPR89} and $\mathcal C_1^{(\diamond)}(P)$ is a $p$-adic period $($cf. Remark~\ref{rem_2025_08_26_1505}$)$;
    \item $\Lambda(-)$ is the completed $L$-function\,.
\end{itemize}
\end{conj}

\begin{remark}
\label{rem_2025_08_26_1505}
\item[i)] The $p$-adic period $\mathcal C_1^{(\diamond)}$ arises to account for the ambiguity in the choice of Deligne's canonical period. 
\item[ii)] One may utilise an automorphic archimedean period $\Omega_{\infty}^{\circ}(-)$ in place of Deligne's canonical period (as in \cite{CPR89}), and this period can be described explicitly. When $\diamond=a$, one expects the periods to be given by the square of the Petersson norm of the corresponding $\GL_2$-factor (suitably normalised as in \cite{Hsieh2018}); cf. \cite[\S1.3]{LR-algebraicity}. Partial evidence for this expectation was obtained by B\"ocherer and Heim~\cite{BH06} in the Saito--Kurokawa case. When we work with these periods, one can take $\mathcal C_1^{(a)}=1$. 

When $\diamond=d$, the recent work of Liu \cite{liu2} in the $\GSp_4 \times \GL_2$ case expresses the period as a product of the Petersson norm attached to one of the $\GL_2$-factors and a modified Petersson norm of the $\GSp_4$-factor; see \S3.3.3 of loc.\,cit.  We expect the same also in the setting of Conjecture~\ref{conj:l-function} with $\diamond=d$. With this choice, one can take $\mathcal C_1^{(d)}=1$. 
\hfill$\blacksquare$
\end{remark}

\begin{remark}
Conjecture~\ref{conj:l-function} can be stated, as in \cite{LZ20}, in terms of the squares of Gan--Gross--Prasad periods (and the appearance of a square on the left is due to that). This is especially relevant to our study, as our point of view rests crucially on the (conjectural) interplay between arithmetic GGP and GGP for $p$-adic families (cf. our discussion in the introduction, where this relationship has been emphasised in scenarios where it is unconditionally established).  
\hfill $\blacksquare$
\end{remark}

\begin{remark}
In GGP regions (b) and (e), the Bloch--Kato--Tate conjecture predicts the existence of families of cycles classes, and the $p$-adic GGP philosophy envisages a link to the $p$-adic $L$-functions of ``adjacent regions''. The case of region (e) is the subject of by Hsu, Jin, and the third-named author, but region (b) seems to be currently out of reach. \hfill$\blacksquare$
\end{remark}

\subsubsection{A conjectural $p$-adic $L$-function for $\GSp_4 \times \GL_3$}
As in the previous sections, we shall notate by $(k_1,k_2)$ the weights of the Siegel modular form and let $c$ denote the weight of an elliptic modular form, where in the present section, we shall consider Gelbart--Jacquet lifts of this form to $\GL_3$. 

\begin{conj}\label{conj:l-function_bis_bis}
There exists an element 
$$ L_p(12e)= L_p^{(e)}(\underline{\Pi} \times \ad^0(\underline{\sigma})) \in \Frac(\mathbb I)  $$  
that is characterised by the following interpolation property: For all good points $(P,Q)$ with weights $(k_1,k_2,c,c)$ in the region (e), we have
$$L_p^{(e)}(\underline{\Pi} \times \ad^0(\underline{\sigma}))(P,Q)=Z_S^{(e)}(\upi \times \ad^0(\underline{\sigma}))(P,Q) \cdot \mathcal E_p^{(e)}(\underline{\Pi} \otimes \ad^0(\underline{\sigma}))(P,Q) \cdot \frac{\Lambda({\tfrac{1}{2}},\underline{\Pi}_{P} \otimes \ad^0(\underline{\sigma})_{Q})}{\mathcal C_2^{(e)}(P) \cdot \Omega_{\infty}^{(e)}(\underline{\Pi} \times \ad^0(\underline{\sigma})(P,Q))}.  $$ 
Here: 
\begin{itemize}
    \item $Z_S^{(e)}(\upi \times \ad^0(\underline{\sigma}))(P,Q)$ is a factor that depends only on the tame data (which we shall not make explicit here);
    \item $\mathcal E_p^{(e)}(\underline{\Pi} \times \ad^0(\underline{\sigma}))(P,Q)$ is an Euler-like factor at $p$ (given according to the recipe of \cite{CPR89}); 
    \item $\Omega_{\infty}^{(e)}( - )$ is the Coates--Perrin-Riou archimedean period and $\mathcal C_2^{(e)}(P)$ is a $p$-adic period $($cf. Remark~\ref{rem_2025_08_26_1505}\textup{(i)}$)$;
\end{itemize}
\end{conj}


\begin{remark} We briefly review the restriction of GGP regions to the locus $c_1=c_2$, to clarify why we consider the conjectural description of $L_p(12\diamond)$ only when $\diamond=e$. 

In regions (a), (b), and (d), there are no classical specialisations with $c_1=c_2$ in the non-endoscopic cases. Moreover, when $\varepsilon(\upi)=-1$ (as we assume in our treatment), the central critical $L$-values identically vanish in the region (c). One is therefore left with the region (e), where the weights $(k_1,k_2,c)$ satisfy
     \[ 
     k_1 = k_2\,, \quad 2 \leq c \leq k_1 - 1\,, \quad k_1-k_2+4 \leq 2c \leq k_1+k_2-2\,.
     \]

 \end{remark}

\subsection{Factorisation conjectures}
We are now ready to state our factorisation conjectures. 
\subsubsection{} In what follows, we shall consider the scenario where $\underline{\Sigma}=\underline{\sigma}\times\underline{\sigma}^c$ in Conjecture~\ref{conj:l-function}. Our aim in this section is to discuss, in light of Artin formalism, the expected factorisation formulae for the $p$-adic $L$-functions
\[
   L_p^{(\diamond)}(\underline{\Pi} \times {\rm ad} \,\underline{\sigma})^2:= L_p^{(\diamond)}(\underline{\Pi} \times \underline{\Sigma})^2_{\vert_{Q_1=Q_2}}
  \qquad\diamond\in\{a,d\}\,,
\]
which is the main purpose of study of this paper.


\subsubsection{}

Based on the wall-crossing principle \ref{item_BDP1} from the introduction, we formulate the following factorisation conjecture. 

\begin{conj}
\label{conj_2025_08_21_1552}
Assume that $\varepsilon(\upi_\kappa)=-1$ for some (equivalently, all) specialisations $\upi_\kappa$ of $\upi$. Let {$\Log_c$} denote the Perrin-Riou exponential map of \S\ref{subsubsec_2025_09_23_0739}. 
Then the family\footnote{The construction of this class is not complete, but a key step in doing so is \cite[Theorem 9.6.4]{LSZ22}.} $\mathrm{LF}_{\underline{\Pi}}^{\dagger}$ of Lemma--Flach class $($cf. \cite{LSZ22}, Theorem 9.6.4$)$ is contained in the Selmer group 
$\mathbf{R}^1\Gamma_{\fin}(G_{\mathbb Q,\Sigma}, T_{\upi}^{\dag},  \Delta_{\mathscr{F}^2})$ (see \S\ref{sec_Selmer}) and we have the factorisation
\[
   L_p^{(d)}(\underline{\Pi} \times {\rm ad}\,\underline{\sigma})^2(P,Q,Q) 
  = \mathcal C_d(P) \cdot 
    L_p^{(e)}(\underline{\Pi} \times {\rm ad}^0\,\underline{\sigma}) (P,Q) \cdot 
    \Log_c \bigl(\mathrm{LF}_{\underline{\Pi}}^{\dagger}\bigr),
\]
where $\mathcal C_d$ is an element in $\cR_\upi[1/p]$. 
\end{conj}

\subsubsection{} We now formulate the second factorisation conjecture, which is based on the double-wall crossing principle \ref{item_BDP2}.

Note that in this case, it follows from our running assumptions that 
$$\varepsilon(\upi_\kappa)=-1=\varepsilon(\upi_\kappa\times {\rm ad}^0 \underline{\sigma}_\lambda),\qquad\qquad   \forall\,(\kappa,\lambda) \in \cX^{\rm cl}_c\,.$$ 
Hence, the Bloch--Kato--Tate conjecture predicts the existence of a supply of cycles (associated to both families of degree-4 and degree-12 motives) to account for the systematic vanishing of the central $L$-values.

\begin{conj}
\label{conj_2025_08_22_0031}
Assume that $\varepsilon(\upi_\kappa)=-1$ for some (equivalently, all) specialisations of $\upi$, and that the family $\upi$ does not admit endoscopic specialisations. Let {$\Log_c$} denote the Perrin-Riou exponential map of \S\ref{subsubsec_2025_09_23_0739}. Then the family $\mathrm{LF}_{\underline{\Pi}}^{\dagger}$ of Lemma--Flach class  is contained in  
$\mathbf{R}^1\Gamma_{\fin}(G_{\mathbb Q,\Sigma}, T_{\upi}^{\dag},  \Delta_{\mathscr{F}^2})$ and 
there exists an element 
$\Delta_{\underline{\Pi} \otimes {\rm ad}\,\underline{\sigma}}^{\dagger}\in \mathbf{R}^1\Gamma_{\fin}(G_{\mathbb Q,\Sigma}, M_3^{\dag}, \tr^* \Delta_c)$ (see Definition \ref{def:T_and_M_3_selmer}), ``a family of twisted diagonal cycles'' associated to $T_{\upi} \otimes {\rm ad}^0\, T_{\underline{\sigma}}$, such that 
\begin{align*}
     L_p^{(a)}(\underline{\Pi} \times {\rm ad}\,\underline{\sigma})^2(P,Q,Q) =S(M,\upi)\cdot \mathcal C_a(P) \cdot 
    {\Log_c}  \bigl(\Delta_{\underline{\Pi} \otimes {\rm ad}\,\underline{\sigma}}^{\dagger}\bigr)^2 \cdot 
    {\Log_c}  \bigl(\mathrm{LF}_{\underline{\Pi}}^{\dagger}\bigr)\,,
\end{align*}
where $\mathcal C_a$ is a non-zero element in $\cR_\upi[1/p]$ that interpolates explicit algebraic fudge factors\,, and $S(M,\upi)$ is the factor\footnote{We record here for the convenience of the reader that $S(M,\upi)$ is a $p$-adic period that measures the $p$-local relative positions of the conjectural cycle $\Delta_{\underline{\Pi} \otimes {\rm ad}\,\underline{\sigma}}^{\dagger}$ and $\mathrm{LF}_{\underline{\Pi}}^{\dagger}$ relative to the splittings determined by ${\rm tr}$ vs ${\rm tr}^*$. We refer the reader to \S\ref{subsubsec_2025_07_18_1307} for a slightly more detailed discussion.} given as in \S\ref{subsubsec_2025_07_18_1307}.
\end{conj}
The conjectured factorisation statement therefore takes the form of a degree-2 $p$-adic GGP formula that we have alluded to in \S\ref{subsubsec_2025_08_18_1601} (in the degenerate scenario we have placed, where the underlying family of motives decompose), and it should be thought of as a higher BDP / $p$-adic Waldspurger formula that we have discussed in \S\ref{sec_Intro}. We refer the reader to \S\ref{sec_2025_02_06_0708} where we present evidence for the existence of the family of cycles $\Delta_{\underline{\Pi} \otimes {\rm ad}\,\underline{\sigma}}^{\dagger}$ as well as to variant of this conjecture in endoscopic cases\footnote{The reason why these cases are omitted from Conjecture~\ref{conj_2025_08_22_0031}, and why the analogues conjecture assumes a different shape in the endoscopic scenario, is explained at the start of \S\ref{sec_2025_02_06_0708}; see also \S\ref{subsubsec_2025_07_22_1459}.}. We also refer the reader to Theorem~\ref{thm_main_double_wall_crossing} and Corollary~\ref{cor_2025_08_14_1551}, where an algebraic form of this conjecture is established.

\section{Selmer complexes, Perrin-Riou logarithm, and Eichler--Shimura morphisms}
\label{sec_Selmer}
\subsection{Generalities on Selmer complexes}
Suppose that $T$ is a representation of $G_\QQ$ over $R$ unramified outside a finite set of places $\Sigma$. Suppose that the restriction of $T$ to $G_p:=G_{\QQ_p}$ fits in a short exact sequence 
$$  0 \longrightarrow \mathscr{F}^+ T \longrightarrow T \longrightarrow T/{\mathscr{F}}^+T \longrightarrow 0 $$  
of $G_{p}$-modules. 

\subsubsection{Greenberg local conditions} 
\label{subsubsec_2025_02_10_1518}
For $v \in \Sigma$, let $I_v \subset G_v$ be the inertia subgroup and let $\Fr_v \in G_v/I_v$ denote the geometric Frobenius element. 
We consider the following Greenberg local conditions (in the sense of Nekov\'a\v{r}) on $T$ on the level of continuous cochains:
$$  U_v^+(T) = \begin{cases} C^{\bullet}(G_p, {\mathscr F}^+T) & \text{ if } v = p, \\ C^{\bullet}(G_v/I_v, T^{I_v}) & \text{ if } v \in \Sigma \backslash \{p\}. \end{cases}  $$ 
These are equipped with a morphism of complexes 
$$  \iota_v^+ : U_v^+(T) \longrightarrow C^{\bullet}(G_v,T), \quad \text{ for all } v \in \Sigma.  $$  
We note that when $v \neq p$, the complex $U_v^+(T)$ is quasi-isomorphic to the complex $\Big( T^{I_v} \xrightarrow{\Fr_v-1} T^{I_v} \Big)$ concentrated in degrees 0 and 1.

In this paper, we assume that the following \emph{Tamagawa number triviality} hypothesis: 
\begin{itemize}
    \item[\mylabel{item_Tam}{\bf Tam})] $H^1(I_v, T)$ is a free $R$-module for any prime $v \in \Sigma$. 
\end{itemize}
Under this assumption, $U_v^+(T)$ is perfect for each prime $v \in \Sigma$ (see \cite[\S4.3]{BS_triple_factor}), and hence the Selmer complexes that appear later are also perfect.

\subsubsection{Selmer complexes associated to Greenberg local conditions}
\label{sec:greenberg_selmer}
Let $\Delta$ be a Greenberg local condition. We define the Selmer complex associated to $(T,\Sigma,\Delta)$ on setting 
$$  \widetilde C_{\fin}^{\bullet}(G_{\mathbb Q, \Sigma},T,\Delta) := \cone \Big( C^{\bullet}(G_{\mathbb Q, \Sigma},T) \oplus U_{\Sigma}^+(T) \xrightarrow{\res_{\Sigma}-\iota_{\Sigma}^+} C_{\Sigma}^{\bullet}(T) \Big) [-1].  $$  We denote the corresponding object in the derived category by $\mathbf{R}\Gamma_{\fin}(G_{\mathbb Q, \Sigma},T,\Delta)$ and its cohomology by $\mathbf{R}^\bullet\Gamma_{\fin}(G_{\mathbb Q,\Sigma},T,\Delta)$. Recall that we may compute the Euler--Poincar\'e characteristic of the complex by $$  \chi(\mathbf{R}\Gamma_{\fin}(G_{\mathbb Q, \Sigma},T,\Delta)) = \rank(T^{c=1}) - \rank({\mathscr F}^+T).  $$ 

\subsection{The cases of $\GL_2$ and $\GL_3$}\label{subsubsec_2025_09_14_1444}

Let $g$ be a normalized cuspidal eigenform of weight $\ell \geq 2$ which is new away from $p$, and let $V_p(g)$ denote its associated Galois representation over a sufficiently large local field $E/\mathbb{Q}_p$ containing the image of the Hecke field of $g$ under the fixed isomorphism $\iota_p \colon \mathbb{C}\to \CC_p$.
We assume henceforth that $V_p(g)$ is $p$-ordinary and crystalline. 
 We let $\mathfrak{a}$ denote the unique root of the Hecke polynomial of $g$ at $p$ which is a $p$-adic unit\footnote{More precisely, the $p$-adic valuation of the image of $\mathfrak{a}$ under the chain of fixed embedding $\overline{Q}\xrightarrow{\jmath}\CC \xrightarrow{\iota_p}\CC_p$ is $0$.}. By slight abuse of notation, we denote the unramified character of $G_p$ mapping the geometric Frobenius to $\mathfrak{a}$ also by $\mathfrak{a}$.

\subsubsection{} The representation $V_p(g)$ admits a $G_p$-stable filtration 
$${\mathscr F}^\bullet V_p(g)\,: V_p(g)={\mathscr F}^0 V_p(g)\supsetneq {\mathscr F}^1 V_p(g)\supsetneq{\mathscr F}^2 V_p(g)=\{0\}\,,$$
where $G_p$ acts on the graded piece ${\mathscr Gr}^0 V_p(g)$ via $\mathfrak{a}$ and via $\chi^{1-\ell}\mathfrak{a}^{-1}$ on ${\mathscr Gr}^1 V_p(g)$.

Moreover, ${\mathscr F}^i V_p(g)$ is the exact the annihilator of ${\mathscr F}^{2-i}V_p(g^c)$ with respect to the Poincar\'e duality pairing, where $g^c(z)=\overline{g(\overline{z})}$ is the conjugate form. We remark that we have a natural isomorphism $V_p(g)^*\simeq V_p(g^c)(\ell-1)$ induced from Poincar\'e duality. We put 
$${\mathscr F}^iV_p(g)^*:={\rm Hom}(V_p(g)/{\mathscr F}^{2-i}V_p(g),E)\simeq {\mathscr F}^iV_p(g^c)(\ell-1)\,.$$

\subsubsection{} 
\label{subsubsec_2025_02_10_1533}
Similarly, the adjoint Galois representation $V_p(\ad^0(g))=\ker\left(V_p(g)\otimes V_p(g)^*\xrightarrow{\rm ev} E\right)$ admits the following filtration of $G_p$-modules:
\[
V_p(\ad^0(g)) = {\mathscr F}^0 V_p(\ad^0(g)) \supsetneq {\mathscr F}^1 V_p(\ad^0(g)) \supsetneq {\mathscr F}^2 V_p(\ad^0(g)) \supsetneq {\mathscr F}^3 V_p(\ad^0(g)) = \{0\}.  
\]
Here 
\begin{align*}
    {\mathscr F}^1 V_p(\ad^0(g)) &:= \ker(\ad^0(g) \longrightarrow \Hom( {\mathscr F}^1 V_p(g),  {\mathscr G}r^1 V_p(g))), 
    \\
    {\mathscr F}^2 V_p(\ad^0(g)) &:= {\mathscr F}^1 V_p(g) \otimes {\mathscr F}^1 V_p(g)^*. 
\end{align*}
Note that the $G_p$-action on each of the one-dimensional graded pieces 
$${\mathscr Gr}^0 V_p(\ad^0(g)):= V_p(\ad^0(g)) /{\mathscr F}^1 V_p(\ad^0(g)), \quad {\mathscr Gr}^1 V_p(\ad^0(g)):={\mathscr F}^1 V_p(\ad^0(g))/{\mathscr F}^2 V_p(\ad^0(g))\,,$$ 
$${\mathscr Gr}^2 V_p(\ad^0(g)):={\mathscr F}^2 V_p(\ad^0(g))$$ 
is given by $\mathfrak a\chi^{\ell-1}$, $\mathds{1}$, and $\mathfrak a ^{-1}\chi^{1-\ell}$, respectively. 

\subsubsection{} 

The analogous definitions are available also when $g$ is replaced by our Hida family $\underline{\sigma}$.

\subsection{Siegel modular forms}\label{subsec:SMF}
Let $\Pi$ be a cuspidal automorphic representation of ${\rm GSp}_4{}_{/\QQ}$ of weight $(k_1,k_2)$ with trivial central character. We assume that $k_1\geq k_2\geq 3$ and write $r_1 = k_1-3$ and $r_2 = k_2-3$. As in \S\ref{subsec:Siegel}, let $(\alpha,\beta,\gamma,\delta)$ stand for the Hecke parameters of the Galois representation $V_p(\Pi)$, with $v_p(\alpha) \leq v_p(\beta) \leq v_p(\gamma) \leq v_p(\delta)$. We also assume that $\Pi$ is not a Saito--Kurokawa lift.

We recall that one says that $\Pi_p$ is Klingen-ordinary if $\frac{\alpha \beta}{p^{k_2-2}}$ is a $p$-adic unit and that it is Siegel-ordinary if $\alpha$ is a $p$-adic unit. We say that it is Borel-ordinary (or simply, ordinary) if it is both Klingen and Siegel ordinary. In that case, $v_p(\alpha) = 0$, $v_p(\beta) = k_2-2$, $v_p(\gamma) = k_1-1$, and $v_p(\delta) = k_1+k_2-3$.

\subsubsection{Duality}
We henceforth assume that $\Pi_p$ is Borel-ordinary. Under the assumptions of \S\ref{ass-gsp4}, the $G_{\QQ_p}$ representation $V_p(\Pi)$ has a complete flag of $G_p$-stable subspaces ${\mathscr F}^\bullet V_p(\Pi)$, where ${\mathscr F}^i V_p(\Pi)$ is of codimension $i$ (where $0\leq i\leq 4$). Further, following our previous discussion the Hodge--Tate weights of $V_p(\Pi)$ are $0$, $k_2-2$, $k_1-1$, and $k_1+k_2-3$ (cf. \cite[Theorem 10.1.3.6]{LSZ22}). Then, as noted in \cite[\S6.3]{LZ} and since the central character $\chi_\Pi$ of $\Pi$ is trivial by assumption, we have a symplectic pairing 
\[
V_p(\Pi) \otimes V_p(\Pi) \longrightarrow E(-k_1-k_2+3)
\] 
induced by Poincar\'e duality. Under this pairing, ${\mathscr F}^i  V_p(\Pi)$ is the exact annihilator of ${\mathscr F}^{4-i} V_p(\Pi)$.

Let us put 
\[
V_p^\dagger(\Pi):=V_p(\Pi)\left(\tfrac{k_1+k_2}{2}-1\right)\,,
\]
so that the pairing above gives rise to the symplectic pairing 
$V_p^{\dagger}(\Pi) \otimes V_p^{\dagger}(\Pi) \to E(1)$.

\subsection{Graded pieces}


\subsubsection{}
Let us consider the filtration
$$  
V_p(\Pi) = {\mathscr F}^0 V_p(\Pi) \supsetneq {\mathscr F}^1 V_p(\Pi) \supsetneq {\mathscr F}^2 V_p(\Pi) \supsetneq {\mathscr F}^3 V_p(\Pi) \supsetneq {\mathscr F}^4 V_p(\Pi) = \{0\},  $$  with graded pieces 
$$  
\begin{aligned} 
\gr^0 V_p(\Pi) = {\mathscr F}^0 V_p(\Pi) /{\mathscr F}^1 V_p(\Pi), \quad & \gr^1 V_p(\Pi) = {\mathscr F}^1 V_p(\Pi)/{\mathscr F}^2 V_p(\Pi), 
\\
\gr^2 V_p(\Pi) = {\mathscr F}^2 V_p(\Pi)/{\mathscr F}^3 V_p(\Pi), \quad & \gr^3 V_p(\Pi) = {\mathscr F}^3 V_p(\Pi)/{\mathscr F}^4 V_p(\Pi)\,. 
\end{aligned}  
$$
The crystalline Frobenius $\varphi$ acts on $\DD_{\rm cris}(\gr^0 V_p(\Pi)), \ldots, \DD_{\rm cris}(\gr^3 V_p(\Pi))$ by $\alpha$, $\beta$, $\gamma$, and $\delta$, respectively.

\subsubsection{Families of Siegel modular forms}
Recall our conventions for (two-variable) families of ordinary Siegel modular forms fixed in \ref{subsec:families}. The analogous definitions are available also when $\Pi$ is replaced by a Hida family $\upi$, satisfying both the Klingen and Siegel ordinary conditions, as discussed in \cite[\S23]{LZ20-bk} and \cite[\S10]{LZ21-erl}. In particular, we may still consider a filtration
$$ T_\upi = {\mathscr F}^0T_\upi \supsetneq {\mathscr F}^1T_\upi \supsetneq {\mathscr F}^2T_\upi \supsetneq {\mathscr F}^3T_\upi \supsetneq {\mathscr F}^4T_\upi = \{0\},  $$  as well as the corresponding graded pieces $$  \begin{aligned} 
\gr^0T_\upi = {\mathscr F}^0T_\upi /{\mathscr F}^1T_\upi, \quad & \gr^1T_\upi = {\mathscr F}^1T_\upi/{\mathscr F}^2T_\upi, \\ \gr^2T_\upi = {\mathscr F}^2T_\upi/{\mathscr F}^3T_\upi, \quad & \gr^3T_\upi = {\mathscr F}^3T_\upi/{\mathscr F}^4T_\upi\,. \end{aligned}  
$$ 
As before, let $(\hk_1, \, \hk_2)$ be the universal characters associated to the 2-dimensional weight space $\Omega$. We set
\[ 
T^{\dagger}_\upi := T_\upi \left(\tfrac{\hk_1+\hk_2}{2} - 1 \right).  
\] 
Then there is a skew-symmetric pairing
$$  
 T_\upi  \otimes T_\upi \longrightarrow \cR_{\upi}(\hk_1+\hk_2-3),  
$$  
with respect to which ${\mathscr F}^i T_\upi$ is the exact annihilator of ${\mathscr F}^{4-i}T_\upi$; equivalently, there is a skew-symmetric pairing
$$  
T^{\dagger}_\upi \otimes T^{\dagger}_\upi \longrightarrow \cR_{\upi}(1).   
$$ 

\subsection{Triple products}
\label{subsec_2025_08_25_1436}
With the assumptions of \ref{subsubsec_2025_09_14_1444} and \ref{subsec:SMF}, let $g$ be a modular forms of weight $\ell$ and $\Pi$ a cuspidal representation of weight $(k_1,k_2)$ with trivial central characters and $k_1 \geq k_2 \geq 3$. Let us put $c:=\frac{k_1+k_2}{2}+\ell-2$. 
Let
$$  
V := V_p(\Pi) \otimes V_p(g) \otimes V_p(g^c)(c)\simeq  V_p^\dagger(\Pi) \otimes V_p(g) \otimes V_p(g)^*.  
$$   
We also define a Galois stable $\cO_E$-lattice $T\subset V$, which comes associated with the choice of lattices $T_? \subset V_p(?)$ (which we henceforth fix). We put
$$  
M := T_\Pi^\dagger \otimes \ad^0 T_g,\,\quad M[1/p] := V_p^\dagger(\Pi) \otimes \ad^0V_p(g).  
$$ 

\subsubsection{} We shall describe the corresponding Greenberg local conditions associated with the regions of weights described in \cite[\S4]{LZ20}.  These will be given in terms of $G_p$-stable subspaces of $V$ of dimension $8$, as described in the definition below. To highlight their relevance to the corresponding region, we shall notate them by $F_?^+X$ with $?=a,b,c,d,e$.

\begin{definition}\label{def:T_and_M_3_selmer}
We put
\begin{align*}
    F_a^+V&:= V_p^\dagger(\Pi) \otimes {\mathscr F}^1 V_p(g) \otimes V_p(g)^*\,,\\
    F_b^+V&:= {\mathscr F}^1 V_p^\dagger(\Pi) \otimes {\mathscr F}^1 V_p(g) \otimes V_p(g)^* + {\mathscr F}^3 V_p^\dagger(\Pi) \otimes V_p(g) \otimes  {\mathscr F}^1 V_p(g)^* \\
    &\hspace{6cm}+ V_p^\dagger(\Pi) \otimes {\mathscr F}^1 V_p(g) \otimes {\mathscr F}^1 V_p(g)^*\,,\\
    F_c^+V&:=V_p^\dagger(\Pi) \otimes {\mathscr F}^1 V_p(g) \otimes {\mathscr F}^1 V_p(g)^* + {\mathscr F}^2 V_p^\dagger(\Pi) \otimes {\mathscr F}^1 V_p(g) \otimes V_p(g)^* \\
    &\hspace{6cm} + {\mathscr F}^2 V_p^\dagger(\Pi) \otimes V_p(g) \otimes {\mathscr F}^1 V_p(g)^*\,,\\
    F_d^+V&:= {\mathscr F}^1 V_p^\dagger(\Pi) \otimes {\mathscr F}^1 V_p(g) \otimes V_p(g)^* + {\mathscr F}^3 V_p^\dagger(\Pi) \otimes V_p(g) \otimes V_p(g)^*\,,\\
    F_e^+V&:={\mathscr F}^3 V_p^\dagger(\Pi) \otimes V_p(g) \otimes V_p(g)^* + {\mathscr F}^1 V_p^\dagger(\Pi) \otimes {\mathscr F}^1 V_p(g) \otimes {\mathscr F}^1 V_p(g)^* \\ & \qquad + {\mathscr F}^2 V_p^\dagger(\Pi) \otimes {\mathscr F}^1 V_p(g) \otimes V_p(g)^* 
    + {\mathscr F}^2 V_p^\dagger(\Pi) \otimes V_p(g) \otimes {\mathscr F}^1 V_p(g)^*\,,\\\\
    &\hspace{-3.1cm}\hbox{We analogously define $F_?^+T$, and put}\\ 
    F^+_?M&:= \im(F_{?}^+ T \hookrightarrow T \xrightarrow{\rm tr^*} M)\,,\qquad ?=a,b,c,d,e\,,
\end{align*}
where $\rm tr^*$ is the map induced from the dual of the map ${\rm ad}^0T_g=\ker({\rm tr})\hookrightarrow {\rm End}(T_g)\simeq T_g \otimes T_g^*$\,.
\end{definition}

\begin{lemma}
    \label{lemma_2025_02_09_1953}
    We have the following isomorphisms of $G_p$-representations.
\begin{align*}
F_a^+T/(F_a^+T \cap F_b^+T) &\simeq {\mathscr Gr}^0 T^\dagger_{\Pi} \otimes{\mathscr F}^1 T_g \otimes  {\mathscr Gr}^0 T_g^* \simeq {\mathscr Gr}^0 T^\dagger_{\Pi}\,,
\\
F_b^+T/(F_c^+T\cap F_b^+T) &\simeq {\mathscr Gr}^1 T^\dagger_{\Pi} \otimes {\mathscr F}^1 T_g \otimes  {\mathscr Gr}^0 T_g^* \simeq {\mathscr Gr}^1 T^\dagger_{\Pi}
\end{align*}
\end{lemma}

\begin{proof}
    Observing that 
    \[
    F_a^+T \cap F_b^+T = {\mathscr F}^1 T^\dagger_{\Pi} \otimes {\mathscr F}^1 T_g \otimes  T_g^* + T^\dagger_{\Pi} \otimes {\mathscr F}^1 T_g \otimes {\mathscr F}^1 T_g^*
    \]
    and
    \[
    F_c^+ T \cap F_b^+T = T^\dagger_{\Pi} \otimes {\mathscr F}^1 T_g \otimes {\mathscr F}^1 T_g^*+{\mathscr F}^2T^\dagger_{\Pi} \otimes {\mathscr F}^1 T_g \otimes  T_g^*+ {\mathscr F}^3T^\dagger_{\Pi} \otimes T_g \otimes {\mathscr F}^1 T_g^* \,,
    \]
 the asserted isomorphisms follow from a direct calculation.
\end{proof}

\subsubsection{} The analogous constructions are available when $\Pi$ is replaced by a Hida family $\upi$ of Siegel modular forms, and $g$ replaced by a primitive Hida family $\underline{\sigma}$ of elliptic modular forms. We note that the conjugate family $\underline{\sigma}^c$ is defined as $\underline{\sigma} \otimes {\chi}_{\underline{\sigma}}^{-1}$, where ${\chi}_{\underline{\sigma}}$ is the tame nebentype of the family $\underline{\sigma}$.

\subsubsection{Greenberg local conditions for families}
\begin{definition}
For $? \in \{a,b,c,d,e\}$, let $$  F_?^+ T_3^{\dag} := F_?^+ T_4^{\dag} \otimes_{\iota_{3,4}^*} \mathcal R_3.  $$ 
We also define the Greenberg-local condition $\tr^* \Delta_{?} = \Delta(F_{?}^+ M_3^{\dag})$ by $$  F_{?}^+ M_3^{\dag} := \im(F_{?}^+ T_3^{\dag} \hookrightarrow T_3^{\dag} \rightarrow M_3^{\dag}).  $$ 
\end{definition}

\begin{lemma}\label{lemma:ranks}
\item[(a)] We have 
$$  
\rank_{\mathcal R_{\underline{\sigma}}} F_a^+ M_3^{\dag} = 8, \quad \rank_{\mathcal R_{\underline{\sigma}}} F_b^+ M_3^{\dag} = 7, \quad \rank_{\mathcal R_{\underline{\sigma}}} F_c^+ M_3^{\dag} = 6, 
 $$ 
$$  \rank_{\mathcal R_{\underline{\sigma}}} F_d^+ M_3^{\dag} = 7, \quad \rank_{\mathcal R_{\underline{\sigma}}} F_e^+ M_3^{\dag} = 6.  $$ 
\item[(b)] We have the following natural short exact sequences 
$$  0 \longrightarrow F_b^+ M_3^{\dag} \longrightarrow F_a^+ M_3^{\dag} \longrightarrow \gr^0 T^\dagger_{\underline{\Pi}} \otimes \gr^1 \ad^0(T_{\underline{\sigma}}) \longrightarrow 0,  $$  
$$  0 \longrightarrow F_c^+ M_3^{\dag} \longrightarrow F_b^+ M_3^{\dag} \longrightarrow \gr^1 T^\dagger_{\underline{\Pi}} \otimes \gr^1 \ad^0(T_{\underline{\sigma}}) \longrightarrow 0,  $$ 
and 
$$  0 \longrightarrow F_e^+ M_3^{\dag} \longrightarrow F_d^+ M_3^{\dag} \longrightarrow \gr^1 T^\dagger_{\underline{\Pi}} \otimes  \gr^1 \ad^0(T_{\underline{\sigma}}) \longrightarrow 0.  $$ 
\end{lemma}

\begin{proof}
This follows from the definitions. 
\end{proof}

\subsection{Selmer complexes (bis)}
For $?=a,b,c,d,e$, let us denote by $\Delta_a$ the Greenberg local conditions (in the sense of \S\ref{subsubsec_2025_02_10_1518}) on $T_3^\dagger$ given by $\iota_p^+: F_?^+T_3^\dagger\hookrightarrow T_3^\dagger$. We denote by ${\rm tr}^*\Delta_?$ the Greenberg local condition on $M_3^\dagger$ that is given by $\iota_p^+: F_?^+M_3^\dagger\hookrightarrow M_3^\dagger$.

\begin{definition}
Let us denote by $\res_{a/b}$ the composition 
\begin{align*}
    \mathbf{R}^1\Gamma_{\fin}(G_{\mathbb Q,\Sigma}, M_3^{\dag}, \tr^* \Delta_a) \xrightarrow{\res_p} H_{\fin}^1(G_p, F_a^+M_3^{\dag}) \lra H^1(G_p, F_a^+ M_3^{\dag}/F_b^+ M_3^{\dag}) \simeq & \,H^1(G_p,{\mathscr Gr}^0 T^\dagger_{\underline{\Pi}}\,\widehat \otimes \gr^1 \ad^0(T_{\underline{\sigma}}))  \\
    &\quad \simeq H^1(G_p,{\mathscr Gr}^0 T^\dagger_{\underline{\Pi}})  \otimes_{\cR_{\upi}} \cR_3\,,
\end{align*} 
where the first map is given by
\begin{align*}
    \mathbf{R}^1\Gamma_{\fin}(G_{\mathbb Q,\Sigma}, M_3^{\dag}, \tr^* \Delta_a)\ni [c,(c_v),(\lambda_v)] \longmapsto \res_p([c])=\iota_p^+([c_p]) \in H^1(G_p,F_a^+ M_3^\dagger)\,,
\end{align*}
the third arises from Lemma \ref{lemma:ranks}(b), and the final isomorphism follows from the fact that $\gr^1 \ad^0(T_{\underline{\sigma}}) \simeq R_{\underline{\sigma}}$ as $G_p$-modules, as explained in \S\ref{subsubsec_2025_02_10_1533}. Similarly, we define the maps 
$$  \res_{b/c}\,:\,\mathbf{R}^1\Gamma_{\fin}(G_{\mathbb Q,\Sigma}, M_3^{\dag}, \tr^* \Delta_b) \lra  H^1(G_p, F_b^+ M_3^{\dag}/F_c^+ M_3^{\dag})\simeq  H^1(G_p,{\mathscr Gr}^1 T^\dagger_{\underline{\Pi}}) \widehat \otimes_{\cR_{\upi}} \cR_3\,,  $$  
$$  \res_{d/e}\,:\,\mathbf{R}^1\Gamma_{\fin}(G_{\mathbb Q,\Sigma}, M_3^{\dag}, \tr^* \Delta_d) \lra  H^1(G_p, F_d^+ M_3^{\dag}/F_e^+ M_3^{\dag})\simeq  H^1(G_p,{\mathscr Gr}^1 T^\dagger_{\underline{\Pi}}) \widehat \otimes_{\cR_{\upi}} \cR_3\,.  $$  
\end{definition}

\subsection{Eichler--Shimura isomorphisms for families of Siegel modular forms}
\label{subsec_2025_02_04_1413}
In this subsection, we discuss how to naturally trivialise the $p$-local cohomology of  the graded pieces of the Galois representations associated with families of Siegel modular forms, relying on the Eichler--Shimura isomorphisms proved in \cite[\S11]{LZ21-erl}. As usual, let $L$ be a $p$-adic field. We borrow the following definition from Definition 10.4.1. of op.\,cit.

\begin{definition}
For a fixed family $\upi$, let $S^i( \upi)$ be the $\upi$-eigenspace in the degree-$i$ coherent cohomology of the corresponding Siegel threefold, and with values in the $p$-adic field $L$.
    \hfill$\blacksquare$
\end{definition}

\begin{remark}
\label{remark_2025_07_21_1028}
    These spaces are 1-dimensional for each $i\in \{0,1,2,3\}$ for ``generic'' families of Siegel forms, whereas for those of Yoshida type, they are 1-dimensional for $i=1,2$, and 0-dimensional when $i=0,3$.
    \hfill$\blacksquare$
\end{remark}

\subsubsection{}  Let us fix a crystalline specialisation $\Pi$ of $\upi$ of weight $(r_1+3,r_2+3)$, with $r_1 \geq r_2$. Then $\DD_{\rm cris}(V_p(\Pi))$ is endowed with a decreasing Hodge filtration 
$$\{0\}=\Fil_{\text{Hdg}}^{r_1+r_2+4} V_p(\Pi)\subset\cdots\subset \Fil_{\text{Hdg}}^i V_p(\Pi) \subset \Fil_{\text{Hdg}}^{i-1} V_p(\Pi)\subset \cdots\subset \Fil_{\text{Hdg}}^0 V_p(\Pi):= \DD_{\rm cris}(V_p(\Pi))\,,$$
with $\Gr_{\text{Hdg}}^i V_p(\Pi):=\Fil_{\text{Hdg}}^{i} V_p(\Pi)/\Fil_{\text{Hdg}}^{i+1} V_p(\Pi)$ for $i=0,r_2+1,r_1+2, r_1+r_2+3$ (which are the Hodge--Tate weights of this Galois representation) the non-trivial graded pieces. 

\subsubsection{}\label{subsubsec_2025_07_22_1459}
Until \S\ref{sec_2025_02_06_0708} (where we treat endoscopic cases), we shall work under the following assumption. As we shall explain, it is especially crucial for the statement of Eichler--Shimura isomorphism (Proposition~\ref{prop:ES} below).

\begin{ass}\label{ass:non-crit}
The family\footnote{or rather, its $p$-refinement $(\upi,\mathscr{F}=\{\mathscr{F}^i\})$.} $\upi$ does not admit any critical specialisations, in the sense of Bella\"iche--Chenevier. This means, for any crystalline specialisation $\Pi$ of $\upi$ of weight $(r_1+3,r_2+3)$ with $r_1 \geq r_2$, the following holds true:
\begin{equation}
\label{eqn_2025_07_22_1212}
    \Fil_{\text{Hdg}}^{r_i+3-i} V_p(\Pi) \cap \mathscr{F}^{i+1}V_p(\Pi) =\{0\}=\Fil_{\text{Hdg}}^{r_1+r_2+3} V_p(\Pi) \cap \mathscr F^1 V_p(\Pi)\,, \qquad i=1,2\,.
\end{equation}
\hfill$\blacksquare$
\end{ass}

\subsubsection{}
\label{subsubsec_2025_07_22_1547}
We briefly review Eichler--Shimura morphisms at a fixed weight. To that end, let us fix a crystalline specialisation $\Pi$ of $\upi$ as in \S\ref{subsubsec_2025_07_22_1459}, and let us assume that its $p$-refinement given by $\{\mathscr{F}^iV_p(\Pi)\}$ verifies \eqref{eqn_2025_07_22_1212} (i.e., this refined pair is non-critical). As usual, we let $\varphi$ denote the crystalline Frobenius acting on $V_p(\Pi)$, and we let $\alpha, \ldots, \delta$ be the Hecke parameters of $\Pi$ at $p$, given as in Section \ref{subsec:SMF} (which coincide with the eigenvalues of $\varphi$). We then have the isomorphisms
\begin{align}
    \label{eqn_2025_07_22_1539}
    \begin{aligned}
            S^0(\Pi,L)&\simeq {\rm Gr}^{r_1+r_2+3}_{\rm Hdg}\, V_p(\Pi)\,,\qquad S^3(\Pi,L)\simeq {\rm Gr}^{0}_{\rm Hdg}\, V_p(\Pi)\,,\\
            S^i(\Pi,L)&\simeq {\rm Gr}^{r_i+3-i}_{\rm Hdg}\, V_p(\Pi)\,,\qquad i=1,2\,,
    \end{aligned}
\end{align}
induced from the comparison isomorphisms of Faltings and Tsuji (cf. \cite{LZ21-erl}, \S11.4). Thanks to our running non-criticality assumption \eqref{eqn_2025_07_22_1212}, we have the following isomorphisms:
\begin{align}
    \label{eqn_2025_07_22_1548}
    \begin{aligned}
    {\rm Gr}^{r_1+r_2+3}_{\rm Hdg} V_p(\Pi)= {\rm Fil}^{r_1+r_2+3}_{\rm Hdg}\, V_p(\Pi)\xrightarrow{\,\sim\,}  \DD_{\rm cris}(V_p(\Pi))/\DD_{\rm cris}(\mathscr{F}^1 V_p(\Pi))=\DD_{\rm cris}(\gr^{0} V_p(\Pi))\,,\\
          \DD_{\rm cris}(\gr^3 V_p(\Pi))= \DD_{\rm cris}(\mathscr{F}^3 V_p(\Pi))  \xrightarrow{\,\sim\,} \DD_{\rm cris}(V_p(\Pi))/{\rm Fil}^{r_1+2}_{\rm Hdg} V_p(\Pi)={\rm Gr}^{0}_{\rm Hdg}V_p(\Pi)\,,\\\\
          {\rm Fil}^{r_i+3-i}_{\rm Hdg} V_p(\Pi)\cap \DD_{\rm cris}(\mathscr{F}^i V_p(\Pi)) \xrightarrow{\sim} {\rm Fil}^{r_i+3-i}_{\rm Hdg} V_p(\Pi)/{\rm Fil}^{r_1+r_2-r_i+i}_{\rm Hdg} V_p(\Pi)={\rm Gr}^{r_i+3-i}_{\rm Hdg}V_p(\Pi) \\
           \xrightarrow{\sim} \DD_{\rm cris}(\mathscr{F}^i V_p(\Pi))/ \DD_{\rm cris}(\mathscr{F}^{i+1} V_p(\Pi))=\DD_{\rm cris}(\gr^i V_p(\Pi)) \,,\,\quad i=1,2\,.
    \end{aligned}
\end{align}
Combining \eqref{eqn_2025_07_22_1539} with \eqref{eqn_2025_07_22_1548}, we arrive at the following (punctual) Eichler--Shimura isomorphisms:
\begin{equation}
    \label{eqn_2025_07_22_1457} 
    \begin{aligned}
            S^i(\Pi,L)&\simeq \DD_{\rm cris}(\gr^{i} V_p(\Pi)),\qquad i=0,1,2,3\,.
    \end{aligned}
\end{equation}
\subsubsection{} We record below the following result due to Loeffler and Zerbes (extending the work of Diao--Rosso--Wu in \cite{DRW21}; see also \cite{DRW25}), on the interpolation of the Eichler--Shimura morphisms \eqref{eqn_2025_07_22_1457} in families of Siegel modular forms.

Let $R$ be a domain and let $X$ be a free $R$-module of rank one, which is equipped with a continuous $G_p$-action, such that the action of $G_p$ on $X\otimes \eta^{-1}$ is unramified for some character $\eta$ of $\Gamma_{\cyc}$. In this scenario, we put $X^\circ:=X\otimes \eta^{-1}$ and define (rather abusively) $ \DD_{\cris}(X):=\DD_{\cris}(X^{\circ})$. We remark that $X={\mathscr Gr}^{i}T_\upi$ have this property.
\begin{proposition}[Loeffler--Zerbes]
\label{prop:ES}
Under Assumption~\ref{ass:non-crit}, we have isomorphisms 
$$  \ES^i \colon S^i( \upi) \cong \DD_{\cris}\left({\mathscr Gr}^{i}T_\upi\right)\,,\qquad \forall \,i\in \{0,1,2,3\}\,.$$
\end{proposition}

\begin{proof}
This follows from the discussion of \cite[Remark 11.4.1]{LZ21-erl} (see also \cite[Theorem 11.6.3]{LZ21-erl}). 
\end{proof}

\begin{remark}
The condition~\ref{eqn_2025_07_22_1212} is not explicitly assumed in \cite{LZ21-erl}, but it is rather implied by the underlying assumptions in op. cit. More precisely, the authors explain that the Eichler--Shimura isomorphisms \eqref{eqn_2025_07_22_1457} can be interpolated over the locus of {\it good points}, cf. \cite[Definition 10.4.2]{LZ21-erl}. Then a good Borel-ordinary point of regular weight as in op. cit. is necessarily of numerically-non-critical (in the sense of \cite{BC09}, Remark 2.4.6), and therefore non-critical (i.e. it verifies Assumption~\ref{ass:non-crit}). Our emphasis on this assumption is due to our interest in formulating Eichler--Shimura isomorphisms in the form of \eqref{eqn_2025_07_22_1457}.     \hfill $\blacksquare$
\end{remark}

\subsubsection{}   Note that we have a pairing 
\begin{equation}
\label{eqn_2025_07_22_1451}
    \begin{aligned} \langle \,\cdot\,,\, \cdot\, \rangle\, \colon\, & \DD_{\cris}({\mathscr Gr}^iT_\upi(\hr_1+\hr_2+3)) \times \DD_{\cris}({\mathscr Gr}^{3-i}T_\upi) \longrightarrow \mathcal O(U) \,,\qquad i=0,1,2,3\,, 
    \end{aligned}  
\end{equation} 
that comes about interpolating Poincar\'e duality pairings (cf. \cite{LZ20}, \S6.3).

\begin{definition}
For any $\eta \in S^i( \upi)$, we denote by
$$ \ES_{\upi}^i(\eta)\, \colon\, \DD_{\cris}({\mathscr Gr}^{3-i}T_\upi) \longrightarrow \mathcal O(U)$$
the homomorphism induced from the pairing \eqref{eqn_2025_07_22_1451}. \hfill$\blacksquare$
\end{definition}

We henceforth fix a choice of a $\cR_\upi$-basis $\{\eta_i\}$ of $S^i(\upi,L)$.


\section{Trivialisations and modules of leading terms}
We recall in this section the notion of the module of leading terms of algebraic $p$-adic $L$-functions introduced in \cite[\S5]{BS_triple_factor}, as well its key properties. We will review the main general constructions in op. cit. \S\ref{sec_Koly_Sys} at a great level of generality. We will then apply these in \S\ref{sec_factorisation_main} to the factorisation problem at hand. 

We also refer the reader to \cite[Remark 5.14]{BS_triple_factor} for the connection of the module of leading terms to $p$-adic $L$-functions.

\subsection{The setup}
For any commutative ring $S$ and any finitely generated $S$-module $M$, we put 
$$ 
M^* := \Hom_{S}(M, S) \,\,\, \textrm{ and } \,\,\, {\bigcap}^t_{S}M := \left({\bigwedge}^t_{S}M^*\right)^*\,, 
 $$ 
and call ${\bigcap}^t_{S}M$ the $t$-th exterior bi-dual of $M$. 

\begin{remark}
    When $G$ is a finte abelian group and  $S = \mathbb{Z}[G]$, the $t$-th exterior bi-dual  ${\bigcap}^t_{S}M$ of $M$ is canonically isomorphic to  the Rubin's lattice $\wedge^t_0M$ of $M$ defined in \cite{Rubin96}. 
    The exterior bi-dual is a natural generalisation of the Rubin's lattice. 
    \hfill$\blacksquare$
\end{remark}

Let $R$ be a  complete Noetherian local ring with finite residue field of characteristic $p\geq3$ and $T$ be a free $R$-module of finite rank with a continuous $G_{\Q, \Sigma}$-action. 

Throughout \S\ref{sec_Koly_Sys}, we assume that 
we have an $R[G_p]$-submodule ${\mathscr F}^+T$ of $T$ such that the quotient $T/{\mathscr F}^+T$ is free as an $R$-module. 
Then we have the Greenberg local conditions $\Delta := \Delta_{\mathscr F^+}$; cf. \S\ref{subsubsec_2025_02_10_1518}. 
We also assume throughout \S\ref{sec_Koly_Sys} that the condition (\ref{item_Tam}) holds. For simplicity, let us put 
\[ 
r=r(T,\Delta) := - \chi(\mathbf{R}\Gamma_{\rm f}(G_{\Q, \Sigma}, T, \Delta)) = \mathrm{rank}_R(T^{c=-1}) - \mathrm{rank}_R(T/{\mathscr F}^+T) \in \Z. 
\]

\subsection{Modules of leading terms}
\label{sec_Koly_Sys}

\begin{definition}
We let 
$$\delta(T,\Delta) \subset {\bigcap}^r_R \mathbf{R}^1\Gamma_{\fin}(G_{\mathbb Q,\Sigma},T,\Delta)$$ 
denote the module of leading terms introduced in \cite[\S5.2]{BS_triple_factor}. ${}$ \hfill$\blacksquare$
\end{definition}

The first key properties of the module $\delta(T,\Delta)$ are recorded below.
\begin{theorem}
\label{thm_2023_03_10}
\item[i)] The $R$-module $\mathbf{R}^2\Gamma_{\fin}(G_{\Q, \Sigma}, T, \Delta)$  is torsion if and only if $\delta(T, \Delta)$ is generated by an $R$-regular element of ${\bigcap}^{r}_{R}\mathbf{R}^1\Gamma_{\fin}(G_{\Q, \Sigma}, T, \Delta)$. 
\item[ii)] Suppose in addition that $R$ is normal. 
If $\delta(T, \Delta) \neq 0$, then the $R$-module $\mathbf{R}^2\Gamma_{\fin}(G_{\Q, \Sigma}, T, \Delta)$  is torsion and 
$$ 
\mathrm{char}_R\left( \left. 
{\bigcap}^r_R \mathbf{R}^1\Gamma_{\fin}(G_{\Q, \Sigma}, T, \Delta)   \middle/  \delta(T, \Delta)  \right. \right) 
= 
\mathrm{char}_R\left( \mathbf{R}^2\Gamma_{\fin}(G_{\Q, \Sigma}, T, \Delta) \right). 
 $$ 
\item[iii)] For any flat ring homomorphism $R \ra S$, we have $\delta(T, \Delta) \otimes_R S = \delta(T \otimes_R S , \Delta)$. 
\end{theorem}

\begin{proof}
This is Proposition 5.5 and Theorem 5.6 of \cite{BS_triple_factor}.
\end{proof}

\subsubsection{}
\label{subsubsec_2025_02_11_1127}
We assume that we have an $R[G_p]$-submodule ${\mathscr F}T$ of ${\mathscr F}^+T$ such that 
\begin{itemize}
    \item the quotient ${\mathscr F}^+T/{\mathscr F}T$ is free as an $R$-module and
    \item $H^0(G_p, {\mathscr F}^+\overline{T}/{\mathscr F}\overline{T}) = H^2(G_p, {\mathscr F}^+\overline{T}/{\mathscr F}\overline{T}) = 0$, where $\overline{T}$ denotes the residual representation of $T$. 
\end{itemize} 
Now we have two Greenberg local conditions $\Delta_{\mathscr F^+}$ and $\Delta_{\mathscr F}$. 
For simplicity, let us put 
$$ 
r :=r(T,\Delta_{{\mathscr F}^+}) \,\,\, \textrm{ and } \,\,\, s := r(T,\Delta_{\mathscr F}). 
 $$ 
We also suppose that 
\begin{itemize}
\item $s \geq 0$. 
\end{itemize}
Since $H^0(G_p, \mathscr{F}^+\overline{T}/\mathscr{F}\overline{T}) = H^2(G_p, \mathscr{F}^+\overline{T}/\mathscr{F}\overline{T}) = 0$ by assumption, the $R$-module $H^1(G_p, {\mathscr F}^+T/{\mathscr F}T)$ is free  of rank $r-s$. 
Fix an $R$-isomorphism 
$$ 
\varphi := \bigoplus_{1 \leq i \leq r-s}\varphi_i \colon H^1(G_p, {\mathscr F}^+T/{\mathscr F}T) \stackrel{\sim}{\longrightarrow} R^{r-s},  
 $$ 
where $\varphi_i \in \Hom_R( H^1(G_p, {\mathscr F}^+T/{\mathscr F}T), R)$. 
Then the homomorphisms $\varphi_1, \ldots, \varphi_{r-s}$ induce an $R$-homomorphism 
\begin{align*}
    \Phi \colon {\bigcap}_R^r \mathbf{R}^1\Gamma_{\fin}(G_{\mathbb Q,\Sigma},T,\Delta_{{\mathscr F}^+}) \longrightarrow {\bigcap}_R^s  \mathbf{R}^1\Gamma_{\fin}(G_{\mathbb Q,\Sigma},T,\Delta_{\mathscr F}) 
\end{align*}
given by 
$$ 
\Phi(x) := x(- \wedge (\varphi_1 \circ \res_p) \wedge \cdots \wedge (\varphi_{r-s} \circ \res_p) ). 
 $$ 

\begin{lemma}\label{lemma_delta_F^+=delta_F_when_R=PID}
    When $R$ is a discrete valuation ring, we have 
    $$ 
    \Phi(\delta(T, \Delta_{\mathscr F^+})) = \delta(T, \Delta_{\mathscr F}). 
     $$ 
\end{lemma}
\begin{proof}
When 
\[
\mathrm{rank}_R \im(\varphi \circ \res_p \colon \mathbf{R}^1\Gamma_{\fin}(G_{\mathbb Q,\Sigma},T,\Delta_{{\mathscr F}^+}) \longrightarrow R^{r-s}) < r-s, 
\]
we have $(\varphi_1 \circ \res_p) \wedge \cdots \wedge (\varphi_{r-s} \circ \res_p) = 0$, and hence 
$$ 
 \Phi(\delta(T, \Delta_{\mathscr F^+})) = 0. 
 $$ 
Since $\mathrm{rank}_R \mathbf{R}^2\Gamma_{\fin}(G_{\mathbb Q,\Sigma},T,\Delta_{\mathscr F}) \geq \mathrm{rank}_R \coker(\varphi \circ \res_p) \geq 1$, the $R$-module $\mathbf{R}^2\Gamma_{\fin}(G_{\mathbb Q,\Sigma},T,\Delta_{\mathscr F})$ is not torsion. Hence Theorem \ref{thm_2023_03_10}(i) shows $\delta(T, \Delta_{\mathscr F}) = 0$. 
Therefore, we obtain $ \Phi(\delta(T, \Delta_{\mathscr F^+})) = 0 = \delta(T, \Delta_{\mathscr F})$. 

Suppose that $\mathrm{rank}_R \im(\varphi \circ \res_p) = r-s$. 
We then have 
$$ 
\mathrm{rank}_R \mathbf{R}^2\Gamma_{\fin}(G_{\mathbb Q,\Sigma},T,\Delta_{\mathscr F}) = \mathrm{rank}_R \mathbf{R}^2\Gamma_{\fin}(G_{\mathbb Q,\Sigma},T,\Delta_{\mathscr F^+}).
 $$ 
If $\mathrm{rank}_R \mathbf{R}^2\Gamma_{\fin}(G_{\mathbb Q,\Sigma},T,\Delta_{\mathscr F^+}) > 0$, then Theorem \ref{thm_2023_03_10}(i) shows $\Phi(\delta(T, \Delta_{\mathscr F^+})) = 0 = \delta(T, \Delta_{\mathscr F})$. 
Hence we may assume that both $\mathbf{R}^2\Gamma_{\fin}(G_{\mathbb Q,\Sigma},T,\Delta_{\mathscr F})$ and $ \mathbf{R}^2\Gamma_{\fin}(G_{\mathbb Q,\Sigma},T,\Delta_{\mathscr F^+})$ are torsion. 
Put 
\begin{align*}
    \ell_{\mathscr F} &:= \mathrm{length}_R\mathbf{R}^2\Gamma_{\fin}(G_{\mathbb{Q},\Sigma},T,\Delta_{\mathscr F}) < +\infty, 
    \\
    \ell_{{\mathscr F}^+} &:= \mathrm{length}_R\mathbf{R}^2\Gamma_{\fin}(G_{\mathbb{Q},\Sigma},T,\Delta_{\mathscr F^+}) < +\infty. 
\end{align*}
Since $R$ is a discrete valuation ring, we have isomorphisms
$$ 
\mathbf{R}^1\Gamma_{\fin}(G_{\mathbb Q,\Sigma},T,\Delta_{\mathscr F}) \cong  R^s \,\,\, \textrm{ and } \,\,\, \mathbf{R}^1\Gamma_{\fin}(G_{\mathbb Q,\Sigma},T,\Delta_{\mathscr F^+}) \cong R^r. 
 $$ 
Moreover, by replacing $\varphi$ with other isomorphism if necessary, one can take a basis $\{e_1, \ldots, e_r\}$ of $\mathbf{R}^1\Gamma_{\fin}(G_{\mathbb Q,\Sigma},T,\Delta_{\mathscr F^+})$ such that 
\begin{align*}
    \varphi_i \circ \res_p(e_j) = \pi^{a_i}\delta_{i+s, j}. 
\end{align*}
Here $\delta_{i+s, j}$ denotes the Kronecker's delta, $\pi$ is a unifomizer of $R$, and $a_i$ is a non-negative integer. Note that  $a_{1} + \cdots + a_{r-s} = \ell_F - \ell_{F^+}$. 
Let $\{e_1^*, \ldots, e_r^*\}$ denote the dual basis of $\{e_1, \ldots, e_r\}$. 
Then the element $x \in {\bigcap}^r_R\mathbf{R}^1\Gamma_{\fin}(G_{\mathbb Q,\Sigma},T,\Delta_{\mathscr F^+})$ defined by $x(e_1^* \wedge \cdots \wedge e_r^*) = 1$ is a basis and  the element $y \in {\bigcap}^s_R\mathbf{R}^1\Gamma_{\fin}(G_{\mathbb Q,\Sigma},T,\Delta_{\mathscr F})$ defined by $y(e_1^* \wedge \cdots \wedge e_s^*) = 1$ is  a basis. 
Then 
$$ 
\Phi(x)(e_1^* \wedge \cdots \wedge e_s^*) = x(e_1^* \wedge \cdots \wedge e_s^* \wedge (\varphi_1 \circ \res_p) \wedge \cdots \wedge (\varphi_{r-s} \circ \res_p)) = \pi^{\sum a_i} = \pi^{\ell_{F}-\ell_{F^+}}. 
 $$ 
Hence $\Phi(x) = \pi^{\ell_{F}-\ell_{F^+}}y$ and Theorem \ref{thm_2023_03_10}(ii) shows that 
$$ 
\Phi(\delta(T, \Delta_{\mathscr F^+})) = \Phi(R\pi^{\ell_{F^+}}x) = R\pi^{\ell_{F}}y = \delta(T, \Delta_{\mathscr F}). 
 $$ 
\end{proof}

\begin{lemma}\label{lemma_cyclic_module_is_determined_by_height_one_primes}
Suppose that $R$ is normal. 
Let $M$ be a torsion-free $R$-module and let $f, g \in M$.
Assume that for every height-one prime $\fp$ of $R$, there exists a unit $u_\fp \in R_\fp^{\times}$ such that $f = u_\fp g$ in $M_\fp$.
Then there exists a unit $u \in R^{\times}$ such that $f = ug$.
\end{lemma}
\begin{proof}
Since there is an injective homomorphism from $Rf + Rg$ into $R$, we may assume that $M = R$.
If we consider $R$ and each $R_\fp$ as subrings of the field $\mathrm{Frac}(R)$, the assumption implies that
\[
\frac{f}{g}, \, \frac{g}{f} \in \bigcap_{\mathrm{ht}(\fp) = 1} R_{\fp}\,.
\]
As $R$ is normal, we have $\bigcap_{\mathrm{ht}(\fp) = 1} R_{\fp} = R$, which shows that $f/g \in R^\times$.
\end{proof}

\begin{theorem}
\label{thm_2025_02_19_2217}
    Suppose that $R$ is normal. 
    Then 
        $$ 
    \Phi(\delta(T, \Delta_{\mathscr F^+})) = \delta(T, \Delta_{\mathscr F}). 
     $$ 
\end{theorem}
\begin{proof}
By Lemma \ref{lemma_delta_F^+=delta_F_when_R=PID} and Theorem \ref{thm_2023_03_10}(ii), we have
\[
\Phi(\delta(T, \Delta_{\mathscr F^+}))R_\fp = \delta(T, \Delta_{\mathscr F})R_\fp
\]
for every height-one prime $\fp$ of $R$.
Since both $\Phi(\delta(T, \Delta_{\mathscr F^+}))$ and $\delta(T, \Delta_{\mathscr F})$ are cyclic $R$-modules, we may apply Lemma \ref{lemma_cyclic_module_is_determined_by_height_one_primes} with
\[
M = \bigcap\nolimits_R^s \mathbf{R}^1\Gamma_{\fin}(G_{\mathbb Q,\Sigma},T,\Delta_{\mathscr F}),
\]
taking $f$ and $g$ to be generators of $\Phi(\delta(T, \Delta_{\mathscr F^+}))$ and $\delta(T, \Delta_{\mathscr F})$, respectively.
It then follows that
\[
\Phi(\delta(T, \Delta_{\mathscr F^+})) = \delta(T, \Delta_{\mathscr F}).
\]
\end{proof}

\subsubsection{Perrin-Riou's large exponential map}
\label{subsec_2025_02_04_1412}
Following closely the discussion in \cite{BS_triple_factor}, we introduce Perrin-Riou's large logarithm maps. These, together with the input from \S\ref{subsec_2025_02_04_1413}, will be used to formulate our factorisation conjecture (and also to trivialise naturally various local cohomology groups, e.g. in place of $\varphi_i$ in \S\ref{subsubsec_2025_02_11_1127} above). 

As in previous sections, let $R$ be a complete local Noetherian domain and let $X$ be a free $R$-module of rank one, which is equipped with a continuous $G_p$-action. Let us assume that $X=X^\circ\otimes \eta$ where the $G_p$-representation $X^\circ$ is unramified and $\eta$ is a character of $\Gamma_{\cyc}$. We consider the following condition on $X$:
\begin{itemize}   
\item[\mylabel{item_non_anom}{\bf NA})] $H^2(G_p,X)=0$\,.
\end{itemize}
We remark that, by local duality, this condition is equivalent to the requirement that $H^0(G_p,\overline{X}^*(1))=0$, where $\overline{X}$ is the residual representation. 

The theory of Coleman maps gives rise to an isomorphism (cf. \cite{KLZ2}, \S8.2)
\begin{equation}
\label{eqn_2022_09_13_0843}    
{\rm LOG}_{X^\circ,0}\,:\,H^1(G_p,X^\circ\widehat{\otimes}_{\ZZ_p} \ZZ_p[[\Gamma_{\cyc}]])\stackrel{\sim}{\lra} \DD_{\rm cris}(X^\circ)\,\widehat{\otimes}_{\ZZ_p}\, \Lambda(\Gamma_{\cyc}) \,,
\end{equation}
where $\DD_{\rm cris}(X^\circ)=(X^\circ\, \widehat\otimes\, \widehat{\ZZ_p^{\mathrm ur}})^{G_{p}}$. 
Let us define the morphism 
\begin{equation}
\label{eqn_2022_09_13_0900}    
{\rm Log}_{X}\,:\, H^1(G_p,X)\lra \DD_{\rm cris}(X^\circ) 
\end{equation}
on tensoring the isomorphism \eqref{eqn_2022_09_13_0843} of 
$R[[\Gamma_{\cyc}]]$-modules by $R[[\Gamma_{\cyc}]]\big{/}(\gamma-1)$ (where $\gamma\in \Gamma_{\cyc}$ is a topological generator) and relying on the fact that the natural injection 
$$ 
H^1(G_p,X\widehat{\otimes}_{R}R[[\Gamma_{\cyc}]])\Big{/}(\gamma-1)\lra H^1(G_p,X)
 $$ 
is an isomorphism whenever \eqref{item_non_anom} holds, in which case the map ${\rm Log}_{X}$ is an isomorphism.


\subsection{Leading terms: Arithmetic examples}\label{subsec:leading_terms}
We apply the general formalism of the preceding subsection in the context of Galois representations attached to forms on ${\rm GSp}_4\times {\rm GL}_2\times {\rm GL}_2$, and their subquotients.

\subsubsection{Large exponential maps (bis)} 
\label{subsubsec_2025_09_23_0739}
Using the construction in \S\ref{subsec_2025_02_04_1412} with the $\cR_{\upi}$-module ${\mathscr Gr}^iT^\dagger_{\underline{\Pi}}$ for $i=0,1,2$, we obtain trivialisations\footnote{We remark that the hypothesis \eqref{item_non_anom} follows from our running Assumption~\ref{ass-gsp4}(b).}
\begin{align}
    \label{eqn_2025_02_09_0946} 
    \begin{aligned}
        {\rm Log}_{\eta_i}\,&:\,H^1(G_p,{\mathscr Gr}^iT^\dagger_{\underline{\Pi}}) \xrightarrow[{\rm Log}_{{\mathscr Gr}^iT^\dagger_{\underline{\Pi}}}]{\sim} \DD_{\rm cris}({\mathscr Gr}^iT^{\dagger}_{\underline{\Pi}}{}^\circ) \xrightarrow[\,\eta_i\,]{\sim} \cR_{\upi}\,,\\
    \end{aligned}
\end{align}
which depends on the choice of the isomorphism $\eta_i$. We henceforth fix a choice\footnote{A natural choice arises from the discussion in \S\ref{subsec_2025_02_04_1413}, after passing to the field of fractions of respective universal Hecke algebra. We will not need this input for our results concerning the factorisation of modules of leading terms.} of such $\eta_i$ and set 
\begin{align}
    \label{eqn_2025_02_09_1632} 
    \begin{aligned}
        {\rm Log}_{a/b}\,&:\,H^1(G_p,{\mathscr Gr}^0 T^\dagger_{\underline{\Pi}}) \xrightarrow[\sim]{{\rm Log}_{\eta_0}}\cR_{\upi}\,,\\
        {\rm Log}_{b/c}\,&:\,H^1(G_p,{\mathscr Gr}^1T^\dagger_{\underline{\Pi}}) \xrightarrow[\sim]{{\rm Log}_{\eta_1}}\cR_{\upi}\,,\\
        {\rm Log}_{c}\,&:\,H^1(G_p,{\mathscr Gr}^2T^\dagger_{\underline{\Pi}}) \xrightarrow[\sim]{{\rm Log}_{\eta_2}}\cR_{\upi}\,.
    \end{aligned}
\end{align}

Depending on the context (which is determined by which exact sequence in Lemma~\ref{lemma:ranks}(b) is relevant to our discussion), we will also write ${\rm Log}_{d/e}$ in place of ${\rm Log}_{b/c}$. Also, with a slight abuse of notation, we will use the same symbols to denote both the maps $H^1(G_p,{\mathscr Gr}^iT^\dagger_{\underline{\Pi}})\otimes_{\cR_{\upi}}\cR_3 \to \cR_3$
obtained from \eqref{eqn_2025_02_09_1632} via base change to $\cR_3$, and the composition of the restriction map at $p$ from a relevant Selmer group with this map. For example, we denote both of the following homomorphisms by the same notation $\Log_c$: 
\begin{align*}
    &\Log_c \colon \mathbf{R}^1\Gamma_{\fin}(G_{\mathbb Q,\Sigma}, T_{\upi}^{\dag},  \Delta_{\mathscr{F}^2})  \longrightarrow  H^1(G_p,{\mathscr Gr}^2T^\dagger_{\underline{\Pi}}) \longrightarrow \cR_{\upi}, 
\\
    &\Log_c \colon \mathbf{R}^1\Gamma_{\fin}(G_{\mathbb Q,\Sigma}, M_3^{\dag}, \tr^* \Delta_c) \longrightarrow H^1(G_p,{\mathscr Gr}^2T^\dagger_{\underline{\Pi}})\otimes_{\cR_{\upi}}\cR_3 \longrightarrow \cR_3. 
\end{align*}


\section{Factorisation of algebraic $p$-adic $L$-functions}
\label{sec_factorisation_main}$\,$
In this section, we establish the factorisations of modules of leading terms. We will carry this out for regions (d) and (a) separately, as these two cases exhibit fundamental differences (cf. \S\ref{remark_intro_BDP_principle}). 

We assume throughout this section, without loss of generality\footnote{i.e. on base-changing to its normal closure.}, that $\cR_3$ is normal.
\subsection{Wall-crossing: (d) to (e)}
\label{subsec_5_1_2025_02_11}
Unlike the case ``(a) to (b)'' (which involves new phenomena), this case can be handled following the ideas of \cite{BS_triple_factor}.

We note that our main results in this subsection are guided by the (BDP) ``wall-crossing principle'' (see \S\ref{remark_intro_BDP_principle}): The restriction of the (conjectural) $p$-adic $L$-function for the region (d) to ${\rm im}({\rm Spec}\,\cR_3\xrightarrow{\iota_{3,4}}{\rm Spec}\,\cR_4)$ has empty interpolation range. Moreover, ${\rm im}(\iota_{3,4})$ contains a dense set of classical points that fall within the region (e) (where the global root number equals $-1$), and the values of the (conjectural) $p$-adic $L$-function for the region (d) at such points in the region (e) should be a $p$-adic avatar for the first derivative of the degree--16 $L$-function at the central critical point.  

\begin{theorem}
\label{thm_5_1_2025_02_14_1424}
Suppose that $\delta(T_3^{\dag},\Delta_d) \neq 0 =\delta(T^\dagger_{\underline{\Pi}}, \Delta_{\mathscr{F}^2})$. Then,
$$  
\delta(T_3^{\dag}, \Delta_d) = \Log_{d/e}(\delta(M_3^{\dag}, \tr^* \Delta_d)) \cdot \varpi_{3,2}^* \Log_{c}(\delta(T^\dagger_{\underline{\Pi}}, \Delta_{\mathscr{F}^1})).
 $$  
In particular, $\delta(M_3^{\dag}, \Delta_e)=\Log_{d/e}(\delta(M_3^{\dag}, \tr^* \Delta_d)) \neq 0$.
\end{theorem}

We will complete the proof of this theorem after some preparatory steps. Before we proceed, let us briefly comment on the hypotheses of this theorem. If we had\footnote{Folklore conjectures imply that this case should never arise.} $\delta(T_3^{\dag},\Delta_d) =0$, then it is easy to see that the claimed factorisation in Theorem~\ref{thm_5_1_2025_02_14_1424} reduces to the uninteresting assertion that $0=0$. The required vanishing $\delta(T^\dagger_{\underline{\Pi}}, \Delta_{\mathscr{F}^2})=0$ under our running assumption on global $\varepsilon$-factors is a consequence of the parity conjecture for ${\rm GSp}_4$.

\begin{lemma}
\label{lemma_5_2_2025_02_14_1214}
We have the following commutative diagram with exact rows:
\footnotesize
\begin{align*}
\xymatrix{
\mathbf{R}\Gamma_{\fin}(G_{\mathbb Q, \Sigma}, T_{\upi}^\dagger, \Delta_{\mathscr{F}^3}) \otimes_{\varpi_{3,2}^*} \cR_3 \ar[r]\ar[d]^-{=}
&
\mathbf{R}\Gamma_{\fin}(G_{\mathbb Q, \Sigma}, T_3^\dagger, \Delta_{d}) \ar[r]^{{\rm tr}^*}\ar[d]^{\rm tr}
&
\mathbf{R}\Gamma_{\fin}(G_{\mathbb Q, \Sigma}, M_3^\dagger, \mathrm{tr}^*\Delta_{d}) \ar[r]^-{+1} \ar[d]^{\res_p}
&
\\
   \mathbf{R}\Gamma_{\fin}(G_{\mathbb Q, \Sigma}, T_{\upi}^\dagger, \Delta_{\mathscr{F}^3})\otimes_{\varpi_{3,2}^*} \cR_3  \ar[r]
    &
   \mathbf{R}\Gamma_{\fin}(G_{\mathbb Q, \Sigma}, T_{\upi}^\dagger, \Delta_{\mathscr{F}^1}) \otimes_{\varpi_{3,2}^*} \cR_3  \ar[r]
    &
    \mathbf{R}\Gamma\left(G_{p}, {\mathscr{F}^1}/{\mathscr{F}^3}\right) \otimes_{\varpi_{3,2}^*} \cR_3  \ar[r]^-{+1}
    &\,,
    }
\end{align*}
\normalsize
Here, we have use the shorthand ${\mathscr{F}^i}/{\mathscr{F}^j}:={\mathscr{F}^iT_{\upi}^\dagger}\,/\,{\mathscr{F}^jT_{\upi}^\dagger}$ for any $i<j$.
\normalsize
\end{lemma}

\begin{lemma}
\label{lemma_5_3_2025_02_14_1213}
   Suppose that $\delta(T_3^{\dag},\Delta_d) \neq 0$.  
   \item[i)] $\mathbf{R}^1\Gamma_{\fin}(G_{\mathbb Q, \Sigma}, T^\dagger_{\underline{\Pi}}, \Delta_{\mathscr{F}^3}) = 0 = \mathbf{R}^1\Gamma_{\fin}(G_{\mathbb Q, \Sigma}, T_3^\dagger, \Delta_{d})$. 
   \item[ii)] 
   We have 
       $$ 
       \xymatrix{
      \mathbf{R}^1\Gamma_{\fin}(G_{\mathbb Q, \Sigma}, M_3^\dagger, \mathrm{tr}^*\Delta_{d}) \ar@{^{(}->}[r]^-{\delta^1} \ar@{^{(}->}[d]_-{\res_p} & 
      \mathbf{R}^2\Gamma_{\fin}(G_{\mathbb Q, \Sigma}, T_{\upi}^\dagger, \Delta_{\mathscr{F}^3}) \otimes_{\varpi_{3,2}^*} \cR_3. 
      \\
     \dfrac{{H}^1\left(G_{p}, {\mathscr{F}^1}/{\mathscr{F}^3}\right)}{\res_p({\mathbf{R}^1\Gamma_{\fin}(G_{\mathbb Q, \Sigma}, T_{\upi}^\dagger, \Delta_{\mathscr{F}^1}))}}\otimes_{\varpi_{3,2}^*} \cR_3 \ar@{^{(}->}[ru]_-{\partial_{\underline{\Pi}}} 
    }
     $$ 
    All the $\cR_3$-modules that appear in this diagram are of rank one.
    \item[iii)]   $\delta(T_3^{\dag},\Delta_d) = \Char_{\cR_3}\left( \dfrac{\mathbf{R}^2\Gamma_{\fin}(G_{\mathbb Q, \Sigma}, T_{\upi}^\dagger, \Delta_{\mathscr{F}^3}) \otimes_{\varpi_{3,2}^*} \cR_3}{\delta^1(\delta(M_3^{\dag}, \mathrm{tr}^*\Delta_d))} \right)$. 
\end{lemma}
\begin{proof}
     Since $\delta(T_3^{\dag},\Delta_d) \neq 0$, we have 
    $\mathbf{R}^1\Gamma_{\fin}(G_{\mathbb Q, \Sigma}, T_3^\dagger, \Delta_{d}) = 0\,, $ 
     and the vanishing $\mathbf{R}^1\Gamma_{\fin}(G_{\mathbb Q, \Sigma}, T^\dagger_{\underline{\Pi}}, \Delta_{\mathscr{F}^3}) = 0$ follows from the first row of the diagram in Lemma~\ref{lemma_5_2_2025_02_14_1214}. This concludes the proof of (i). Utilising the same lemma again, we obtain the following commutative diagram with exact rows: 
    \small
    $$ 
    \xymatrix{
      \mathbf{R}^1\Gamma_{\fin}(G_{\mathbb Q, \Sigma}, T_3^\dagger, \Delta_{d}) = 0 \ar[r] \ar[d]^{\rm tr} &
      \mathbf{R}^1\Gamma_{\fin}(G_{\mathbb Q, \Sigma}, M_3^\dagger, \mathrm{tr}^*\Delta_{a}) \ar[r]^-{\delta^1} \ar[d]^{\res_p} & 
      \mathbf{R}^2\Gamma_{\fin}(G_{\mathbb Q, \Sigma}, T_{\upi}^\dagger, \Delta_{\mathscr{F}^3}) \otimes_{\varpi_{3,2}^*} \cR_3 \ar[d]^{=}
      \\
      \mathbf{R}^1\Gamma_{\fin}(G_{\mathbb Q, \Sigma}, T_{\upi}^\dagger, \Delta_{\mathscr{F}^1})\otimes_{\varpi_{3,2}^*} \cR_3  \ar[r]
    &
    H^1(G_{p}, \mathscr{F}^1/\mathscr{F}^3) \otimes_{\varpi_{3,2}^*} \cR_3 \ar[r]_-{\partial_{\upi}}
    & \mathbf{R}^2\Gamma_{\fin}(G_{\mathbb Q, \Sigma}, T_{\upi}^\dagger, \Delta_{\mathscr{F}^3}) \otimes_{\varpi_{3,2}^*} \cR_3. 
    }
     $$ 
    \normalsize
    The proof of (ii) follows immediately chasing this diagram. Part (iii) follows from the exactness of the sequence 
    \begin{align*}
        0\lra &\,\frac{\mathbf{R}^1\Gamma_{\fin}(G_{\mathbb Q, \Sigma}, M_3^\dagger, \mathrm{tr}^*\Delta_{d})}{\delta(M_3^{\dag}, \mathrm{tr}^*\Delta_d)}\xrightarrow{\delta^1}\frac{\mathbf{R}^2\Gamma_{\fin}(G_{\mathbb Q, \Sigma}, T_{\upi}^\dagger, \Delta_{\mathscr{F}^3}) \otimes_{\varpi_{3,2}^*} \cR_3}{\delta^1(\delta(M_3^{\dag}, \mathrm{tr}^*\Delta_d))}\\
        &\lra \mathbf{R}^2\Gamma_{\fin}(G_{\mathbb Q, \Sigma}, T_3^\dagger, \Delta_{d})\lra  \mathbf{R}^2\Gamma_{\fin}(G_{\mathbb Q, \Sigma}, M_3^\dagger, \mathrm{tr}^*\Delta_{d})\lra  \mathbf{R}^3\Gamma_{\fin}(G_{\mathbb Q, \Sigma}, T_{\upi}^\dagger, \Delta_{\mathscr{F}^3})=0
    \end{align*}
    combined with Theorem~\ref{thm_2023_03_10}(ii).
\end{proof}

\begin{lemma}\label{lemma_5_4_2025_02_14_1324}
Suppose that $\delta(T_3^{\dag},\Delta_d) \neq 0$. Then:
\begin{align*}
    \Char_{\cR_3}\left( \frac{\mathbf{R}^2\Gamma_{\fin}(G_{\mathbb Q, \Sigma}, T_{\upi}^\dagger, \Delta_{\mathscr{F}^3}) \otimes_{\varpi_{3,2}^*} \cR_3}{\delta^1(\delta(M_3^{\dag}, \mathrm{tr}^*\Delta_d))} \right)= & 
    \Char_{\cR_3}\left( \frac{\dfrac{H^1(G_{p}, \mathscr{F}^1/\mathscr{F}^3)}{\res_p(\mathbf{R}^1\Gamma_{\fin}(G_{\mathbb Q, \Sigma}, T_{\upi}^\dagger, \Delta_{\mathscr{F}^1}))} \otimes_{\varpi_{3,2}^*} \cR_3}{\mathrm{res}_p(\delta(M_3^{\dag}, \mathrm{tr}^*\Delta_d))} \right)\\
    &\qquad\qquad  \times\,\varpi^*_{3,2}\Char_{\cR_{\underline{\Pi}}}\left(\mathbf{R}^2\Gamma_{\fin}(G_{\mathbb{Q},\Sigma},T_{\upi}^\dagger, \Delta_{\mathscr{F}^1}) \right). 
\end{align*}
\end{lemma}
\begin{proof}
This is an immediate consequence of the exact sequence 
\begin{align*}
0\lra  &\, \frac{\dfrac{H^1(G_{p}, \mathscr{F}^1/\mathscr{F}^3)}{\res_p(\mathbf{R}^1\Gamma_{\fin}(G_{\mathbb Q, \Sigma}, T_{\upi}^\dagger, \Delta_{\mathscr{F}^1}))} \otimes_{\varpi_{3,2}^*} \cR_3}{\mathrm{res}_p(\delta(M_3^{\dag}, \mathrm{tr}^*\Delta_d))} \lra  \frac{\mathbf{R}^2\Gamma_{\fin}(G_{\mathbb Q, \Sigma}, T_{\upi}^\dagger, \Delta_{\mathscr{F}^3}) \otimes_{\varpi_{3,2}^*} \cR_3}{\delta^1(\delta(M_3^{\dag}, \mathrm{tr}^*\Delta_d))}\\
&\hspace{4cm} \lra \mathbf{R}^2\Gamma_{\fin}(G_{\mathbb{Q},\Sigma},T_{\upi}^\dagger, \Delta_{\mathscr{F}^1})\otimes_{\varpi_{3,2}^*} \cR_3 \lra H^2(G_{p}, \mathscr{F}^1/\mathscr{F}^3\otimes_{\varpi_{3,2}^*} \cR_3)=0
\end{align*}
obtained from the diagram in Lemma~\ref{lemma_5_2_2025_02_14_1214}.
\end{proof}

\begin{lemma}
\label{lemma_5_5_2025_02_14_1331} 
Suppose that $\delta(T_3^{\dag},\Delta_d) \neq 0 =\delta(T^\dagger_{\underline{\Pi}}, \Delta_{\mathscr{F}^2})$. Then:
$$\mathbf{R}^1\Gamma_{\fin}(G_{\mathbb Q, \Sigma}, T_{\upi}^\dagger, \Delta_{\mathscr{F}^1}) = \mathbf{R}^1\Gamma_{\fin}(G_{\mathbb Q, \Sigma}, T_{\upi}^\dagger, \Delta_{\mathscr{F}^2})\,.$$ 
\end{lemma}
\begin{proof}
Recall that $\mathbf{R}^1\Gamma_{\fin}(G_{\mathbb Q, \Sigma}, T_{\upi}^\dagger, \Delta_{\mathscr{F}^3}) = 0$ under our running hypotheses.
It follows from global duality combined with the fact that $\chi(\mathbf{R}\Gamma_{\fin}(G_{\mathbb Q, \Sigma}, T_{\upi}^\dagger, \Delta_{\mathscr{F}^1}))=1$ and the vanishing above that $$\mathrm{rank}_{\cR_{\underline{\Pi}}}  \mathbf{R}^1\Gamma_{\fin}(G_{\mathbb Q, \Sigma}, T_{\upi}^\dagger, \Delta_{\mathscr{F}^1}) = 1.$$ 
This shows that $\mathrm{rank}_{\cR_{\underline{\Pi}}}  \mathbf{R}^1\Gamma_{\fin}(G_{\mathbb Q, \Sigma}, T_{\upi}^\dagger, \Delta_{\mathscr{F}^2}) \leq 1$. Moreover, our running assumption that $\delta(T^\dagger_{\underline{\Pi}}, \Delta_{\mathscr{F}^2})=0$ implies that $\mathbf{R}^1\Gamma_{\fin}(G_{\mathbb Q, \Sigma}, T_{\upi}^\dagger, \Delta_{\mathscr{F}^2})$ has positive rank, which in turn shows that 
$$\mathrm{rank}_{\cR_{\underline{\Pi}}}  \mathbf{R}^1\Gamma_{\fin}(G_{\mathbb Q, \Sigma}, T_{\upi}^\dagger, \Delta_{\mathscr{F}^1}) = 1=\mathrm{rank}_{\cR_{\underline{\Pi}}}  \mathbf{R}^1\Gamma_{\fin}(G_{\mathbb Q, \Sigma}, T_{\upi}^\dagger, \Delta_{\mathscr{F}^2})\,.$$
Finally, the quotient $\mathbf{R}^1\Gamma_{\fin}(G_{\mathbb Q, \Sigma}, T_{\upi}^\dagger, \Delta_{\mathscr{F}^1})/\mathbf{R}^1\Gamma_{\fin}(G_{\mathbb Q, \Sigma}, T_{\upi}^\dagger, \Delta_{\mathscr{F}^2})$ is torsion-free since it is canonically isomorphic to a submodule of $H^1(G_{p}, \mathscr{F}^1/\mathscr{F}^2)$. Our lemma follows on combining these facts. 
\end{proof}

\begin{lemma}
\label{lemma_5_6_2025_02_14_1507}
Suppose that $\delta(T_3^{\dag},\Delta_d) \neq 0 =\delta(T^\dagger_{\underline{\Pi}}, \Delta_{\mathscr{F}^2})$. Then:
    \begin{align*}
\Char_{\cR_3}\left( \frac{\mathbf{R}^2\Gamma_{\fin}(G_{\mathbb Q, \Sigma}, T_{\upi}^\dagger, \Delta_{\mathscr{F}^3}) \otimes_{\varpi_{3,2}^*} \cR_3}{\delta^1(\delta(M_3^{\dag}, \mathrm{tr}^*\Delta_d))} \right)
        = \varpi^*_{3,2}\Log_{c}(\delta(T_{\upi}^\dagger, \Delta_{\mathscr{F}^1})) \Log_{d/e}(\delta(M_3^\dagger, \mathrm{tr}^* \Delta_{d})). 
    \end{align*}
\end{lemma}
\begin{proof}
Since $\mathbf{R}^1\Gamma_{\fin}(G_{\mathbb Q, \Sigma}, T_{\upi}^\dagger, \Delta_{\mathscr{F}^1}) = \mathbf{R}^1\Gamma_{\fin}(G_{\mathbb Q, \Sigma}, T_{\upi}^\dagger, \Delta_{\mathscr{F}^2})$ by the previous lemma,  we have an exact sequence 
    \begin{align}
        \label{eqn_2025_02_14_1459}
        \begin{aligned}
        0 \ra 
     \dfrac{H^1(G_{p}, \mathscr{F}^2/\mathscr{F}^3)}{\res_p(\mathbf{R}^1\Gamma_{\fin}(G_{\mathbb Q, \Sigma}, T_{\upi}^\dagger, \Delta_{\mathscr{F}^1}))} 
    \longrightarrow 
    \dfrac{H^1(G_{p}, \mathscr{F}^1/\mathscr{F}^3)}{\res_p(\mathbf{R}^1\Gamma_{\fin}(G_{\mathbb Q, \Sigma}, T_{\upi}^\dagger, \Delta_{\mathscr{F}^1}))} 
    \longrightarrow 
    H^1(G_{p}, \mathscr{F}^1/\mathscr{F}^2)
    \ra 0. 
            \end{aligned}
    \end{align}
Relying on the fact that the left-most term in \eqref{eqn_2025_02_14_1459} is torsion (so that the intersection of $\res_p(\delta(M_3^{\dag}, \mathrm{tr}^*\Delta_d))$ with this module is trivial), we therefore have 
    \begin{align}
    \label{eqn_2025_02_14_1500}
        \begin{aligned}
\Char_{\cR_3}&\left( \frac{\dfrac{H^1(G_{p}, \mathscr{F}^1/\mathscr{F}^3)}{\res_p(\mathbf{R}^1\Gamma_{\fin}(G_{\mathbb Q, \Sigma}, T_{\upi}^\dagger, \Delta_{\mathscr{F}^1}))} \otimes_{\varpi_{3,2}^*} \cR_3}{\mathrm{res}_p(\delta(M_3^{\dag}, \mathrm{tr}^*\Delta_d))} \right)\\
&\hspace{2.5cm} =
\varpi_{3,2}^*\Char_{\cR_{\underline{\Pi}}}\left( \dfrac{H^1(G_{p}, \mathscr{F}^2/\mathscr{F}^3)}{\res_p(\mathbf{R}^1\Gamma_{\fin}(G_{\mathbb Q, \Sigma}, T_{\upi}^\dagger, \Delta_{\mathscr{F}^1})}  \right) 
\times \Log_{d/e}(\delta(M_3^\dagger, \mathrm{tr}^* \Delta_{d})). 
        \end{aligned}
    \end{align}
Finally, by Theorem~\ref{thm_2023_03_10}(ii), we have
    $$ \Char_{\cR_{\underline{\Pi}}}\left(\mathbf{R}^2\Gamma_{\fin}(G_{\mathbb{Q},\Sigma},T_{\upi}^\dagger, \Delta_{\mathscr{F}^1}) \right)    
    =     \Char_{\cR_{\underline{\Pi}}}\left(\mathbf{R}^1\Gamma_{\fin}(G_{\mathbb{Q},\Sigma},T_{\upi}^\dagger, \Delta_{\mathscr{F}^1})/\delta(T_{\upi}^\dagger, \Delta_{\mathscr{F}^1}) \right) $$ 
Since $\res_p$ is injective, we conclude that 
    \begin{align}
    \label{eqn_2025_02_14_1457}
    \begin{aligned}
        \Char_{\cR_{\underline{\Pi}}}\left( \dfrac{H^1(G_{p}, \mathscr{F}^2/\mathscr{F}^3)}{\res_p(\mathbf{R}^1\Gamma_{\fin}(G_{\mathbb Q, \Sigma}, T_{\upi}^\dagger, \Delta_{\mathscr{F}^1})}  \right)  &\,\Char_{\cR_{\underline{\Pi}}}\left(\mathbf{R}^2\Gamma_{\fin}(G_{\mathbb{Q},\Sigma},T_{\upi}^\dagger, \Delta_{\mathscr{F}^1}) \right) \\
&\hspace{2cm}= \Char_{\cR_{\underline{\Pi}}}\left(\dfrac{H^1(G_{p}, \mathscr{F}^2/\mathscr{F}^3)}{\res_p(\delta(T_{\upi}^\dagger, \Delta_{\mathscr{F}^1}))}\right) 
\\
&\hspace{2cm} = \Log_{c}(\delta(T_{\upi}^\dagger, \Delta_{\mathscr{F}^1})). 
    \end{aligned}
    \end{align}
The proof of our lemma follows on combining \eqref{eqn_2025_02_14_1500} and \eqref{eqn_2025_02_14_1457} with Lemma~\ref{lemma_5_4_2025_02_14_1324}. 
\end{proof}

\begin{proof}[Proof of Theorem~\ref{thm_5_1_2025_02_14_1424}]
The asserted factorisation follows from Lemma~\ref{lemma_5_3_2025_02_14_1213}(iii) and Lemma~\ref{lemma_5_6_2025_02_14_1507}.    
\end{proof}

\subsection{Double Wall-crossing: (a) to (c)}
\label{subsec_5_2_2025_02_11}
As in the previous section, the wall-crossing principle serves as our signpost. We summarize the portion of this philosophy relevant to our setting before stating our main result in this section.  

The restriction of the (conjectural) $p$-adic $L$-function for the region (a) to ${\rm im}({\rm Spec}\,\cR_3\xrightarrow{\iota_{3,4}}{\rm Spec}\,\cR_4)$ has empty interpolation range. Moreover, ${\rm im}(\iota_{3,4})$ contains a dense set of classical points that fall within the region (b) (where the global root number equals $-1$). However, as a key difference with the case considered in \S\ref{subsec_5_1_2025_02_11}, the degree--12 motives associated with these specialisations that belong to the region (b) are no longer critical. The correct interpretation, therefore, is that the values of the (conjectural) $p$-adic $L$-function for the region (a) at classical points in the region (c) should be a $p$-adic avatar for the second-order central critical derivative of the complex $L$-series associated to degree--16 motives with weights in region (c). Note that the global sign for the degree--16 motives with weights in region (c) equals +1, whereas the same for the degree--12 and degree--4 motives is $-1$.  

Our main result in this subsection is the factorisation statement in Theorem~\ref{thm_main_double_wall_crossing}, which reflects this principle: The first factor on the right is a $p$-adic avatar of the central critical derivatives of the complex $L$-series associated with degree--12 motives with weights in the region (c), whereas the second factor is the same for the degree--4 motives.

\subsubsection{Hypotheses} 
\label{subsubsec_2024_07_15_1105}
Before we state this result, we record the hypotheses that we will rely on in due course:
\begin{itemize}
\item[\mylabel{item_AJ_Pi}{$\mathbf{AJ}_{p}^{\upi}$})] $\delta(T^\dagger_{\underline{\Pi}},\Delta_{\mathscr{F}^1}) \neq 0 $.
\item[\mylabel{item_parity_Pi}{$\mathbf{P}^-_{\upi}$})] $\delta(T^\dagger_{\underline{\Pi}}, \Delta_{\mathscr{F}^2})=0$.
\item[\mylabel{item_parity_M}{$\mathbf{P}^-_{M}$})] $\delta(M_3^\dagger, {\rm tr}^*\Delta_{c})=0$. 
\item[\mylabel{item_AJ_M}{$\mathbf{AJ}_{p}^{M}$})]$\delta(M_3^\dagger,\Delta_b) \neq 0 $.
\item[\mylabel{item_nVdeltaT3}{$\mathbf{nV}_{T}^a$})] $\delta(T_3^{\dag},\Delta_a) \neq 0$.
    \end{itemize}

\begin{lemma}
    \label{lemma_2025_02_19_2205_bis}
    If the parity conjecture for classical specialisations of $T_\upi^\dagger$ holds true, then so does \eqref{item_parity_Pi}. Consequently, $\delta(T_\upi^\dagger, \Delta_{\mathscr{F}^1})\in \mathbf{R}^1\Gamma_{\fin}(G_{\mathbb Q, \Sigma}, T_\upi^\dagger, \Delta_{\mathscr{F}^2})$. 
\end{lemma}

\begin{proof}
    Since the global root number of the classical specialisations of $T_\upi^\dagger$ is equal to $-1$, it follows that the rank of $\mathbf{R}^2\Gamma_{\fin}(G_{\mathbb Q, \Sigma}, T_\upi^\dagger, \Delta_{\mathscr{F}^2})$ is odd. In particular, it cannot be torsion, and the conclusion follows from Theorem~\ref{thm_2023_03_10}(i). The second conclusion follows from Theorem~\ref{thm_2025_02_19_2217}, which tells us that
    $$\Log_{b/c}\delta(T_\upi^\dagger, \Delta_{\mathscr{F}^1})=\delta(T_\upi^\dagger, \Delta_{\mathscr{F}^2})=0\,.$$
\end{proof}

\begin{lemma}
\label{lemma_5_9_2025_02_14_1513}
We have the following commutative diagram with exact rows:
\footnotesize
\begin{align}
\begin{aligned}
\label{eqn_2025_07_18_1313}
    \xymatrix{
\mathbf{R}\Gamma_{\fin}(G_{\mathbb Q, \Sigma}, T_{\upi}^\dagger, \Delta_{0}) \otimes_{\varpi_{3,2}^*} \cR_3 \ar[r]\ar[d]_{=}
&
\mathbf{R}\Gamma_{\fin}(G_{\mathbb Q, \Sigma}, T_3^\dagger, \Delta_{a}) \ar[r]^{{\rm tr}^*}\ar[d]^{\rm tr}   \ar@{}[dr] | {\refsymbolA}
&
\mathbf{R}\Gamma_{\fin}(G_{\mathbb Q, \Sigma}, M_3^\dagger, \mathrm{tr}^*\Delta_{a}) \ar[r]^-{+1}\ar[d]^{\res_p}
&
\\
   \mathbf{R}\Gamma_{\fin}(G_{\mathbb Q, \Sigma}, T_{\upi}^\dagger, \Delta_{0})\otimes_{\varpi_{3,2}^*} \cR_3  \ar[r]
    &
   \mathbf{R}\Gamma_{\fin}(G_{\mathbb Q, \Sigma}, T_{\upi}^\dagger, \Delta_{\emptyset}) \otimes_{\varpi_{3,2}^*} \cR_3  \ar[r]
    &
    \mathbf{R}\Gamma(G_{p}, T_{\upi}^\dagger) \otimes_{\varpi_{3,2}^*} \cR_3  \ar[r]^-{\partial_{\upi}}_-{+1}
    &. 
    }
    \end{aligned}
\end{align}
\normalsize
Here, we have set $\Delta_{\emptyset} := \Delta_{\mathscr{F}^0}$ and $\Delta_{0} := \Delta_{\mathscr{F}^4}$.
\end{lemma}
\begin{proof}
    This is clear.
\end{proof}

\begin{lemma}\label{lemma_2025_02_20_1200} 
        Assume that \eqref{item_nVdeltaT3} holds.
        \item[i)] $\mathbf{R}^1\Gamma_{\fin}(G_{\mathbb Q, \Sigma}, T^\dagger_{\underline{\Pi}}, \Delta_0) = 0 = \mathbf{R}^1\Gamma_{\fin}(G_{\mathbb Q, \Sigma}, T_3^\dagger, \Delta_{a})$. 
   \item[ii)]  Assume in addition that the parity conjecture for classical specialisations of $T_\upi^\dagger$ is valid. Then \eqref{item_AJ_Pi} is equivalent to the requirement\footnote{This requirement also follows from the non-vanishing of a $p$-adic Abel--Jacobi map, hence our choice of notation.} that the morphism
        $$\mathbf{R}^1\Gamma_{\fin}(G_{\mathbb Q, \Sigma}, T^\dagger_{\underline{\Pi}}, \Delta_{\mathscr F^2}) \lra H^1(G_p,\gr^2 T^\dagger_\upi)$$
be non-zero. When this is the case, $\mathbf{R}^1\Gamma_{\fin}(G_{\mathbb Q, \Sigma}, T^\dagger_{\underline{\Pi}}, \Delta_{\mathscr F^2})=\mathbf{R}^1\Gamma_{\fin}(G_{\mathbb Q, \Sigma}, T^\dagger_{\underline{\Pi}}, \Delta_{\mathscr F^1})$.
\end{lemma}

\begin{proof}
  Since $\delta(T_3^{\dag},\Delta_a) \neq 0$, we have 
    $\mathbf{R}^1\Gamma_{\fin}(G_{\mathbb Q, \Sigma}, T_3^\dagger, \Delta_{a}) = 0\,, $  and the vanishing $\mathbf{R}^1\Gamma_{\fin}(G_{\mathbb Q, \Sigma}, T^\dagger_{\underline{\Pi}}, \Delta_{0}) = 0$ follows from the first row of the diagram in Lemma~\ref{lemma_5_9_2025_02_14_1513}. This concludes the proof of (i). 

    It follows from the first part and global duality that $\mathbf{R}^1\Gamma_{\fin}(G_{\mathbb Q, \Sigma}, T^\dagger_{\underline{\Pi}}, \Delta_\emptyset)$ is of rank $2$. Hence, under the parity conjecture for classical specialisations of $T_\upi^\dagger$, the $\cR_\upi$-module $\mathbf{R}^1\Gamma_{\fin}(G_{\mathbb Q, \Sigma}, T^\dagger_{\underline{\Pi}}, \Delta_{\mathscr F^2})$ is of rank $1$. In that case, the non-vanishing statement \eqref{item_AJ_Pi} is equivalent to the requirement that the morphism
    \begin{equation}
        \label{eqn_2025_02_17_1057}
        \delta(T^\dagger_{\underline{\Pi}},\Delta_{\mathscr{F}^1})\subset \mathbf{R}^1\Gamma_{\fin}(G_{\mathbb Q, \Sigma}, T^\dagger_{\underline{\Pi}}, \Delta_{\mathscr F^2}) \lra H^1(G_p,\gr^2 T^\dagger_\upi)
    \end{equation}
 be injective (as this is equivalent to the vanishing of $\mathbf{R}^1\Gamma_{\fin}(G_{\mathbb Q, \Sigma}, T^\dagger_{\underline{\Pi}}, \Delta_{\mathscr F^3})$, which is the same as the final asserted equality in our lemma by global duality), or equivalently, the map \eqref{eqn_2025_02_17_1057} be non-zero. 
\end{proof}

\begin{lemma}
    \label{lemma_2025_02_19_2205}
    Assume that the parity conjecture for classical specialisations of $M_3^\dagger$ in the region (c) holds true. Then $\delta(M_3^\dagger, {\rm tr}^*\Delta_{c})=0$. Consequently, $\delta(M_3^\dagger, {\rm tr}^*\Delta_{b})\in \mathbf{R}^1\Gamma_{\fin}(G_{\mathbb Q, \Sigma}, M_3^\dagger, {\rm tr}^*\Delta_c)$. 
\end{lemma}

\begin{proof}
The proof of this lemma proceeds in a manner identical to that of Lemma~\ref{lemma_2025_02_19_2205_bis}, where we once again rely on Theorem~\ref{thm_2023_03_10}(i) and Theorem~\ref{thm_2025_02_19_2217}, noting that the latter tells us that $$\Log_{b/c}\delta(M_3^\dagger, {\rm tr}^*\Delta_{b})=\delta(M_3^\dagger, {\rm tr}^*\Delta_{c})=0\,.$$
\end{proof}

\begin{definition}
   Let us put 
   \begin{align*}
   F_b^\perp M_3^\dagger &:=\ker\left(F_c^+M_3^\dagger\ra \frac{F_c^+M_3^\dagger}{T_{\upi}^\dagger\otimes {\rm tr}^*(\mathscr{F}^1T_{\underline{\sigma}}\otimes \mathscr{F}^1T_{\underline{\sigma}}^*)}\to \gr^2 T_\upi^\dagger\otimes \gr^1 {\rm ad}^0(T_{\underline{\sigma}})\right)\\
   &\,\,=T_\upi^\dagger\otimes \mathscr{F}^2 {\rm ad}^0(T_g)+\mathscr{F}^3 T_\upi^\dagger\otimes \mathscr{F}^1 {\rm ad}^0(T_g),
   \end{align*}
${}$\hfill$\blacksquare$ 
\end{definition}

A direct calculation shows that under the symplectic duality induced from the perfect pairing $M_3^\dagger \otimes M_3^\dagger \to \cR_3(1)$, the module $F_b^+M_3^\dagger$ is indeed the orthogonal complement of $F_b^\perp M_3^\dagger$. We denote the associated local conditions on $M_3^\dagger$ by ${\rm tr}^*\Delta_b^\perp$.

\begin{lemma}
        \label{lemma_2025_02_20_1200_bis}
        If the parity conjecture for classical specialisations of $M_3^\dagger$ holds true and \eqref{item_nVdeltaT3} is valid, then \eqref{item_AJ_M} is equivalent to the requirement\footnote{This requirement also follows from the non-vanishing of a $p$-adic Abel--Jacobi map, hence our choice of notation.} that the morphism
        $$\mathbf{R}^1\Gamma_{\fin}(G_{\mathbb Q, \Sigma}, M_3^\dagger, {\rm tr}^*\Delta_{b}) \lra H^1(\QQ_p,\gr^2 T^\dagger_\upi)$$
be non-zero. When this condition also holds, $\mathbf{R}^1\Gamma_{\fin}(G_{\mathbb Q, \Sigma}, M_3^\dagger, {\rm tr}^*\Delta_{c})=\mathbf{R}^1\Gamma_{\fin}(G_{\mathbb Q, \Sigma}, M_3^\dagger, {\rm tr}^*\Delta_{b})$.
\end{lemma}

\begin{proof}
    We begin by noting that \eqref{item_nVdeltaT3} (which implies that the rank of $\mathbf{R}^1\Gamma_{\fin}(G_{\mathbb Q, \Sigma}, M_3^\dagger, {\rm tr}^*\Delta_a)$ is $2$), together with Lemma~\ref{lemma_2025_02_19_2205}, shows that  $\mathbf{R}^1\Gamma_{\fin}(G_{\mathbb Q, \Sigma}, M_3^\dagger, {\rm tr}^*\Delta_c)$ is of rank one. Hence, the morphism 
 $$\delta(M_3^\dagger,\Delta_b)\subset \mathbf{R}^1\Gamma_{\fin}(G_{\mathbb Q, \Sigma}, M_3^\dagger, {\rm tr}^*\Delta_c)\lra H^1(\QQ_p,\gr^2T_\upi^\dagger\otimes \gr^1 {\rm ad}^0(T_g))\xrightarrow[\sim]{\Log_c} \cR_3$$
 is injective (where the containment is explained in Lemma~\ref{lemma_2025_02_19_2205}) if and only if it is non-zero. The kernel of the first arrow, which is isomorphic to $\mathbf{R}^1\Gamma_{\fin}(G_{\mathbb Q, \Sigma}, M_3^\dagger, {\rm tr}^*\Delta_b^\perp)$, vanishes if and only if \eqref{item_AJ_M} holds (by Theorem~\ref{thm_2023_03_10}(ii) and global duality for Selmer complexes), in which case we also have $\mathbf{R}^1\Gamma_{\fin}(G_{\mathbb Q, \Sigma}, M_3^\dagger, {\rm tr}^*\Delta_b)=\mathbf{R}^1\Gamma_{\fin}(G_{\mathbb Q, \Sigma}, M_3^\dagger, {\rm tr}^*\Delta_c)$. 
\end{proof}
In summary, in the setting of Lemma~\ref{lemma_2025_02_20_1200}, we have
\begin{equation}
 \label{eqn_2025_07_15_1116}
     \hbox{\eqref{item_AJ_Pi}}\,\iff \, \Log_c \delta(T_\upi^\dagger,\Delta_{\mathscr{F}^1})\neq 0\,\iff\, \,\Log_c \hbox{ is non-zero on } \mathbf{R}^1\Gamma_{\fin}(G_{\mathbb Q, \Sigma}, T^\dagger_\upi, \Delta_{\mathscr{F}^2})\,.  
 \end{equation}
Likewise, in the setting of Lemma~\ref{lemma_2025_02_20_1200_bis},
 \begin{equation}
 \label{eqn_2025_07_15_1115}
     \hbox{\eqref{item_AJ_M}}\,\iff \, \Log_c \delta(M_3^\dagger,{\rm tr}^*\Delta_b)\neq 0\,\iff \,\Log_c \hbox{ is non-zero on } \mathbf{R}^1\Gamma_{\fin}(G_{\mathbb Q, \Sigma}, M_3^\dagger, {\rm tr}^*\Delta_c)\,.  
 \end{equation}
 This explains our notation: \eqref{item_AJ_Pi} and \eqref{item_AJ_M} follow from the generic non-triviality of the relevant $p$-adic Abel--Jacobi maps for the classical members of the underlying families of motives.

\subsubsection{} 
\label{subsubsec_2025_07_18_1307}
We are now ready to state one of our main results. This is Theorem~\ref{thm_main_double_wall_crossing} below; see also Theorem~\ref{thm_2025_03_11} for an important intermediate step: the factorisation of $\delta(T_3^{\dag}, \Delta_a)$, echoing the double-wall-crossing principle (from the region $a$ to region $c$). We recall that we work under Hypotheses~\ref{subsubsec_2024_07_15_1105} until the end of \S\ref{subsec_5_2_2025_02_11}, unless we explicitly state otherwise.

As the final preparatory step, we set 
$$S(M,\upi):=d_\upi^{-2}d_M^{-2}(c_Md_\upi-c_\upi d_M)^2,$$ 
where the quantities $c_M, d_M, c_\upi, d_\upi$ are defined in \S\ref{subsubsec_2025_07_18_1211}--\S\ref{subsubsec_2025_07_18_1212}. We roughly note that $S(M,\upi)$ measures the relative positions of the Selmer groups for the families $M_3^\dagger$ and $T_\upi^\dagger$ (where the modules of leading terms associated to these live, as well as the relevant Abel--Jacobi images of the conjectural cycles) with respect to the splittings determined by ${\rm tr}$ versus ${\rm tr}^*$; cf. the cartesian square {\refsymbolA} in \eqref{eqn_2025_07_18_1313} and {\refsymbolB} in \eqref{eqn_2025_03_10_1518}.  We finally remark that such a fudge factor is relevant to the double-wall-crossing scenario as opposed to the single-wall-crossing, as a reflection of the fact that (conjectural) cycles associated with both factors $M_3^\dagger$ and $T_\upi^\dagger$ are required to explain the arithmetic meaning of $L_p^{(a)}{}_{\vert_{(c)}}$ (hence also a comparison of the ambient spaces where these live in), whereas for $L_p^{(d)}{}_{\vert_{(e)}}$ (single-wall-crossing scenario), only the (conjectural) family of cycles associated with $T_\upi^\dagger$ are required for the same purpose. 

\begin{theorem}
\label{thm_main_double_wall_crossing}
Assume that Hypotheses~\ref{subsubsec_2024_07_15_1105} hold. Then,
\begin{align*}
    \delta(T_3^{\dag}, \Delta_a) &= S(M,\upi)\cdot {\rm Log}_{c}\wedge {\rm Log}_{a/b}(\delta(M_3^{\dag}, \tr^* \Delta_a)) \cdot \varpi_{3,2}^* \,\circ\,\Log_{c}\wedge \Log_{a/b}(\delta(T^\dagger_\upi, \Delta_{\emptyset}))\\
    &= S(M,\upi)\cdot {\rm Log}_{c} \delta(M_3^{\dag}, \tr^* \Delta_b)\cdot \varpi_{3,2}^* \,\circ\,\Log_{c}(\delta(T^\dagger_\upi, \Delta_{\mathscr F^1}))\,.
\end{align*}
\end{theorem}

The proof of Theorem~\ref{thm_main_double_wall_crossing} will occupy the remainder of Section~\ref{subsec_5_2_2025_02_11}. 

\begin{lemma}
\label{lemma_5_8_2025_02_14_1642}
Let $S$ be a regular ring, and let $X_1$ and $X_2$ be torsion-free $S$-modules of rank $r$. Suppose that $f \colon X_1 \to X_2$ be an injective $S$-homomorphism. Then for any $\delta \in {\bigcap}^r_S\,X_1$, we have 
$$ 
\Char_S\left(
\left. {\bigcap}^r_S\,X_2 \middle/ Sf^{(r)}(\delta) \right. \right) = 
\Char_S\left(\left.{\bigcap}^r_S\,X_1\middle/S\delta \right.\right)
\Char_S\left(\coker(f) \right).
 $$ 
Here $f^{(r)} \colon {\bigcap}^r_S\,X_1 \longrightarrow {\bigcap}^r_S\,X_2$
denotes the $S$-homomorphism induced by $f$. 
\end{lemma}
\begin{proof}
    On localizing the ring $S$ at a height-one prime, we may assume that $S$ is a discrete valuation ring. Then $X_1$ and $X_2$ are free $S$-module of rank $r$, and hence we may also assume that $X_1 = X_2 = S^r$.   In this case, this lemma follows from the fact that ${\bigcap}^r_S\, S^r = S$ combined with the fact that $f^{(r)} \colon S \to S$ is given by $x \mapsto \det(f)x$. 
\end{proof}

\begin{definition}
    Given an integral domain $S$ and an $S$-module $X$, let us denote by $X_{\rm tor}\subset X$ the torsion-submodule of $X$ and denote by $X_{\rm tf}:=X/X_{\rm tor}$ its maximal torsion-free quotient. ${}$ \hfill$\blacksquare$
\end{definition}

\begin{lemma}
\label{lemma_5_10_2025_02_14_1513}
   Suppose that \eqref{item_nVdeltaT3} holds true.  
   \item[i)] 
 The following diagram, where all the entries are $\cR_3$-modules of rank two, commutes:
       $$ 
       \xymatrix{
      \mathbf{R}^1\Gamma_{\fin}(G_{\mathbb Q, \Sigma}, M_3^\dagger, \mathrm{tr}^*\Delta_{a}) \ar@{^{(}->}[r]^-{\delta} \ar@{^{(}->}[d]_-{\res_p} & 
      \mathbf{R}^2\Gamma_{\fin}(G_{\mathbb Q, \Sigma}, T_{\upi}^\dagger, \Delta_{0}) \otimes_{\varpi_{3,2}^*} \cR_3. 
      \\
     \dfrac{{H}^1(G_{p}, T_{\upi}^\dagger)}{\res_p({\mathbf{R}^1\Gamma_{\fin}(G_{\mathbb Q, \Sigma}, T_{\upi}^\dagger, \Delta_{\emptyset}))}}\otimes_{\varpi_{3,2}^*} \cR_3 \ar@{^{(}->}[ru]_-{\partial_{\underline{\Pi}}} 
    }
     $$ 
    \item[ii)]   We have 
    \begin{align*}
       \begin{aligned}
            \delta(T_3^\dagger, \Delta_{a})=\Char_{\cR_3}\left(\frac{{\bigcap}^2_{\cR_\upi}\mathbf{R}^2\Gamma_{\fin}(G_{\mathbb Q, \Sigma}, T_{\upi}^\dagger, \Delta_{0})_{\rm tf}\otimes \cR_3}{\delta^{(2)}\left(\delta(M_3^\dagger,\mathrm{tr}^*\Delta_{a})\right)}\right)\varpi_{3,2}^*\Char_{\cR_\upi}\left(\mathbf{R}^2\Gamma_{\fin}(G_{\mathbb Q, \Sigma}, T_{\upi}^\dagger, \Delta_{0})_{\rm tor}\right).
       \end{aligned} 
   \end{align*}
\end{lemma}
\begin{proof}
   Lemma~\ref{lemma_5_9_2025_02_14_1513} combined with Lemma~\ref{lemma_2025_02_20_1200}(i) yields the following commutative diagram with exact rows: 
   \small
    \begin{align}
    \begin{aligned}
    \label{eqn_2025_03_10_1518}
                 \xymatrix{
      \mathbf{R}^1\Gamma_{\fin}(G_{\mathbb Q, \Sigma}, T_3^\dagger, \Delta_{a}) = 0 \ar[r]^{\rm tr^*} \ar[d]_{\rm tr}   \ar@{}[dr] | {\refsymbolB} &
      \mathbf{R}^1\Gamma_{\fin}(G_{\mathbb Q, \Sigma}, M_3^\dagger, \mathrm{tr}^*\Delta_{a}) \ar[r]^-{\delta} \ar[d]^{\res_p} & 
      \mathbf{R}^2\Gamma_{\fin}(G_{\mathbb Q, \Sigma}, T_{\upi}^\dagger, \Delta_{0}) \otimes_{\varpi_{3,2}^*} \cR_3 \ar[d]^{=}
      \\
      \mathbf{R}^1\Gamma_{\fin}(G_{\mathbb Q, \Sigma}, T_{\upi}^\dagger, \Delta_{\emptyset})\otimes_{\varpi_{3,2}^*} \cR_3  \ar[r]
    &
    \mathbf{R}^1\Gamma_{\fin}(G_{p}, T_{\upi}^\dagger) \otimes_{\varpi_{3,2}^*} \cR_3 \ar[r]_-{\partial_{\upi}}
    & \mathbf{R}^2\Gamma_{\fin}(G_{\mathbb Q, \Sigma}, T_{\upi}^\dagger, \Delta_{0}) \otimes_{\varpi_{3,2}^*} \cR_3. 
    }
        \end{aligned}
    \end{align} 
    \normalsize
    The proof of (i) follows immediately chasing this diagram. 
    
    We next prove (ii). In the exact sequence 
   \begin{align}
   \begin{aligned}
          \label{eqn_2025_02_14_1700}
       0\lra \,{\mathbf{R}^1\Gamma_{\fin}(G_{\mathbb Q, \Sigma}, M_3^\dagger, \mathrm{tr}^*\Delta_{a})} \xrightarrow{\,\delta\,}& \, {\mathbf{R}^2\Gamma_{\fin}(G_{\mathbb Q, \Sigma}, T_{\upi}^\dagger, \Delta_{0}) \otimes_{\varpi_{3,2}^*} \cR_3} \\
   &\quad \lra \ker\left(\mathbf{R}^2\Gamma_{\fin}(G_{\mathbb Q, \Sigma}, T_3^\dagger, \Delta_{a})\twoheadrightarrow  \mathbf{R}^2\Gamma_{\fin}(G_{\mathbb Q, \Sigma}, M_3^\dagger, \mathrm{tr}^*\Delta_{a})\right)\lra 0\,,
   \end{aligned}
   \end{align}
   the left-most module is torsion-free, and hence the induced map 
   $$\delta_{\rm tf} \colon \mathbf{R}^1\Gamma_{\fin}(G_{\mathbb Q, \Sigma}, M_3^\dagger, \mathrm{tr}^*\Delta_{a})\lra \mathbf{R}^2\Gamma_{\fin}(G_{\mathbb Q, \Sigma}, T_{\upi}^\dagger, \Delta_{0})_{\rm tf} \otimes_{\varpi_{3,2}^*} \cR_3$$
   is still injective. Moreover, both ${\rm coker}(\delta_{\rm tf})$ and ${\rm coker}(\delta)$ are $\cR_3$-torsion modules, and we have
   \begin{equation}
          \label{eqn_2025_02_14_1648}
       0\lra \mathbf{R}^2\Gamma_{\fin}(G_{\mathbb Q, \Sigma}, T_{\upi}^\dagger, \Delta_{0})_{\rm tor} \otimes_{\varpi_{3,2}^*} \cR_3 \lra {\rm coker}(\delta)\lra {\rm coker}(\delta_{\rm tf})\lra 0\,.
   \end{equation}
   
   It follows from Lemma~\ref{lemma_5_8_2025_02_14_1642} and \eqref{eqn_2025_02_14_1648} that 
   \begin{align}
       \label{eqn_2025_02_14_1642}
       \begin{aligned}
           \Char_{\cR_3}&\left(  {\rm coker}(\delta) \right)\, \Char_{\cR_3}\left(\frac{{\bigcap}^2_{\cR_\upi} \mathbf{R}^1\Gamma_{\fin}(G_{\mathbb Q, \Sigma}, M_3^\dagger, \mathrm{tr}^*\Delta_{a})\otimes \cR_3}{\delta(M_3^\dagger,\mathrm{tr}^*\Delta_{a})} \right)\\
           & = \Char_{\cR_3}\left(\frac{{\bigcap}^2_{\cR_\upi} \mathbf{R}^2\Gamma_{\fin}(G_{\mathbb Q, \Sigma}, T_{\upi}^\dagger, \Delta_{0})_{\rm tf}\otimes \cR_3}{\delta^{(2)}\left(\delta(M_3^\dagger,\mathrm{tr}^*\Delta_{a})\right)}\right)\varpi_{3,2}^*\Char_{\cR_\upi}\left(\mathbf{R}^2\Gamma_{\fin}(G_{\mathbb Q, \Sigma}, T_{\upi}^\dagger, \Delta_{0})_{\rm tor}\right).
       \end{aligned} 
   \end{align}
This, combined with \eqref{eqn_2025_02_14_1700} and Theorem~\ref{thm_2023_03_10}(ii) yields
\begin{align*}
       \begin{aligned}
           \delta(T_3^\dagger, \Delta_{a})=\Char_{\cR_3}\left(\frac{{\bigcap}^2_{\cR_\upi}\mathbf{R}^2\Gamma_{\fin}(G_{\mathbb Q, \Sigma}, T_{\upi}^\dagger, \Delta_{0})_{\rm tf}\otimes \cR_3}{\delta^{(2)}\left(\delta(M_3^\dagger,\mathrm{tr}^*\Delta_{a})\right)}\right)\varpi_{3,2}^*\Char_{\cR_\upi}\left(\mathbf{R}^2\Gamma_{\fin}(G_{\mathbb Q, \Sigma}, T_{\upi}^\dagger, \Delta_{0})_{\rm tor}\right)\,,
       \end{aligned} 
   \end{align*}
   and Part (iii) (which implicitly asserts that $\delta^{(2)}(\delta(M_3^\dagger,\mathrm{tr}^*\Delta_{a}))\subset {\bigcap}^2_{\cR_\upi} \mathbf{R}^2\Gamma_{\fin}(G_{\mathbb Q, \Sigma}, T_{\upi}^\dagger, \Delta_{0})\otimes \cR_3$) follows since we have ${\bigcap}^r_{\cR_\upi} X={\bigcap}^r_{\cR_\upi} X_{\rm tf}$ (by the definition of exterior bi-duals). 
\end{proof}
Let us set $Z:=\dfrac{H^1(G_{p}, T_{\upi}^\dagger)}{\res_p(\mathbf{R}^1\Gamma_{\fin}(G_{\mathbb Q, \Sigma}, T_{\upi}^\dagger, \Delta_{\emptyset}))}$ to ease notation. The map 
    $$\mathbf{R}^1\Gamma_{\fin}(G_{\mathbb Q, \Sigma}, M_3^\dagger, \mathrm{tr}^*\Delta_{a}) \xrightarrow{\res_p} Z\otimes_{\varpi_{3,2}^*}\cR_3$$ 
    (see Lemma~\ref{lemma_5_10_2025_02_14_1513}(ii)) is injective, the source is torsion-free, and both the source and target are of rank $2$. By Lemma~\ref{lemma_5_8_2025_02_14_1642}, $\res_p$ induces an injective map
\[
{\bigcap}^2_{\cR_3}\,\mathbf{R}^1\Gamma_{\fin}(G_{\mathbb Q, \Sigma}, M_3^\dagger, \mathrm{tr}^*\Delta_{a}) \xrightarrow{\res_p^{(2)}} {\bigcap}^2_{\cR_\upi} Z_{\rm tf}\otimes_{\varpi_{3,2}^*}\cR_3= {\bigcap}^2_{\cR_\upi} Z\otimes_{\varpi_{3,2}^*}\cR_3\,,
\]
allowing us to define the submodule $\res_p^{(2)}\left(\delta( M_3^\dagger, \mathrm{tr}^*\Delta_{a})\right)$ of ${\bigcap}^2_{\cR_\upi} Z\otimes_{\varpi_{3,2}^*}\cR_3$.

\begin{lemma}
\label{lemma_5_12_2025_02_14_1944}
Suppose that \eqref{item_nVdeltaT3} holds. Then:
\begin{align*}
          \begin{aligned}
              \Char_{\cR_3}\,&\left(\dfrac{{\bigcap}^2_{\cR_\upi} Z \otimes \cR_3}{\res_p^{(2)}(\delta(M_3^\dagger, {\rm tr}^*\Delta_a))}\right) \varpi_{3,2}^*\Char_{\cR_{\upi}}\left(\mathbf{R}^2\Gamma_{\fin}(G_{\mathbb{Q},\Sigma},T_{\upi}^\dagger, \Delta_{\emptyset}) \right)\,\Char_{\cR_{\upi}}\left(Z_{\rm tor}\right)
              \\
              &\, =\Char_{\cR_3}\left(\frac{{\bigcap}^2_{\cR_\upi} \mathbf{R}^2\Gamma_{\fin}(G_{\mathbb Q, \Sigma}, T_{\upi}^\dagger, \Delta_{0})\otimes \cR_3}{\delta^{(2)}(\delta(M_3^\dagger, {\rm tr}^*\Delta_a))}\right) \,\varpi_{3,2}^*\Char_{\cR_{\upi}}\left(\mathbf{R}^2\Gamma_{\fin}(G_{\mathbb Q, \Sigma}, T_{\upi}^\dagger, \Delta_{0})_{\rm tor} \right).
          \end{aligned}
   \end{align*}
\end{lemma}
\begin{proof}
Consider the exact sequence 
\begin{align}
\label{eqn_2025_02_14_2045}
0\lra Z \xrightarrow{\,\,\partial\,\,} &\mathbf{R}^2\Gamma_{\fin}(G_{\mathbb Q, \Sigma}, T_{\upi}^\dagger, \Delta_{0}) \lra \mathbf{R}^2\Gamma_{\fin}(G_{\mathbb{Q},\Sigma},T_{\upi}^\dagger, \Delta_{\emptyset})\lra H^2(G_{p},  T_{\upi}^\dagger)=0\,.  
\end{align}
It induces an injective map $\partial_{\rm tf} \colon Z_{\rm tf} \longrightarrow \mathbf{R}^2\Gamma_{\fin}(G_{\mathbb Q, \Sigma}, T_{\upi}^\dagger, \Delta_{0})_{\rm tf}$ of $\cR_\upi$-modules of rank 2, and Lemma~\ref{lemma_5_8_2025_02_14_1642} gives rise to a map
\[
\partial^{(2)} \colon {\bigcap}^2_{\cR_\upi} Z = {\bigcap}^2_{\cR_\upi} Z_{\rm tf}\lra {\bigcap}^2_{\cR_\upi} \mathbf{R}^2\Gamma_{\fin}(G_{\mathbb Q, \Sigma}, T_{\upi}^\dagger, \Delta_{0})_{\rm tf} = {\bigcap}^2_{\cR_\upi} \mathbf{R}^2\Gamma_{\fin}(G_{\mathbb Q, \Sigma}, T_{\upi}^\dagger, \Delta_{0})\,,
\]
which has the following properties.

\item[a)] The diagram
\[
\xymatrix{
{\bigcap}^2_{\cR_3}\,\mathbf{R}^1\Gamma_{\fin}(G_{\mathbb Q, \Sigma}, M_3^\dagger, \mathrm{tr}^*\Delta_{a}) \ar[rd]_{\res_p^{(2)}}\ar[rr]^{\delta^{(2)}} && {\bigcap}^2_{\cR_\upi} \mathbf{R}^2\Gamma_{\fin}(G_{\mathbb Q, \Sigma}, T_{\upi}^\dagger, \Delta_{0})\otimes \cR_3\\
&{\bigcap}^2_{\cR_\upi} Z \otimes \cR_3 \ar[ur]_-{\partial^{(2)}}&
}
\]
commutes.
\item[b)] We have
\begin{align*}
    \begin{aligned}
         \Char_{\cR_3}& \left({\bigcap}^2_{\cR_\upi} Z \otimes \cR_3\,\big/\,\res_p^{(2)}(\delta(M_3^\dagger, {\rm tr}^*\Delta_a))\right)\varpi_{3,2}^*\Char_{\cR_{\upi}}\left({\rm coker}(\partial_{\rm tf})\right)  \\
         &\hspace{2cm}= \Char_{\cR_3}\left({\bigcap}^2_{\cR_{\upi}} \mathbf{R}^2\Gamma_{\fin}(G_{\mathbb Q, \Sigma}, T_{\upi}^\dagger, \Delta_{0})\otimes \cR_3\,\big/\,\partial^{(2)}\circ\res_p^{(2)}(\delta(M_3^\dagger, {\rm tr}^*\Delta_a))\right)\\
         &\hspace{2cm}=\Char_{\cR_3}\left({\bigcap}^2_{\cR_{\upi}} \mathbf{R}^2\Gamma_{\fin}(G_{\mathbb Q, \Sigma}, T_{\upi}^\dagger, \Delta_{0})\otimes \cR_3\,\big/\,\delta^{(2)}(\delta(M_3^\dagger, {\rm tr}^*\Delta_a))\right)\,,
    \end{aligned}
\end{align*}
where the second equality follows from property (a) above.

We note further, thanks to \eqref{eqn_2025_02_14_2045}, that the sequence
\begin{equation}
          \label{eqn_2025_02_14_2043}
       0\lra Z_{\rm tor}\lra \mathbf{R}^2\Gamma_{\fin}(G_{\mathbb Q, \Sigma}, T_{\upi}^\dagger, \Delta_{0})_{\rm tor} \lra {\rm coker}(\partial)\lra {\rm coker}(\partial_{\rm tf})\lra 0
   \end{equation}
is exact. Combining \eqref{eqn_2025_02_14_2043} with property (b), we conclude that
\begin{align*}
          \begin{aligned}
              \Char_{\cR_3}\,&\left(\dfrac{{\bigcap}^2_{\cR_{\upi}} Z \otimes \cR_3}{\res_p^{(2)}(\delta(M_3^\dagger, {\rm tr}^*\Delta_a))}\right) \varpi_{3,2}^*\Char_{\cR_{\upi}}\left(\mathbf{R}^2\Gamma_{\fin}(G_{\mathbb{Q},\Sigma},T_{\upi}^\dagger, \Delta_{\emptyset}) \right)\,\Char_{\cR_{\upi}}\left(Z_{\rm tor}\right)
              \\
              &\, =\Char_{\cR_3}\left(\frac{{\bigcap}^2_{\cR_{\upi}} \mathbf{R}^2\Gamma_{\fin}(G_{\mathbb Q, \Sigma}, T_{\upi}^\dagger, \Delta_{0})\otimes \cR_3}{\delta^{(2)}(\delta(M_3^\dagger, {\rm tr}^*\Delta_a))}\right) \,\varpi_{3,2}^*\Char_{\cR_{\upi}}\left(\mathbf{R}^2\Gamma_{\fin}(G_{\mathbb Q, \Sigma}, T_{\upi}^\dagger, \Delta_{0})_{\rm tor} \right).
          \end{aligned}
   \end{align*}
as required.
\end{proof}

\begin{corollary}
    \label{cor_2025_02_15_0829}
    In the setting of Lemma~\ref{lemma_5_12_2025_02_14_1944}, we have
    \begin{align*}
       \begin{aligned}
           \delta(T_3^\dagger, \Delta_{a})=
            \Char_{\cR_{3}}\,&\left(\dfrac{{\bigcap}^2_{\cR_{\upi}} Z \otimes \cR_3}{\res_p^{(2)}(\delta(M_3^\dagger, {\rm tr}^*\Delta_a))}\right) \varpi_{3,2}^*\Char_{\cR_{\upi}}\left(\mathbf{R}^2\Gamma_{\fin}(G_{\mathbb{Q},\Sigma},T_{\upi}^\dagger, \Delta_{\emptyset}) \right)\,\Char_{\cR_{\upi}}\left(Z_{\rm tor}\right).
       \end{aligned} 
   \end{align*}
\end{corollary}

The proof of the following lemma is identical to that of Lemma~\ref{lemma_5_5_2025_02_14_1331}.

\begin{lemma}
    \label{lemma_5_5_2025_02_14_1331_bis}
Assume that \eqref{item_nVdeltaT3}, \eqref{item_AJ_Pi} and \eqref{item_parity_Pi} hold. Then:
$$\mathbf{R}^1\Gamma_{\fin}(G_{\mathbb Q, \Sigma}, T_{\upi}^\dagger, \Delta_{\mathscr{F}^1}) = \mathbf{R}^1\Gamma_{\fin}(G_{\mathbb Q, \Sigma}, T_{\upi}^\dagger, \Delta_{\mathscr{F}^2})\,.$$  
\end{lemma}


\subsubsection{} In preparation for our main result, we will prove Lemma~\ref{lemma_2025_02_24_1519} below. Let us fix a height-one prime $\mathfrak{q}$ of $\cR_3$. For notational simplicity, we put 
\begin{align*}
    S:= \cR_{3, \mathfrak{p}}, 
 \quad   W_i:= H^1(\mathbb{Q}_p, \mathscr{F}^iT_\upi^\dagger) \otimes_{\cR_{\upi}} S, \quad
    N_? &:= \textrm{the image of } \mathbf{R}^1\Gamma_{\fin}(G_{\mathbb{Q},\Sigma}, M_3^\dagger, {\rm tr}^*\Delta_{?}) \otimes_{\cR_3} S \textrm{ in } W_0, 
    \\
    L_? &:= \textrm{the image of } \mathbf{R}^1\Gamma_{\fin}(G_{\mathbb{Q},\Sigma}, T^\dagger_{\underline{\Pi}}, \Delta_{\mathscr{F}^?}) \otimes_{\cR_{\upi}} S \textrm{ in } W_0. 
\end{align*}
Note that $Z \otimes_{\cR_{\upi}} S = W_0/L_0$ and $Z_{\rm tor} \otimes_{\cR_{\upi}} S =  (W_0/L_0)_{\rm tor}$. 

\begin{lemma}
\label{lemma_2025_02_24_1519} We have
\[
\Char_S(\bigwedge^2 (W_0/L_0)_{\rm tf}/\bigwedge^2 N_a) 
\Char_S((W_0/L_0)_{\rm tors}) = \Char_S(W_0/(L_0+N_a)). 
\]
\end{lemma}
\begin{proof}
One can choose a basis $\{e_1, e_2, e_3, e_4\}$ of $W_0$ such that $L_0 = Sae_1 + Sbe_2$. Let us put 
$$N_a = S(a_1e_1 + b_1e_2 + c_1e_3 + d_1e_4) + S(a_2e_1 + b_2e_2 + c_2e_3 + d_2e_4)\,.$$ 
Since $(W_0/L_0)_{\rm tors} = W_0/(Se_1 + Se_2)$, we have 
\[
\Char_S\left(\bigwedge^2 (W_0/L_0)_{\rm tf}\,\Big{/}\,\bigwedge^2 N_a\right) = S(c_1d_2 - c_2d_1). 
\]
Hence, 
\[
\Char_S\left(\bigwedge^2 (W_0/L_0)_{\rm tf}\big{/}\bigwedge^2 N_a\right)\cdot \Char_S\left((W_0/L_0)_{\rm tors}\right)  = S(ab(c_1d_2 - c_2d_1)). 
\]
The proof of our lemma follows since we also have
\begin{align*}
    \Char_S\left(W_0\Big{/}(L_0+N_a)\right) = S \det\begin{pmatrix}
        a&0&0&0
        \\
        0&b&0&0
        \\
        a_1&b_1&c_1&d_1
        \\
        a_2&b_2&c_2&d_2
    \end{pmatrix}
    = S(ab(c_1d_2 - c_2d_1)). 
\end{align*}
\end{proof}

\begin{lemma}
        \label{lemma_2025_03_10_1431}
Assume that \eqref{item_nVdeltaT3} holds. We have,
\begin{align}
\begin{aligned}
       \label{eqn_2025_02_17_1116}
       \delta(T_3^\dagger, \Delta_a) &= \Char_{\cR_3}\left(\dfrac{H^1(G_{p}, T_{\upi}^\dagger) \otimes \cR_3}{\res_p\,\mathbf{R}^1\Gamma_{\fin}(G_{\mathbb Q, \Sigma}, \varpi_{3,2}^*T_{\upi}^\dagger, \Delta_{\emptyset}) \,\oplus \, \res_p\,\mathbf{R}^1\Gamma_{\fin}(G_{\mathbb Q, \Sigma}, M_3^\dagger, \mathrm{tr}^*\Delta_{a})}\right) \\
        &\hspace{2cm}\times \Char_{\cR_3}\left(\frac{{\bigcap}^2_{\cR_3} \mathbf{R}^1\Gamma_{\fin}(G_{\mathbb{Q},\Sigma}, M_3^\dagger, {\rm tr}^*\Delta_{a}) }{\delta(M_3^\dagger, {\rm tr}^*\Delta_a)}\right)
        \times \Char_{\cR_{\upi}}\left(\frac{{\bigcap}^2_{\cR_{\upi}} \mathbf{R}^1\Gamma_{\fin}(G_{\mathbb{Q},\Sigma}, T^\dagger_{\underline{\Pi}}, \Delta_{\emptyset}) }{\delta(T^\dagger_{\underline{\Pi}}, \Delta_{\emptyset})}\right) \,.
               \end{aligned} 
\end{align}
\end{lemma}

\begin{proof}
The exact sequence 
$$0 \lra \frac{{\bigcap}^2_{\cR_3} \mathbf{R}^1\Gamma_{\fin}(G_{\mathbb{Q},\Sigma}, M_3^\dagger, {\rm tr}^*\Delta_{a}) }{\delta(M_3^\dagger, {\rm tr}^*\Delta_a)} \xrightarrow{\res_p^{(2)}} \dfrac{{\bigcap}^2_{\cR_{\upi}} Z \otimes \cR_3}{\res_p^{(2)}(\delta(M_3^\dagger, {\rm tr}^*\Delta_a))} \lra \frac{{\bigcap}^2_{\cR_{\upi}} Z \otimes \cR_3}{{\bigcap}^2_{\cR_{3}} \mathbf{R}^1\Gamma_{\fin}(G_{\mathbb{Q},\Sigma}, M_3^\dagger, {\rm tr}^*\Delta_{a}) }\lra 0 $$
combined with Theorem~\ref{thm_2023_03_10}(ii), Corollary~\ref{cor_2025_02_15_0829}, and Lemma~\ref{lemma_2025_02_24_1519} shows that the left side of \eqref{eqn_2025_02_17_1116} can be written as 
\begin{align}
    \label{eqn_2025_02_18_0536}
    \begin{aligned}
        \Char_{\cR_3}\left(\frac{{\bigcap}^2_{\cR_{\upi}} Z \otimes \cR_3}{{\bigcap}^2_{\cR_{3}} \mathbf{R}^1\Gamma_{\fin}(G_{\mathbb{Q},\Sigma}, M_3^\dagger, {\rm tr}^*\Delta_{a}) }\right)& \times \Char_{\cR_{\upi}}\left(Z_{\rm tor}\right)\\
        &\hspace{-2.5cm}\times \Char_{\cR_3}\left(\frac{{\bigcap}^2_{\cR_{3}} \mathbf{R}^1\Gamma_{\fin}(G_{\mathbb{Q},\Sigma}, M_3^\dagger, {\rm tr}^*\Delta_{a})}{\delta(M_3^\dagger, {\rm tr}^*\Delta_a)}\right) 
        \times \Char_{\cR_{\upi}}\left(\frac{{\bigcap}^2_{\cR_{\upi}} \mathbf{R}^1\Gamma_{\fin}(G_{\mathbb{Q},\Sigma}, T^\dagger_{\underline{\Pi}}, \Delta_{\emptyset}) }{\delta(T^\dagger_{\underline{\Pi}}, \Delta_{\emptyset})}\right)\\
        &\hspace{-5cm} = \Char_{\cR_3}\left(\dfrac{H^1(G_{p}, T_{\upi}^\dagger) \otimes \cR_3}{\res_p\,\mathbf{R}^1\Gamma_{\fin}(G_{\mathbb Q, \Sigma}, \varpi^*_{3,2}T_{\upi}^\dagger, \Delta_{\emptyset}) \,\oplus \, \res_p\,\mathbf{R}^1\Gamma_{\fin}(G_{\mathbb Q, \Sigma}, M_3^\dagger, \mathrm{tr}^*\Delta_{a})}\right) 
        \\
        &\hspace{-2.5cm}\times \Char_{\cR_3}\left(\frac{{\bigcap}^2_{\cR_{3}}\mathbf{R}^1\Gamma_{\fin}(G_{\mathbb{Q},\Sigma}, M_3^\dagger, {\rm tr}^*\Delta_{a}) }{\delta(M_3^\dagger, {\rm tr}^*\Delta_a)}\right)
        \times \Char_{\cR_{\upi}}\left(\frac{{\bigcap}^2_{\cR_{\upi}} \mathbf{R}^1\Gamma_{\fin}(G_{\mathbb{Q},\Sigma}, T^\dagger_{\underline{\Pi}}, \Delta_{\emptyset}) }{\delta(T^\dagger_{\underline{\Pi}}, \Delta_{\emptyset})}\right) \,.
    \end{aligned}
\end{align}
\end{proof}

\begin{lemma}
\label{lemma_2025_07_10_1431}
    Assume that \eqref{item_nVdeltaT3} holds. 
The following canonical homomorphism is an isomorphism: 
    \begin{align*}
            {\bigcap}^2_{\cR_{3}} &\,\mathbf{R}^1\Gamma_{\fin}(G_{\mathbb{Q},\Sigma}, M_3^\dagger, {\rm tr}^*\Delta_a) \,\otimes\,  
       {\bigcap}^2_{\cR_{3}} \mathbf{R}^1\Gamma_{\fin}(G_{\mathbb{Q},\Sigma}, \varpi^*_{3,2}T^\dagger_{\underline{\Pi}}, \Delta_{\emptyset})\\
&  \hspace{3cm}\xrightarrow{\,\,\sim\,\,}  
{\bigcap}^4_{\cR_{3}} \left( \mathbf{R}^1\Gamma_{\fin}(G_{\mathbb{Q},\Sigma}, M_3^\dagger, {\rm tr}^*\Delta_a)\, \oplus  \,
       \mathbf{R}^1\Gamma_{\fin}(G_{\mathbb{Q},\Sigma}, \varpi^*_{3,2}T^\dagger_{\underline{\Pi}}, \Delta_{\emptyset})\right). 
        \end{align*}
\end{lemma}
\begin{proof}
    This lemma follows from both $\cR_3$-modules $\mathbf{R}^1\Gamma_{\fin}(G_{\mathbb{Q},\Sigma}, M_3^\dagger, {\rm tr}^*\Delta_a)$ and $\mathbf{R}^1\Gamma_{\fin}(G_{\mathbb{Q},\Sigma}, \varpi^*_{3,2}T^\dagger_{\underline{\Pi}}, \Delta_{\emptyset})$ are of rank 2 under our running hypothesis. 
\end{proof}

\subsubsection{}\label{sec:delta_a_decomp}
The restriction map  
\begin{align*}
    \mathbf{R}^1\Gamma_{\fin}(G_{\mathbb{Q},\Sigma}, M_3^\dagger, {\rm tr}^*\Delta_a) \oplus  
       \mathbf{R}^1\Gamma_{\fin}(G_{\mathbb{Q},\Sigma}, \varpi^*_{3,2}T^\dagger_{\underline{\Pi}}, \Delta_{\emptyset}) &\xrightarrow{\mathrm{res}_p } 
       H^1(G_p, \varpi^*_{3,2}T^\dagger_{\underline{\Pi}})\\
       (x,y) &\longmapsto \mathrm{res}_p(x) + \mathrm{res}_p(y)
\end{align*}
where $\res_p(x)$ is defined using the middle vertical arrow in \eqref{eqn_2025_03_10_1518}, together with Lemma~\ref{lemma_2025_07_10_1431} induce a homomorphism 
    \begin{align*}
     \mathrm{Log}_{\mathscr{F}^0}^{(4)}  \,\colon\,    {\bigcap}^2_{\cR_{3}} &\,\mathbf{R}^1\Gamma_{\fin}(G_{\mathbb{Q},\Sigma}, M_3^\dagger, {\rm tr}^*\Delta_a)\,\otimes  \,
       {\bigcap}^2_{\cR_{3}} \mathbf{R}^1\Gamma_{\fin}(G_{\mathbb{Q},\Sigma}, \varpi^*_{3,2}T^\dagger_{\underline{\Pi}}, \Delta_{\emptyset})
       \\
(\hbox{Lemma}~\ref{lemma_2025_07_10_1431})\qquad   &\qquad  \xrightarrow{\,\sim\,}  \,
{\bigcap}^4_{\cR_{3}} \left( \mathbf{R}^1\Gamma_{\fin}(G_{\mathbb{Q},\Sigma}, M_3^\dagger, {\rm tr}^*\Delta_a) \oplus  
       \mathbf{R}^1\Gamma_{\fin}(G_{\mathbb{Q},\Sigma}, \varpi^*_{3,2}T^\dagger_{\underline{\Pi}}, \Delta_{\emptyset})\right)
       \\
       &\hspace{5.7cm}  \xrightarrow{\res_p^{\otimes 4}} {\bigcap}^4_{\cR_{3}}  H^1(G_p, \varpi^*_{3,2}T^\dagger_{\underline{\Pi}}) \cong \cR_3,
        \end{align*}
where the final isomorphism amounts to choosing a basis of the cyclic $\cR_\upi$-module ${\bigcap}^4_{\cR_{\upi}}  H^1(G_p, T^\dagger_{\underline{\Pi}})$.
\begin{proposition}\label{prop:delta_a-formula2}
$\delta(T_3^\dagger, \Delta_a) =  \mathrm{Log}_{\mathscr{F}^0}^{(4)}(\delta(M_3^\dagger, {\rm tr}^*\Delta_a)        \otimes \delta(T^\dagger_{\underline{\Pi}}, \Delta_{\emptyset}))$. 
\end{proposition}
\begin{proof}
If \eqref{item_nVdeltaT3} fails while both $\delta(M_3^\dagger, {\rm tr}^*\Delta_a)$ and $\delta(T^\dagger_{\underline{\Pi}}, \Delta_{\emptyset})$ are non-zero, then Lemma \ref{lemma_5_9_2025_02_14_1513} implies that the homomorphism $\mathrm{res}_p^{\otimes 4}$ vanishes.
Consequently, $\mathrm{Log}_{\mathscr{F}^0}^{(4)}\big(\delta(M_3^\dagger, {\rm tr}^*\Delta_a)\otimes \delta(T^\dagger_{\underline{\Pi}}, \Delta_{\emptyset})\big) = 0$.
Thus, we may assume without loss of generality that \eqref{item_nVdeltaT3} holds.

Both $\cR_3$-modules ${\bigcap}^2_{\cR_{3}} \mathbf{R}^1\Gamma_{\fin}(G_{\mathbb{Q},\Sigma}, M_3^\dagger, {\rm tr}^*\Delta_a)$ and ${\bigcap}^2_{\cR_{3}} \mathbf{R}^1\Gamma_{\fin}(G_{\mathbb{Q},\Sigma}, \varpi^*_{3,2}T^\dagger_{\underline{\Pi}}, \Delta_{\emptyset})$ are torsion-free of rank one under \eqref{item_nVdeltaT3}. Hence, 
\begin{align}
\begin{aligned}
\label{eqn_2025_03_10_1537}
     &\Char_{\cR_3}\left(\frac{{\bigcap}^2_{\cR_{3}} \mathbf{R}^1\Gamma_{\fin}(G_{\mathbb{Q},\Sigma}, M_3^\dagger, {\rm tr}^*\Delta_{a}) }{\delta(M_3^\dagger, {\rm tr}^*\Delta_a)}\right)
        \times \Char_{\cR_{\upi}}\left(\frac{{\bigcap}^2_{\cR_{\upi}} \mathbf{R}^1\Gamma_{\fin}(G_{\mathbb{Q},\Sigma}, T^\dagger_{\underline{\Pi}}, \Delta_{\emptyset}) }{\delta(T^\dagger_{\underline{\Pi}}, \Delta_{\emptyset})}\right) 
        \\
        &\qquad\qquad = \Char_{\cR_3}\left(\frac{{\bigcap}^2_{\cR_{3}} \mathbf{R}^1\Gamma_{\fin}(G_{\mathbb{Q},\Sigma}, M_3^\dagger, {\rm tr}^*\Delta_{a}) \otimes {\bigcap}^2_{\cR_{3}} \mathbf{R}^1\Gamma_{\fin}(G_{\mathbb{Q},\Sigma}, \varpi^*_{3,2}T^\dagger_{\underline{\Pi}}, \Delta_{\emptyset}) }{\delta(M_3^\dagger, {\rm tr}^*\Delta_a) \otimes \delta(T^\dagger_{\underline{\Pi}}, \Delta_{\emptyset})} \right). 
\end{aligned}
\end{align}
Moreover, observe that we have 
\begin{align}
\begin{aligned}
\label{eqn_2025_03_10_1538}
\Char_{\cR_3}(\coker( \mathrm{Log}_{\mathscr{F}^0}^{(4)})) =   \Char_{\cR_3}\left(\dfrac{H^1(G_{p}, T_{\upi}^\dagger) \otimes \cR_3}{\res_p\,\mathbf{R}^1\Gamma_{\fin}(G_{\mathbb Q, \Sigma}, \varpi^*_{3,2}T_{\upi}^\dagger, \Delta_{\emptyset}) \,\oplus \, \res_p\,\mathbf{R}^1\Gamma_{\fin}(G_{\mathbb Q, \Sigma}, M_3^\dagger, \mathrm{tr}^*\Delta_{a})}\right)
\end{aligned}
\end{align}
The proof of our proposition follows on combining  \eqref{eqn_2025_02_17_1116}, \eqref{eqn_2025_03_10_1537} with \eqref{eqn_2025_03_10_1538} and Lemma~\ref{lemma_5_8_2025_02_14_1642}. 
\end{proof}

\begin{remark}
\label{remark_2025_03_10_1606}
In this remark, we briefly explain that an analogue \eqref{eqn_2025_03_10_1541} of Proposition \ref{prop:delta_a-formula2} holds in the situation of \S\ref{subsec_5_1_2025_02_11} as well, and Theorem~\ref{thm_5_1_2025_02_14_1424} can be deduced from that.

Suppose that $\delta(T_3^{\dag},\Delta_d) \neq 0 =\delta(T^\dagger_{\underline{\Pi}}, \Delta_{\mathscr{F}^2})$. By Lemma~\ref{lemma_5_3_2025_02_14_1213} and Lemma~\ref{lemma_5_4_2025_02_14_1324}, we have 
    \begin{align*}
    \delta(T_3^\dagger, \Delta_d)= & 
    \Char_{\cR_3}\left(\dfrac{H^1(G_{p}, \mathscr{F}^1/\mathscr{F}^3)\otimes \cR_3}{\res_p(\mathbf{R}^1\Gamma_{\fin}(G_{\mathbb Q, \Sigma}, \varpi^*_{3,2}T_{\upi}^\dagger, \Delta_{\mathscr{F}^1})) \oplus \res_p(\mathbf{R}^1\Gamma_{\fin}(G_{\mathbb Q, \Sigma}, M_{3}^\dagger, \mathrm{tr}^*\Delta_{d}))} \right)
    \\
    &\qquad \quad 
    \times \Char_{\cR_3}\left(\frac{\mathbf{R}^1\Gamma_{\fin}(G_{\mathbb Q, \Sigma}, M_{3}^\dagger, \mathrm{tr}^*\Delta_{d}))}{\delta(M_3^{\dag}, \mathrm{tr}^*\Delta_d)}\right) 
    \Char_{\cR_{\upi}}\left(\frac{\mathbf{R}^1\Gamma_{\fin}(G_{\mathbb{Q},\Sigma},T_{\upi}^\dagger, \Delta_{\mathscr{F}^1})}{\delta(T_{\underline{\Pi}}^\dagger, \Delta_{\mathscr{F}^1})} \right). 
\end{align*}
The same argument as in the proof of Proposition~\ref{prop:delta_a-formula2} shows that 
\begin{equation}
    \label{eqn_2025_03_10_1541}
    \delta(T_3^\dagger, \Delta_d) =  \mathrm{Log}_{\mathscr{F}^1/\mathscr{F}^3}^{(2)}(\delta(M_3^\dagger, {\rm tr}^*\Delta_d)   \otimes \delta(T^\dagger_{\underline{\Pi}}, \Delta_{\mathscr{F}^1})). 
\end{equation}
Here, $\mathrm{Log}_{\mathscr{F}^1/\mathscr{F}^3}^{(2)}$ is given as the compositum of the morphisms
\begin{align*} 
    \mathbf{R}^1\Gamma_{\fin}(G_{\mathbb Q, \Sigma}, M_{3}^\dagger, \mathrm{tr}^*\Delta_{d}) &\otimes
    \mathbf{R}^1\Gamma_{\fin}(G_{\mathbb Q, \Sigma}, \varpi^*_{3,2}T_{\upi}^\dagger, \Delta_{\mathscr{F}^1})
    \\
    &\xrightarrow{\,\sim\,} {\bigcap}^2_{\cR_3} \left(\mathbf{R}^1\Gamma_{\fin}(G_{\mathbb Q, \Sigma}, M_{3}^\dagger, \mathrm{tr}^*\Delta_{d}) \oplus
    \mathbf{R}^1\Gamma_{\fin}(G_{\mathbb Q, \Sigma}, \varpi^*_{3,2}T_{\upi}^\dagger, \Delta_{\mathscr{F}^1})\right)\\
    &\xrightarrow{\mathrm{res}_p^{(2)}} {\bigcap}^2_{\cR_3} H^1(G_p, \mathscr{F}^1/\mathscr{F}^3) \simeq \cR_3\,. 
\end{align*}
As before, let us denote by $\mathrm{Log}_{d/e}$ the composite map
\begin{align*}
\mathbf{R}^1\Gamma_{\fin}(G_{\mathbb Q, \Sigma}, M_{3}^\dagger, \mathrm{tr}^*\Delta_{d}) \oplus
    \mathbf{R}^1\Gamma_{\fin}(G_{\mathbb Q, \Sigma}, \varpi^*_{3,2}T_{\upi}^\dagger, \Delta_{\mathscr{F}^1}) \stackrel{\mathrm{res}_p}{\longrightarrow} H^1(G_p, \mathscr{F}^1) \otimes \cR_3 \lra H^1(G_p, \mathscr{F}^1/\mathscr{F}^2)\otimes \cR_3 \cong \cR_3
\end{align*}
and $\mathrm{Log}_{c}$ as the composite
\begin{align*}
         \ker(\mathrm{Log}_{d/e}) \stackrel{\mathrm{res}_p}{\longrightarrow} H^1(G_p, \mathscr{F}^2)\otimes \cR_3 \lra  H^1(G_p, \mathscr{F}^2/\mathscr{F}^3) \otimes \cR_3 \cong \cR_3\,.
\end{align*}
Then, we have by definition
\[
 \mathrm{Log}^{(2)}_{\mathscr{F}^1/\mathscr{F}^3} =  \mathrm{Log}_{d/e} \wedge  \mathrm{Log}_{c}. 
\]
Since $\mathrm{Log}_{d/e}(\delta(T_{\upi}^\dagger, \Delta_{\mathscr{F}^1})) = 0$ by Lemma \ref{lemma_5_5_2025_02_14_1331}, we obtain 
\begin{align*}
     \delta(T_3^\dagger, \Delta_d) &=  \mathrm{Log}_{\mathscr{F}^1/\mathscr{F}^3}^{(2)}(\delta(M_3^\dagger, {\rm tr}^*\Delta_d)   \otimes \delta(T^\dagger_{\underline{\Pi}}, \Delta_{\mathscr{F}^1}))
     \\
     &= (\mathrm{Log}_{d/e} \wedge  \mathrm{Log}_{c})(\delta(M_3^\dagger, {\rm tr}^*\Delta_d)   \otimes \delta(T^\dagger_{\underline{\Pi}}, \Delta_{\mathscr{F}^1}))
     \\
     &= \mathrm{Log}_{d/e}(\delta(M_3^\dagger, {\rm tr}^*\Delta_d) ) \,\cdot\,\mathrm{Log}_{c}(\delta(T^\dagger_{\underline{\Pi}}, \Delta_{\mathscr{F}^1}))\,,
\end{align*}
concluding this alternative proof of Theorem~\ref{thm_5_1_2025_02_14_1424}.
    \hfill$\blacksquare$
\end{remark}

\subsubsection{} 
Our proof of Theorem~\ref{thm_main_double_wall_crossing} will mimic the second proof of Theorem~\ref{thm_5_1_2025_02_14_1424} we have recorded in Remark~\ref{remark_2025_03_10_1606}. To execute this plan, we will show (cf. Lemma~\ref{lemma:quotient_by_saturation_is_free} below) that the saturations $H^1_\upi(G_{p}, T_{\upi}^\dagger)$ and $H^1_M(G_{p}, T_{\upi}^\dagger)$ of the isomorphic images of $\mathbf{R}^1\Gamma_{\fin}(G_{\mathbb Q, \Sigma}, \varpi_{3,2}^*T_{\upi}^\dagger, \Delta_{\emptyset})$ and  $\mathbf{R}^1\Gamma_{\fin}(G_{\mathbb Q, \Sigma}, M_3^\dagger, \mathrm{tr}^*\Delta_{a})$ inside $H^1(G_{p}, \varpi_{3,2}^*T_{\upi}^\dagger)$, respectively, have the following property: The quotient $H^1(G_{p}, \varpi_{3,2}^*T_{\upi}^\dagger)/H^1_?(G_{p}, T_{\upi}^\dagger)$ is a free $\cR_3$-module of rank two for each $?\in\{\upi,M\}$.

We will need the following general facts (recorded as Lemma~\ref{lemma:algebraic-lemmas-free} and Lemma~\ref{Lemma:ext^1=0_torsion-free_case} below) in our proof of this property.

\begin{lemma}\label{lemma:algebraic-lemmas-free}
A reflexive module of rank one over a regular local ring $S$  is free. 
\end{lemma}
\begin{proof}
Let $I$ be a reflexive ideal of $S$, and let $\fp$ be a height-one prime of $S$. Since regular local rings are unique factorisation domains (by a well-known theorem of Auslander--Buchsbaum), the height-one prime $\fp$ is principal. Let $x_\fp$ be a generator of $\fp$, and let $n_\fp$ be the non-negative integer such that $IS_\fp = \fp^{n_\fp}S_\fp$. We put $x := \prod_{\mathrm{ht}\,{\fp}=1}x_\fp^{n_\fp} \in S$. Then $IS_\fp = xS_\fp$ for any height-one prime $\fq$ of $S$.    Since $I$ is reflexive,  this implies $I = \bigcap_{\mathrm{ht}\,{\fp}=1}xS_\fp = xS$ (cf. \cite{sakamoto23}, Lemma C.11). 
\end{proof}

\begin{lemma}\label{Lemma:ext^1=0_torsion-free_case}
Let $S$ be a regular local ring. 
For any finitely generated torsion-free $S$-module $M$, we have 
\[
\mathrm{Ext}_S^1(M, R) = 0. 
\]
\end{lemma}
\begin{proof}
This is well-known. See, for example, \cite[\S9.1.3(v) and (vi)]{nekovar06}.
\end{proof}

\begin{lemma}\label{lemma:quotient_by_saturation_is_free}
Assume that Hypotheses~\ref{subsubsec_2024_07_15_1105} hold.
Then both $\cR_3$-modules 
$$
H^1(G_{p}, \varpi^*_{3,2}T_{\upi}^\dagger)/H^1_\upi(G_{p},T_{\upi}^\dagger) \quad \hbox{ and } \quad H^1(G_{p}, \varpi^*_{3,2}T_{\upi}^\dagger)/H^1_M(G_{p}, T_{\upi}^\dagger)$$ 
are free of rank 2. 
\end{lemma}
\begin{remark}
    In the situation of \S\ref{subsec_5_1_2025_02_11}, the saturation of $\mathrm{res}_p\,\mathbf{R}^1\Gamma_{\fin}(G_{\mathbb{Q},\Sigma}, T^\dagger_{\underline{\Pi}}, \Delta_{\mathscr{F}^1})$ in $H^1(G_p, \mathscr{F}^1/\mathscr{F}^3)$ 
    equals $H^1(G_p, \mathscr{F}^2/\mathscr{F}^3)$. 
    Therefore, the analogous quotient in this case is isomorphic to $H^1(G_p, \mathscr{F}^1/\mathscr{F}^2)$, which is free of rank one. 
    \hfill$\blacksquare$
\end{remark}
\begin{proof}[Proof of Lemma~\ref{lemma:quotient_by_saturation_is_free}]
Let us show that $H^1(G_{p}, \varpi^*_{3,2}T_{\upi}^\dagger)/H^1_\upi(G_{p}, T_{\upi}^\dagger)$ is free. For $i\in \{0,1,2,3\}$, we let $H^1_\upi(\Delta_{\mathscr{F}^i})$ denote  the saturation of the isomorphic image of $\mathbf{R}^1\Gamma_{\fin}(G_{\mathbb Q, \Sigma}, T_{\upi}^\dagger, \Delta_{\mathscr{F}^i})$ inside $H^1(G_{p}, T_{\upi}^\dagger)$. 
Then it suffices to show that  
\[
H^1(G_{p}, T_{\upi}^\dagger)/H^1_\upi(\Delta_{\mathscr{F}^0}) \textrm{ is free (over $\cR_\upi$)}. 
\]
Note that $H^1_\upi(\Delta_{\mathscr{F}^1}) = H^1_\upi(\Delta_{\mathscr{F}^2})$ under our running hypotheses (cf. Lemma~\ref{lemma_5_5_2025_02_14_1331_bis}), and it is of rank 1. 

Since $T_{\upi}^\dagger$ is self-dual, the cup-product pairing induces a perfect pairing on $H^1(G_{p}, T_{\upi}^\dagger)$. 
This pairing induces a perfect pairing
\begin{equation}
\label{eqn_2025_03_11_1232}
H^1(G_{p}, \mathscr{F}^0/\mathscr{F}^2) \times H^1(G_{p}, \mathscr{F}^2) \longrightarrow \cR_\upi. 
\end{equation}
To lighten our notation, let us denote by $L$ the image of $\mathbf{R}^1\Gamma_{\fin}(G_{\mathbb Q, \Sigma}, T_{\upi}^\dagger, \Delta_{\emptyset})$ inside $H^1(G_{p}, \mathscr{F}^0/\mathscr{F}^2)$, and write $L^{\rm sat}$ for its saturation. 
 We first prove that 
\begin{equation}
\label{eqn_2025_03_11_1231}
    L^{\rm sat}{}^\perp = L^\perp = H^1_\upi(\Delta_{\mathscr{F}^2}). 
\end{equation}
Note that the first equality in \eqref{eqn_2025_03_11_1231} is clear since the pairing \eqref{eqn_2025_03_11_1232} takes values in the integral domain $\cR_\upi$. Moreover, since both $L$ and $H^1_\upi(\Delta_{\mathscr{F}^2})$ are of rank 1 and $H^1(G_{p}, \mathscr{F}^2)/H^1_\upi(\Delta_{\mathscr{F}^2})$ is torsion-free, we have 
$$L^\perp = H^1_\upi(\Delta_{\mathscr{F}^2}) \,\Longleftrightarrow\, L^\perp \supset H^1_\upi(\Delta_{\mathscr{F}^2})\,\Longleftrightarrow\, L\,\perp \, H^1_\upi(\Delta_{\mathscr{F}^2})\,.$$ 
As a result, once again relying on the fact that the pairing \eqref{eqn_2025_03_11_1232} takes values in the integral domain $\cR_\upi$, the sought after equality \eqref{eqn_2025_03_11_1231} follows once we check that
\begin{equation}
\label{eqn_2025_03_11_1324}
L \otimes \mathcal{Q} \quad\perp\quad  H^1_\upi(\Delta_{\mathscr{F}^2}) \otimes \mathcal{Q}\,, 
\end{equation}
where $\mathcal{Q}$ is the field of fractions of $\cR_\upi$. 

As the complex $\mathbf{R}\Gamma(G_\ell, T^\dagger_{\upi}) \otimes \mathcal{Q}$ is acyclic for each prime $\ell \neq p$, it follows from Nekov\'a\v{r}'s duality \cite[Proposition 6.7.7]{nekovar06}  that the cup-product induces an isomorphism 
\[
\Hom_{\mathcal{Q}}(\mathbf{R}\Gamma_{\fin}(G_{\mathbb{Q},\Sigma},T_\upi^\dagger, \Delta_{\emptyset}) \otimes \mathcal{Q},  \mathcal{Q})[-3] \cong  
\mathbf{R}\Gamma_{\fin}(G_{\mathbb{Q},\Sigma},T_\upi^\dagger, \Delta_{0})  \otimes \mathcal{Q}\,. 
\]
We, therefore, have an exact triangle 
\begin{align*}
  \Hom_{\mathcal{Q}}(\mathbf{R}\Gamma_{\fin}(G_{\mathbb{Q},\Sigma},T_\upi^\dagger, \Delta_{\emptyset}) \otimes \mathcal{Q},  \mathcal{Q})[-3] \longrightarrow   
  \mathbf{R}\Gamma_{\fin}(G_{\mathbb{Q},\Sigma},T_\upi^\dagger, \Delta_{\mathscr{F}^2})  \otimes \mathcal{Q} 
  \longrightarrow \mathbf{R}\Gamma(G_{p},\mathscr{F}^2)  \otimes \mathcal{Q} \stackrel{+1}{\longrightarrow}\,, 
\end{align*}
which yields the exact sequence 
\[
H^1_\upi(\Delta_{\mathscr{F}^2}) \otimes \mathcal{Q}=\mathbf{R}\Gamma_{\fin}^1(G_{\mathbb{Q},\Sigma},T_\upi^\dagger, \mathscr{F}^2)  \otimes \mathcal{Q} 
  \longrightarrow H^1(G_{p},\mathscr{F}^2)  \otimes \mathcal{Q} \longrightarrow 
   \Hom_{\mathcal{Q}}(\mathbf{R}\Gamma_{\fin}^1(G_{\mathbb{Q},\Sigma},T_\upi^\dagger, \Delta_{\emptyset}) \otimes \mathcal{Q}, \mathcal{Q}). 
\]
This precisely means \eqref{eqn_2025_03_11_1324} holds, and the proof of our claim \eqref{eqn_2025_03_11_1231} is complete. 

Since $H^1(G_{p}, \mathscr{F}^0/\mathscr{F}^2)/L^{\rm sat}$ is torsion-free, Lemma \ref{Lemma:ext^1=0_torsion-free_case} shows that 
\[
\Hom_{\cR_3}(H^1(G_{p}, \mathscr{F}^0/\mathscr{F}^2), \cR_3) \longrightarrow \Hom_{\cR_\upi}(L^{\rm sat}, \cR_\upi)
\]
is surjective. We therefore obtain an $\cR_\upi$-isomorphism 
\[
H^1(G_{p}, \mathscr{F}^2)/H^1_\upi(\Delta_{\mathscr{F}^2}) \stackrel{\sim}{\longrightarrow} \Hom_{\cR_\upi}(L^{\rm sat}, \cR_\upi). 
\]
This fact shows that $H^1(G_{p}, \mathscr{F}^2)/H^1_\upi(\Delta_{\mathscr{F}^2})$ is reflexive, and hence it is free by Lemma \ref{lemma:algebraic-lemmas-free}; which, in turn, implies that the inclusion morphism $H^1_\upi(\Delta_{\mathscr{F}^2}) \to H^1(G_{p}, T_\upi^\dagger)$ admits a section. Therefore, the inclusion map $H^1_\upi(\Delta_{\mathscr{F}^2}) \to H^1_\upi(\Delta_{\mathscr{F}^0})$
also admits a section, and hence, we have an $\cR_\upi$-isomorphism 
\begin{align}\label{eq:upi-local-decomposition}
    H^1_\upi(\Delta_{\mathscr{F}^0}) \cong H^1_\upi(\Delta_{\mathscr{F}^2}) \oplus H^1_\upi(\Delta_{\mathscr{F}^0})/H^1_\upi(\Delta_{\mathscr{F}^2})\,. 
\end{align}
Since both $H^1_\upi(\Delta_{\mathscr{F}^0})$ and $H^1_\upi(\Delta_{\mathscr{F}^2})$ are reflexive, the quotient $H^1_\upi(\Delta_{\mathscr{F}^0})/H^1_\upi(\Delta_{\mathscr{F}^2})$ is also reflexive. We infer from Lemma~\ref{lemma:algebraic-lemmas-free} that 
$H^1_\upi(\Delta_{\mathscr{F}^0})/H^1_\upi(\Delta_{\mathscr{F}^2})$ is free. In particular, $H^1_\upi(\Delta_{\mathscr{F}^0})$ is free. 
Using again Nekov\'a\v{r}'s duality, we obtain  $H^1_\upi(\Delta_{\mathscr{F}^0})^\perp = H^1_\upi(\Delta_{\mathscr{F}^0})$, and hence an isomorphism  
\[
H^1(G_p, T_\upi^\dagger)/H^1_\upi(\Delta_{\mathscr{F}^0}) \stackrel{\sim}{\longrightarrow} \Hom_{\cR_\upi}(H^1_\upi(\Delta_{\mathscr{F}^0}), \cR_\upi). 
\]
This concludes our proof that $H^1(G_p, T_\upi^\dagger)/H^1_\upi(\Delta_{\mathscr{F}^0})$ is free, as required.

The same argument shows that $H^1(G_{p}, \varpi_{3,2}^*T_{\upi}^\dagger)/H^1_M(G_{p}, T_{\upi}^\dagger)$ is free, and our proof is complete. 
\end{proof}

\subsubsection{} By Lemma \ref{lemma:quotient_by_saturation_is_free}, there is a surjection of $\cR_3$-modules $H^1(G_p, \varpi^*_{3,2}T_\upi^\dagger) \xrightarrow{\ell_\upi} R_3^2$ such that $\ker(\ell_\upi) = H^1_\upi(G_p, T_\upi^\dagger)$. 
We put
\[
\psi_i := \mathrm{pr}_i \circ \ell_\upi \circ \mathrm{res}_p \quad (i=1,2)
\,\,\, \textrm{ and } \,\,\, 
\mathrm{Log}_{/\upi}^{(2)} := \psi_1 \wedge \psi_2. 
\]
Let $\psi_3, \psi_4 \in {\rm Hom}_{\cR_3}(H^1(G_p, \varpi_{3,2}^*T_\upi^\dagger), \cR_3)$ be such that
\[
\psi_3 \oplus \psi_4\oplus \ell_\upi \, \colon\, H^1(G_p, \varpi_{3,2}^*T_\upi^\dagger)\,\longrightarrow\, \cR_3^4
\]
is an isomorphism.

\begin{theorem}
    \label{thm_2025_03_11} 
    Assume that Hypotheses~\ref{subsubsec_2024_07_15_1105} hold. Then,
    $$ \delta(T_3^\dagger, \Delta_a)= \mathrm{Log}_{/\upi}^{(2)}(\delta(M_3^\dagger, {\rm tr}^*\Delta_a) )  \times 
     ( \psi_3 \circ\mathrm{res}_p  \wedge \psi_4\circ \mathrm{res}_p) ( \delta(T^\dagger_{\underline{\Pi}}, \Delta_{\emptyset}))\,.$$
\end{theorem}
\begin{proof}
   The asserted factorisation follows from the following chain of equalities:  
\begin{align*}
     \delta(T_3^\dagger, \Delta_a) &=  \mathrm{Log}_{\mathscr{F}^0}^{(4)}(\delta(M_3^\dagger, {\rm tr}^*\Delta_a)  \otimes \delta(T^\dagger_{\underline{\Pi}}, \Delta_{\emptyset}))
     \\
     &= (\psi_3 \circ \mathrm{res}_p \wedge \psi_4 \circ \mathrm{res}_p \wedge \mathrm{Log}_{/\upi}^{(2)})(\delta(M_3^\dagger, {\rm tr}^*\Delta_a)  \otimes \delta(T^\dagger_{\underline{\Pi}}, \Delta_{\emptyset}))
     \\
     &= \mathrm{Log}_{/\upi}^{(2)}(\delta(M_3^\dagger, {\rm tr}^*\Delta_a) )  \times 
     (\psi_3 \circ \mathrm{res}_p \wedge \psi_4 \circ \mathrm{res}_p ) ( \delta(T^\dagger_{\underline{\Pi}}, \Delta_{\emptyset})). 
\end{align*}
Here, the first equality is Proposition~\ref{prop:delta_a-formula2}, and the final equality holds thanks to definitions.
\end{proof}


\subsubsection{Basis of $H^1(G_p, T^\dagger_\upi)$} 
\label{subsubsec_2025_07_18_1211}
For each $i\in\{0,1,2,3\}$, let us choose basis $\{f_0,\cdots,f_i\} \subset H^1(G_p, \mathscr{F}^{3-i})$ so that
\[
\langle f_i, f_{3-j} \rangle = (-1)^i\delta_{ij}, 
\]
where $\delta_{ij}$ is the Kronecker delta function. 

\subsubsection{Basis of $H^1_\upi(G_p, T_\upi^\dagger)$}
By Lemma \ref{lemma:quotient_by_saturation_is_free}, the $\cR_3$-module  $H^1_\upi(G_p, T^\dagger_\upi)$ is free of rank 2. 
Moreover, the isomorphism \eqref{eq:upi-local-decomposition} shows that 
$H^1_\upi(G_p, T^\dagger_\upi)/(H^1_\upi(\Delta_{\mathscr{F}^2}) \otimes \cR_3)$ is free of rank 1.  One can therefore take a basis of $H^1_\upi(G_p, T_\upi^\dagger)$ of the form
\[
\{\varphi_2:=a_1f_0 + b_1f_1 + c_1f_2 + d_1f_3\,,\,\, \varphi_1:=c_2f_0+d_2f_1\}\,.
\] 
As we have seen in the proof of Lemma \ref{lemma:quotient_by_saturation_is_free}, 
the quotient $H^1(G_p, \mathscr{F}^2)/H^1_\upi(\Delta_{\mathscr{F}^2})$ is free, that is, 
$\cR_3^2/(c_2, d_2)\cR_3$ is free. 
As a result, since $\cR_3$ is a local ring, 
\begin{itemize}
    \item $c_2$ or $d_2$ is a unit. 
\end{itemize}
Moreover, since $H^1_\upi(G_p, T_\upi^\dagger) = H^1_\upi(G_p, T_\upi^\dagger)^\perp$, we have 
\[
0 = \langle \varphi_1, \varphi_2 \rangle = c_2d_1-d_2c_1. 
\]
Putting all this together, we conclude that there is an element $r \in \cR_3$ such that 
\[
 \begin{pmatrix} c_1\\d_1\end{pmatrix} = r \begin{pmatrix}c_2\\d_2\end{pmatrix}\,.
\]
We thus obtain the following result. 
\begin{lemma}\label{lemma:basis_of_H^1_upi(G_p, T_upi^dagger)}
Under Hypotheses~\ref{subsubsec_2024_07_15_1105}, the $\cR_3$-module $H^1_\upi(G_p, T_\upi^\dagger)$ has a basis of the form 
    \[
\{\varphi_{2, \upi} :=a_{\upi}f_0 + b_{\upi}f_1 + r_{\upi}c_{\upi}f_2 + r_{\upi}d_{\upi}f_3\,,\,\, \varphi_{1, \upi} := c_{\upi}f_0+d_{\upi}f_1\}\,,
\] 
where at least one of $c_{\upi}$ and $d_{\upi}$ is a unit. 
\end{lemma}

\subsubsection{Basis of $H^1_M(G_p, T_\upi^\dagger)$}
\label{subsubsec_2025_07_18_1212}
Applying a similar argument as in the previous section, we obtain the following:

\begin{lemma}\label{lemma:basis_of_H^1_M(G_p, T_M^dagger)}
Under Hypotheses~\ref{subsubsec_2024_07_15_1105}, the $\cR_3$-module $H^1_M(G_p, T_\upi^\dagger)$ has a basis of the form 
    \[
\{\varphi_{2, M} :=a_{M}f_0 + b_{M}f_1 + r_{M}c_{M}f_2 + r_{M}d_{M}f_3\,,\,\, \varphi_{1, M} := c_{M}f_0+d_{M}f_1\}\,,
\] 
where at least one of $c_{M}$ and $d_{M}$ is a unit. 
\end{lemma}

\subsubsection{An explicit version of Theorem~\ref{thm_2025_03_11}}

Since $\mathrm{res}_p(\delta(T^\dagger_\upi, \Delta_{\emptyset})) \subset \bigwedge^2_{\cR_\upi}H^1_\upi(G_p, T^\dagger_\upi)$, by Lemma \ref{lemma:basis_of_H^1_upi(G_p, T_upi^dagger)}, there is an element $x_\upi \in \cR_3$ such that 
\[
\mathrm{res}_p(\delta(T^\dagger_\upi, \Delta_{\emptyset})) = \cR_3 x_\upi \varphi_{1, \upi } \wedge \varphi_{2, \upi}. 
\]
Similarly, by  Lemma \ref{lemma:basis_of_H^1_M(G_p, T_M^dagger)}, there is an element $x_M \in \cR_3$ such that 
\[
\mathrm{res}_p(\delta(M_3^\dagger, {\rm tr}^*\Delta_a)) = \cR_3 x_M \varphi_{1, M} \wedge \varphi_{2, M}. 
\]

\begin{proposition}
\label{prop_2025_07_15_1443}
   In the setting of Theorem~\ref{thm_2025_03_11}, we have 
    \begin{align*}
 \delta(T_3^\dagger, \Delta_a) &= \cR_3 x_\upi x_M r_\upi r_M (c_Md_\upi - d_Mc_\upi)^2\,,\\
 \mathrm{Log}_{/\upi}^{(2)}\,\delta(M_3^\dagger, {\rm tr}^*\Delta_a)   &=  \cR_3 x_M r_\upi r_M (c_Md_\upi - d_Mc_\upi)^2\,,\\
 (\psi_3 \circ \mathrm{res}_p \wedge \psi_4 \circ \mathrm{res}_p) \,  \delta(T^\dagger_{\underline{\Pi}}, \Delta_{\emptyset}) &=  \cR_3 x_\upi  \,.
\end{align*}
\end{proposition}
\begin{proof}
Since 
\begin{align*}
    \varphi_{1, M} \wedge \varphi_{1, \upi} \wedge \varphi_{2, M} \wedge \varphi_{2, \upi} = -r_\upi r_M (c_Md_\upi - d_Mc_\upi)^2  f_0 \wedge f_1 \wedge f_2 \wedge f_3, 
\end{align*}
the first assertion in our proposition follows from Proposition \ref{prop:delta_a-formula2}: 
\begin{align*}
 \delta(T_3^\dagger, \Delta_a) =  \mathrm{Log}_{\mathscr{F}^0}^{(4)}(\delta(M_3^\dagger, {\rm tr}^*\Delta_a)  \otimes \delta(T^\dagger_{\underline{\Pi}}, \Delta_{\emptyset}))
 =  \cR_3x_\upi x_M r_\upi r_M (c_Md_\upi - d_Mc_\upi)^2. 
\end{align*}  
Next, let us prove that 
\[
\mathrm{Log}_{/\upi}^{(2)}\,\delta(M_3^\dagger, {\rm tr}^*\Delta_a)   =  \cR_3 x_M r_M (c_Md_\upi - d_Mc_\upi)^2. 
\]
Since $H^1_\upi(G_p, T^\dagger_\upi)^\perp = H^1_\upi(G_p, T^\dagger_\upi)$, we have 
\[
\mathrm{Log}_{/\upi}^{(2)}(\,-\,) = \langle -, \varphi_{1,\upi} \rangle \wedge  \langle -, \varphi_{2,\upi} \rangle. 
\]
As $\mathrm{res}_p(\delta(M_3^\dagger, {\rm tr}^*\Delta_a)) = \cR_3 x_M \varphi_{1, M} \wedge \varphi_{2, M}$, it follows that 
\begin{align*}
  \mathrm{Log}_{/\upi}^{(2)}\,\delta(M_3^\dagger, {\rm tr}^*\Delta_a) =    \cR_3 x_M \times \det \begin{pmatrix}
      \langle \varphi_{1,M}, \varphi_{1,\upi} \rangle & \langle \varphi_{2,M}, \varphi_{1,\upi} \rangle
      \\
      \langle \varphi_{1,M}, \varphi_{2,\upi} \rangle & \langle \varphi_{2,M}, \varphi_{2,\upi} \rangle
  \end{pmatrix}. 
\end{align*}
Since $\langle \varphi_{1,M}, \varphi_{1,\upi} \rangle = 0$, we have 
\begin{align*}
     \det \begin{pmatrix}
      \langle \varphi_{1,M}, \varphi_{1,\upi} \rangle & \langle \varphi_{2,M}, \varphi_{1,\upi} \rangle
      \\
      \langle \varphi_{1,M}, \varphi_{2,\upi} \rangle & \langle \varphi_{2,M}, \varphi_{2,\upi} \rangle 
      \end{pmatrix} 
      &= - \langle \varphi_{2,M}, \varphi_{1,\upi} \rangle \times  \langle \varphi_{1,M}, \varphi_{2,\upi} \rangle 
      \\
      &= - r_M(c_M d_\upi - d_M c_\upi) \times r_\upi (c_M d_\upi - d_M c_\upi). 
\end{align*}
This concludes the proof of the second equality. 

Let us next show that $ (\psi_3 \circ \mathrm{res}_p \wedge \psi_4 \circ \mathrm{res}_p) \,  \delta(T^\dagger_{\underline{\Pi}}, \Delta_{\emptyset}) =  \cR_3 x_\upi$. 
Note that there exist $s_3, s_4 \in H^1(G_p, T^\dagger_\upi)$ such that $\psi_i = \langle -, s_i \rangle$. Then $\{s_1:=\varphi_{1,\upi}, s_2:=\varphi_{2,\upi}, s_3, s_4\}$ is a basis of $H^1(G_p, T^\dagger_\upi)$, and 
\[
\det(\langle s_i, s_j \rangle) \in \cR_3^\times
\]
since the local Tate-duality pairing on $H^1(G_p, T^\dagger_\upi)$ is perfect. Since $\langle s_i, s_j \rangle = 0$ for $i,j \in \{1, 2\}$, it follows that 
\[
\det \begin{pmatrix}
      \langle s_1, s_3 \rangle & \langle s_2 , s_3 \rangle
      \\
      \langle s_1, s_4 \rangle & \langle s_2, s_4 \rangle 
      \end{pmatrix} \in \cR_3^\times. 
\]
This concludes the proof that
\[
(\psi_3 \circ \mathrm{res}_p \wedge \psi_4 \circ \mathrm{res}_p) \,  \delta(T^\dagger_{\underline{\Pi}}, \Delta_{\emptyset}) = \cR_\upi x_\upi \times \det \begin{pmatrix}
      \langle s_1, s_3 \rangle & \langle s_2 , s_3 \rangle
      \\
      \langle s_1, s_4 \rangle & \langle s_2, s_4 \rangle 
      \end{pmatrix}
= \cR_\upi x_\upi\,,
\]
as required.
\end{proof}

\subsubsection{Proof of Theorem~\ref{thm_main_double_wall_crossing}}
\label{subsubsec_2025_07_17_1421}
Note that we have 
\begin{align*}
\delta(M_3^\dagger, \Delta_b) = \cR_3 x_M r_M d_M \varphi_{1,M}, \qquad
\delta(T^\dagger_\upi, \Delta_{\mathscr{F}^1}) = \cR_\upi x_\upi r_\upi d_\upi \varphi_{1, \upi}
\end{align*}
by definition. As a result,
\begin{align*}
    \Log_c\,\delta(M_3^\dagger, \Delta_b)\, \times \Log_c\,\delta(T^\dagger_\upi, \Delta_{\mathscr{F}^1}) = \cR_3 x_M x_\upi r_M   r_\upi d_\upi^2 d_M^2.  
\end{align*}
The proof of Theorem~\ref{thm_main_double_wall_crossing} follows from this, combined with Proposition~\ref{prop_2025_07_15_1443}.


\section{Endoscopic cases}
\label{sec_2025_02_06_0708}

In this section, we discuss the endoscopic cases where the $\GSp_4$-representation $\Pi$ (and families of such) arises from a smaller group. We remark that the formulation of our factorisation conjectures dwells on the interpolation of Eichler--Shimura isomorphisms (cf. \cite{LZ21-erl, DRW2025}). The critical points (in the sense of Bella\"iche--Chenevier, which contains the set of ``endoscopic points'' that we describe below) over the $\GSp_4$-eigenvariety are excluded from the locus of interpolation. As a result, the factorisation results in this particular case assume a slightly different shape, in that the graded piece of the local Galois representation that one trivialises is different from the generic scenario. Our Corollaries~\ref{cor_2025_08_06_1325} and~\ref{cor_2025_08_07_0919} below reflect this fact; see also the discussion in Remark~\ref{rem:yoshida_log}.

\begin{definition}
Let $\Pi$ be a cuspidal automorphic representation of $\GSp_4$ of weight $(k_1,k_2)$, with $k_1 \geq k_2 \geq 2$, and trivial central character.
\begin{enumerate}
\item[(a)] We say that $\Pi$ is of Yoshida type if there is a pair $(\pi_1,\pi_2)$ of cuspidal automorphic representations of $\GL_2$, corresponding to two elliptic modular newforms of weights $t_1 = k_1+k_2-2$ and $t_2=k_1-k_2+2$, such that for all but finitely many places $v$ we have \[ L(s,\Pi_v) = L(s,\pi_{1,v}) L(s,\pi_{2,v}). \]
\item[(b)] We say that $\Pi$ is of Saito--Kurokawa type if $k_1=k_2$ and there exists a cuspidal automorphic representation $\pi$ of $\GL_2$ attached to some holomorphic newform of weight $2k_1-2$ such that for all but finitely many places $v$ we have \[ L(s,\Pi_v) = L(s,\pi_v) \zeta \left(s-\tfrac{1}{2}\right) \zeta \left(s+\tfrac{1}{2}\right). \]
\end{enumerate}
\end{definition}

These definitions may be extended to the case where the central character is non-trivial, but considering the appropriate twists but we shall limit our discussion to the scenario where it is trivial.

\subsection{Yoshida lifts}
\label{subsec_2025_08_06_1504}

Let us put\footnote{We implicitly consider $T_\upi^\dagger$ as the Galois representation associated to a family $\upi$ of Yoshida lifts. Such families of Siegel modular forms were recently constructed by M-L. Hsieh and Z. Liu. The ad hoc definition of the Galois representation above suffices for our purposes in this subsection.} $T_\upi^\dagger=(T_{\uf_1}^\dagger\otimes_{\cR_{\uf_1}}\cR_\upi)\, \oplus \, (T_{\uf_2}^\dagger\otimes_{\cR_{\uf_2}}\cR_\upi)$, where $\uf_i$ is a Hida family, $\cR_{\uf_i}$ is the irreducible component of the universal Hecke algebra determined by $\uf_i$, and $\cR_\upi:=\cR_{\uf_1}\,\widehat\otimes\,\cR_{\uf_2}$.

Recall our running assumption that the global root number of the members of $\upi$ (at the central critical point) equals $-1$. That means one of the families $\uf_i$ has global root number $-1$, and the other $+1$; so let us assume that
\begin{equation}
    \label{eqn_2025_08_19_0840}
    \varepsilon(T_{\uf_1}^\dagger)=1 =-\varepsilon(T_{\uf_2}^\dagger)\,.
\end{equation}

\subsubsection{} 
\label{subsubsec_2025_08_25_1205}
The Galois representation $T_?^\dagger$ (where $?=3,4$) admits the following decomposition:
\begin{equation}
    \label{eqn_2025_08_19_0754_1}
    T_?^\dagger\simeq \underbrace{\left(T_{\uf_1}^\dagger \widehat{ \otimes} T_{\underline{\sigma}} \widehat \otimes T_{\underline{\sigma}}^*\right) \otimes_{\cR_{\uf_1}}\cR_{\upi}}_{T_{?,1}^\dagger}\quad \oplus \quad  \underbrace{\left(T_{\uf_2}^\dagger \widehat \otimes T_{\underline{\sigma}} \widehat \otimes T_{\underline{\sigma}}^* \right)\otimes_{\cR_{\uf_2}}\cR_{\upi}}_{T_{?,2}^\dagger}\,.
\end{equation}
When $?=3$, each summand above further decomposes as
\begin{equation}
    \label{eqn_2025_08_19_0754}
    T_3^\dagger\simeq \underbrace{\left((T_{\uf_1}^\dagger\otimes {\rm ad}^0 T_{\underline{\sigma}})\oplus T_{\uf_1}^\dagger\right)\otimes_{\cR_{\uf_1}}\cR_{\upi}}_{T_{3,1}^\dagger}\quad\oplus\quad  \underbrace{\left((T_{\uf_2}^\dagger\otimes {\rm ad}^0 T_{\underline{\sigma}})\oplus T_{\uf_2}^\dagger\right)\otimes_{\cR_{\uf_2}}\cR_{\upi}}_{T_{3,2}^\dagger}\,.
\end{equation}
Moreover, in view of our discussion in \S\ref{subsubsec_2025_08_19_0801}, we have:
\begin{itemize}
\item[-] Region (a) corresponds to the $\underline{\sigma}^c$-dominant region for both triple products $T_{4,?}^\dagger$ ($?=1,2$).

\item[-] Region (b) corresponds to the $\underline{\sigma}^c$-dominant region for $T_{4,2}^\dagger$ and balanced region for $T_{4,1}^\dagger$. 

\item[-] Region (c) corresponds to the balanced region for both.
\end{itemize}
Based on this, \eqref{eqn_2025_08_19_0754_1}, and \eqref{eqn_2025_08_19_0757}, we set $ L_p^{(a)}(T_?^\dagger):= L_p^{(\underline{\sigma}^c)}(T_{?,1}^\dagger)\times L_p^{(\underline{\sigma}^c)}(T_{?,2}^\dagger)$.
\subsubsection{} We also have the decomposition
\begin{equation}
    \label{eqn_2025_08_19_0757}
M_3^\dagger\simeq \underbrace{\left(T_{\uf_1}^\dagger\otimes {\rm ad}^0 T_{\underline{\sigma}}\right)\otimes_{\cR_{\uf_1}}\cR_{\upi}}_{M_{3,1}^\dagger}\quad\oplus\quad \underbrace{\left(T_{\uf_2}^\dagger\otimes {\rm ad}^0 T_{\underline{\sigma}}\right)\otimes_{\cR_{\uf_2}}\cR_{\upi}}_{M_{3,2}^\dagger}.
\end{equation}

\subsubsection{} Recall that we work under the assumption that the global root number of the specialisations of $T_4^\dagger$ in region (b) is $-1$, and equivalently, it is $+1$ for those in region (c):
\begin{align*}
    1=-\varepsilon^{(b)}(T_4^\dagger)&\,=\varepsilon^{(c)}(T_4^\dagger)=\varepsilon^{\rm bal}(T_{4,1}^\dagger)\,\cdot\,\varepsilon^{\rm bal}(T_{4,2}^\dagger)\,=\varepsilon^{\rm bal}(M_{3,1})\,\varepsilon(T_{\uf_1})\,\cdot\,\varepsilon^{\rm bal}(M_{3,2})\,\varepsilon(T_{\uf_2})\,.
\end{align*}
That means $T_{4,1}^\dagger$ and $T_{4,2}^\dagger$ have the same sign in the balanced region. We assume that this sign is $-1$, so that we have
$$-\varepsilon^{\uf_1}(M_{3,1})=\varepsilon^{\rm bal}(M_{3,1})=-1=-\varepsilon^{\rm bal}(M_{3,2})=\varepsilon^{\uf_2}(M_{3,2})\,.$$

\begin{remark}
Unlike in the non-endoscopic cases, classical specialisations of $M_3^\dagger$ in the region (a) are critical in the sense of Deligne: As we have noted above, these correspond to $\underline{\sigma}^c$-dominant conditions on both of the $\GL_2 \times \GL_3$ factors. Note that ``$\underline{\sigma}^c$-dominant''  is synonymous with ``balanced'' for these factors. 
\end{remark}


\begin{conj}
\label{conj_2025_08_06_1401}
Under the current assumptions, the following factorisation formula holds:
\begin{align}
    \begin{aligned}
        \label{eqn_2025_08_06_1332}
        L_p^{(a)}(T_3^\dagger)^2&:= L_p^{(\underline{\sigma}^c)}(T_{3,1}^\dagger)^2\times L_p^{(\underline{\sigma}^c)}(T_{3,2}^\dagger)^2\\
        &\,=\left(\Log_{\omega_{M_{3,1}}} (\Delta_{M_{3,1}^\dagger})^2\times L_p(\uf_1)\right)\times \left(L_p^{\rm bal}(M_{3,2}^\dagger)\times \Log_{\omega_{\uf_2}} (\Delta_{\uf_2}^\dagger)^2\right)\\
        &\,=\underbrace{\left(\Log_{\omega_{M_{3,1}}}  (\Delta_{M_{3,1}^\dagger})^2\times L_p^{\rm bal}(M_{3,2}^\dagger)\right)}_{\Log_{\omega_{M_3}}(\Delta_{M_3^\dagger}^{(c)})^2}\times \underbrace{\left(L_p(\uf_1)\times \Log_{\omega_{\uf_2}} (\Delta_{\uf_2}^\dagger)^2\right)}_{\Log_{\omega_\upi}(\Delta_\upi^\dagger)^2}\,.
    \end{aligned}
\end{align}
up to multiplication by an element of $\cR_{\upi}[\frac{1}{p}]^\times$ that specialises to an explicit algebraic number at all classical specialisations of $(\uf_1,\uf_2)$. Here:
\begin{itemize}
\item The trivialisation $\Log_{\omega_{M_{3,1}}}$ is defined at the start of \S\ref{subsubsec_2025_08_06_1543} below.
    \item $\Delta_{M_{3,1}^\dagger}\in H^1(\QQ,M_{3,1}^\dagger)$ is the conjectural family of twisted diagonal cycles (cf. \cite{HsiehYamana2023} and the forthcoming work of Chida--Hsieh--Prasanna), whereas $\Delta_{\uf_2}^\dagger\in  H^1(\QQ,T_{\uf_2}^\dagger)$ is the family of diagonal cycles of Darmon--Rotger and Bertolini--Seveso--Venerucci.
    \item $L_p(\uf_1)$ is the restricion of the Mazur--Kitagawa $p$-adic $L$-function to the central critical line, whereas $L_p^{\rm bal}(M_{3,2}^\dagger)$ is the same for the conjectural $p$-adic $L$-function associated to the family $M_{3,2}^\dagger$ (cf. \cite{BC}, Conjecture 2.4).
    \item  $\Delta_{M_3^\dagger}^{(c)}:=L_p^{\rm bal}(M_{3,2}^\dagger)\cdot \Delta_{M_{3,1}^\dagger}$ and $\Delta_\upi^\dagger:=L_p(\uf_1)\cdot \Delta_{\uf_2}^\dagger$\,.
\end{itemize}
\end{conj}

In \eqref{eqn_2025_08_06_1332}, assuming the truth of the Gross--Kudla conjecture\footnote{In its more general form as \cite[Conjecture 5.1]{bcpv}.} and the ``non-anomality condition'' \eqref{item_non_anom}, the factorisation 
$$L_p^{(\underline{\sigma}^c)}(T_{3,2}^\dagger)^2\,=\,L_p^{\rm bal}(M_{3,2}^\dagger)\times \Log (\Delta_{\uf_2}^\dagger)^2 \pmod{\cR_{\uf_2}^\times}$$
is proved in \cite{BS_triple_factor} and its sequel \cite{bcpv}. The remaining part of the conjectured equality~\eqref{eqn_2025_08_06_1332} is the assertion that
\begin{equation}
\label{eqn_2025_08_25_1222}
    L_p^{(\underline{\sigma}^c)}(T_{3,1}^\dagger)^2\,\dot=\,\Log (\Delta_{M_{3,1}^\dagger})^2 \times L_p(\uf_1)\,,
\end{equation}
and it supplements the results of \cite{BS_triple_factor, bcpv}; its proof is work in progress by the first and third named authors with A. Cauchi. We believe that similar techniques to those in \cite{BC} would allow one to give an unconditional proof of Conjecture~\ref{conj_2025_08_06_1401} when the family $\underline{\sigma}$ has CM.

We refer the reader to Theorem~\ref{thm:delta_decomp_yoshida} for its algebraic counterpart in terms of modules of leading terms.

\subsection{The algebraic counterpart of Conjecture~\ref{conj_2025_08_06_1401}}
\label{subsec_2025_08_25_1428}

As we have remarked at the start of \S\ref{sec_2025_02_06_0708}, Theorem~\ref{thm_main_double_wall_crossing} needs to be modified in the endoscopic cases. In this subsection, we explain why indeed this is necessary (cf. Corollary~\ref{cor_2025_08_06_1325}, Corollary~\ref{cor_2025_08_07_0919} and Remark~\ref{rem:yoshida_log}) and prove (in Corollary~\ref{cor_2025_08_14_1551}) the appropriate variant of Theorem~\ref{thm_main_double_wall_crossing} in the setting of \S\ref{subsec_2025_08_06_1504} (which is the algebraic counterpart of Conjecture~\ref{conj_2025_08_06_1401}).


\subsubsection{Filtration on $T_\upi^\dagger$}
We consider the filtration $\mathscr{F}^\bullet$  on $T_\upi^\dagger=(T_{\uf_1}^\dagger\otimes_{\cR_{\uf_1}}\cR_\upi)\, \oplus \, (T_{\uf_2}^\dagger\otimes_{\cR_{\uf_2}}\cR_\upi)$ and determine how it propagates to each summand. 
Since $(\mathscr{F}^2)^\perp = \mathscr{F}^2$, we have 
\[
\mathscr{F}^2T_\upi^\dagger = F^+T_{\uf_1}^\dagger\otimes_{\cR_{\uf_1}}\cR_\upi \quad \oplus\quad F^+T_{\uf_2}^\dagger \otimes_{\cR_{\uf_2}}\cR_\upi\,. 
\]
Here $F^+T_{\uf_i}^\dagger$ denotes the rank one  $G_p$-submodule of $T_{\uf_i}^\dagger$,  whose existence is guaranteed by the $p$-ordinarity. 
Moreover, the fact that $\mathscr{F}^3 T_\upi^\dagger$ has rank one implies that it must be equal to $F^+ T_{\uf_i}^\dagger$ for some $i \in \{1, 2\}$. 
Furthermore, using the relation $(\mathscr{F}^3)^\perp = \mathscr{F}^1$, it follows that
\[
\mathscr{F}^1 T_\upi^\dagger = 
\begin{cases}
    F^+ T_{\uf_1}^\dagger\otimes_{\cR_{\uf_1}}\cR_\upi\quad \oplus \quad T_{\uf_2}^\dagger\otimes_{\cR_{\uf_2}}\cR_\upi & \qquad \textrm{if} \quad i = 1, 
    \\
     T_{\uf_1}^\dagger\otimes_{\cR_{\uf_1}}\cR_\upi\quad \oplus \quad F^+ T_{\uf_2}^\dagger\otimes_{\cR_{\uf_2}}\cR_\upi &\qquad  \textrm{if} \quad i = 2. 
\end{cases}
\]
To summarize, we have two possible candidates for the filtration $\mathscr{F}^\bullet$  on $T_\upi^\dagger$, each of which is determined by the choice of $\mathscr{F}^3T_\upi^\dagger$: 

\begin{definition}
    For each integer $i \in \{1,2\}$, we call 
    the filtration $\mathscr{F}^\bullet_i$ on $T_\upi^\dagger$ satisfying $\mathscr{F}^3 T_\upi^\dagger =  F^+ T_{\uf_i}^\dagger$  an $f_i$-dominant filtration. 
\end{definition}

\subsubsection{} 
In what follows, for any ${\cR_{\uf_i}}$-module (or a complex of ${\cR_{\uf_i}}$-modules) $M$, let us write $M_{\cR_\upi}$ (resp., $M_{\cR_3}$) in place of $M\otimes_{\cR_{\uf_i}}\cR_\upi$ (resp. of $M\otimes_{\cR_{\uf_i}}\cR_3$) to ease our notation. The following three propositions follow immediately from the definition of $\mathscr{F}^\bullet_i$. 

\begin{proposition}\label{prop:decomp_yohida_Tupi}
Suppose that $\mathscr{F}^\bullet = \mathscr{F}^\bullet_1$. 
\begin{align*}
\xymatrix@C=6pt@R=0pt{
  &\mathbf{R}\Gamma_{\fin}(G_{\mathbb Q, \Sigma}, T_{\upi}^\dagger, \Delta_{\mathscr{F}^0})  &= &\mathbf{R}\Gamma_{\fin}(G_{\mathbb Q, \Sigma}, T_{\uf_1}^\dagger, \Delta_{\emptyset})_{\cR_\upi}  &\oplus 
 &\mathbf{R}\Gamma_{\fin}(G_{\mathbb Q, \Sigma}, T_{\uf_2}^\dagger, \Delta_{\emptyset})_{\cR_\upi},
 \\
 &\mathbf{R}\Gamma_{\fin}(G_{\mathbb Q, \Sigma}, T_{\upi}^\dagger, \Delta_{\mathscr{F}^1})  &= &\mathbf{R}\Gamma_{\fin}(G_{\mathbb Q, \Sigma}, T_{\uf_1}^\dagger, \Delta_{F^+})_{\cR_\upi}  &\oplus 
&\mathbf{R}\Gamma_{\fin}(G_{\mathbb Q, \Sigma}, T_{\uf_2}^\dagger, \Delta_{\emptyset})_{\cR_\upi}, 
\\
 &\mathbf{R}\Gamma_{\fin}(G_{\mathbb Q, \Sigma}, T_{\upi}^\dagger, \Delta_{\mathscr{F}^2})  &= &\mathbf{R}\Gamma_{\fin}(G_{\mathbb Q, \Sigma}, T_{\uf_1}^\dagger, \Delta_{F^+})_{\cR_\upi} &\oplus &\mathbf{R}\Gamma_{\fin}(G_{\mathbb Q, \Sigma}, T_{\uf_2}^\dagger, \Delta_{F^+})_{\cR_\upi}, 
\\
 &\mathbf{R}\Gamma_{\fin}(G_{\mathbb Q, \Sigma}, T_{\upi}^\dagger, \Delta_{\mathscr{F}^3})  &= &\mathbf{R}\Gamma_{\fin}(G_{\mathbb Q, \Sigma}, T_{\uf_1}^\dagger, \Delta_{F^+})_{\cR_\upi} &\oplus &\mathbf{R}\Gamma_{\fin}(G_{\mathbb Q, \Sigma}, T_{\uf_2}^\dagger, \Delta_{0})_{\cR_\upi}, 
\\
 &\mathbf{R}\Gamma_{\fin}(G_{\mathbb Q, \Sigma}, T_{\upi}^\dagger, \Delta_{\mathscr{F}^4})  &=  &\mathbf{R}\Gamma_{\fin}(G_{\mathbb Q, \Sigma}, T_{\uf_1}^\dagger, \Delta_{0})_{\cR_\upi} &\oplus &\mathbf{R}\Gamma_{\fin}(G_{\mathbb Q, \Sigma}, T_{\uf_2}^\dagger, \Delta_{0})_{\cR_\upi}\,.  
 }
\end{align*}
In the case of $\mathscr{F}^\bullet_2$, the roles of $\uf_1$ and $\uf_2$ are reversed. 
\end{proposition}

By definition, we have a canonical homomorphism 
\begin{align}\label{hom:decomp_Tf1_Tf2_Tupi}
 \mathbf{R}^1\Gamma_{\fin}(G_{\mathbb Q, \Sigma}, T_{\uf_1}^\dagger, \Delta_{\emptyset}) _{\cR_\upi} \otimes  
 \mathbf{R}^1\Gamma_{\fin}(G_{\mathbb Q, \Sigma}, T_{\uf_2}^\dagger, \Delta_{\emptyset})_{\cR_\upi} \longrightarrow  {\bigcap}^2_{\cR_{\upi}} \mathbf{R}^1\Gamma_{\fin}(G_{\mathbb Q, \Sigma}, T_{\upi}^\dagger, \Delta_{\mathscr{F}^0}).     
\end{align}
It therefore becomes possible to compare $\delta(T_\upi^\dagger, \Delta_{\emptyset})$ with  $\delta(T_{\uf_1}^\dagger, \Delta_{\emptyset})_{\cR_\upi} \otimes \,\delta(T_{\uf_2}^\dagger, \Delta_{\emptyset})_{\cR_\upi}$. 

\begin{corollary}\label{cor:decomp_delta_yohida_Tupi}
We have 
 \begin{equation}
 \label{eqn_2025_09_23_1703}
     \delta(T_\upi^\dagger, \Delta_{\emptyset}) =  \delta(T_{\uf_1}^\dagger, \Delta_{\emptyset})_{\cR_\upi}\, \otimes \,\delta(T_{\uf_2}^\dagger, \Delta_{\emptyset})_{\cR_\upi}\,. 
 \end{equation}
\end{corollary}
\begin{proof}
If $\delta(T_\upi^\dagger, \Delta_{\emptyset}) = 0$, then at least one of $\delta(T_{\uf_1}^\dagger, \Delta_{\emptyset})_{\cR_\upi}$ or  $\delta(T_{\uf_2}^\dagger, \Delta_{\emptyset})_{\cR_\upi}$ must vanish.
Thus, we may assume that all three are nonzero. 
Then Theorem \ref{thm_2023_03_10}(ii) and Proposition \ref{prop:decomp_yohida_Tupi} shows that 
\begin{align*}    
&\Char_{\cR_{\upi}}\left( \left. {\bigcap}^2_{\cR_{\upi}} \mathbf{R}^1\Gamma_{\fin}(G_{\mathbb Q, \Sigma}, T_{\upi}^\dagger, \Delta_{\mathscr{F}^0})   \middle/  \delta(T_{\upi}^\dagger, \Delta_{\emptyset})  \right. \right) 
\\
&= 
\Char_{\cR_{\upi}}\left( \mathbf{R}^2\Gamma_{\fin}(G_{\mathbb Q, \Sigma}, T_{\upi}^\dagger, \Delta_{\mathscr{F}^0}) \right)
\\
&=  \Char_{\cR_{\upi}}\left(\mathbf{R}^2\Gamma_{\fin}(G_{\mathbb Q, \Sigma}, T_{\uf_1}^\dagger, \Delta_{\emptyset}) \right)_{\cR_\upi}  \times \Char_{\cR_{\upi}}\left(
 \mathbf{R}^2\Gamma_{\fin}(G_{\mathbb Q, \Sigma}, T_{\uf_2}^\dagger, \Delta_{\emptyset}) \right) _{\cR_\upi}
 \\
 &= \Char_{\cR_{\upi}}\left( \left. \mathbf{R}^1\Gamma_{\fin}(G_{\mathbb Q, \Sigma}, T_{\uf_1}^\dagger, \Delta_{\emptyset})_{\cR_\upi}\,  \otimes  
 \mathbf{R}^1\Gamma_{\fin}(G_{\mathbb Q, \Sigma}, T_{\uf_2}^\dagger, \Delta_{\emptyset}) _{\cR_\upi}   \middle/  \delta(T_{\uf_1}^\dagger, \Delta_{\emptyset})_{\cR_\upi}  \otimes  \delta(T_{\uf_2}^\dagger, \Delta_{\emptyset})_{\cR_\upi} \right. \right). 
\end{align*}
    Since both ${\bigcap}^2_{\cR_{\upi}} \mathbf{R}^1\Gamma_{\fin}(G_{\mathbb Q, \Sigma}, T_{\upi}^\dagger, \Delta_{\mathscr{F}^0})$ and $  \mathbf{R}^1\Gamma_{\fin}(G_{\mathbb Q, \Sigma}, T_{\uf_1}^\dagger, \Delta_{\emptyset})_{\cR_\upi} \,  \otimes  \,
 \mathbf{R}^1\Gamma_{\fin}(G_{\mathbb Q, \Sigma}, T_{\uf_2}^\dagger, \Delta_{\emptyset})_{\cR_\upi} $ are free of rank $1$ and the homomorphism \eqref{hom:decomp_Tf1_Tf2_Tupi} is an isomorphism, we obtain the desired equality $ \delta(T_\upi^\dagger, \Delta_{\emptyset}) =  \delta(T_{\uf_1}^\dagger, \Delta_{\emptyset})_{\cR_\upi}  \otimes  \delta(T_{\uf_2}^\dagger, \Delta_{\emptyset})_{\cR_\upi}$. 
\end{proof}

Recall that we have two Greenberg local conditions $\Delta_{\underline{\sigma}}$ and $\Delta_{\mathrm{bal}}$ on $T_{3,i}^\dagger$, as defined in \cite[Example 4.1]{BS_triple_factor}. These are referred to as the ${\underline{\sigma}}$-dominant Greenberg local condition and the balanced Greenberg local condition, respectively. They are defined via the following choices of submodules of $T_{3,i}^\dagger$:
\begin{itemize}
\item	The ${\underline{\sigma}}$-dominant condition $\Delta_{\underline{\sigma}}$ is defined by the submodule
\[
F_{\underline{\sigma}}^+T_{3,i}^\dagger := (T_{\uf_i}^\dagger\otimes_{\cR_{\uf_i}}\cR_\upi) \,\widehat\otimes\, F^+T_{\underline{\sigma}}^\dagger \,\otimes_{\cR_{\underline{\sigma}}} \,T_{\underline{\sigma}^c}^\dagger \hookrightarrow T_{3,i}^\dagger.
\]
\item	The balanced condition $\Delta_{\mathrm{bal}}$ is defined by the sum of submodules
\begin{align*}
    F_{\rm bal}^+T_{3,i}^\dagger := (F^+T_{\uf_i}^\dagger \otimes_{\cR_{\uf_i}}\cR_\upi) \,\widehat\otimes \,F^+T_{\underline{\sigma}}^\dagger \otimes T_{\underline{\sigma}^c}^\dagger + (T_{\uf_i}^\dagger\otimes_{\cR_{\uf_i}}\cR_\upi) \,&\widehat\otimes F^+T_{\underline{\sigma}}^\dagger \otimes F^+T_{\underline{\sigma}^c}^\dagger \\
    &+ (F^+T_{\uf_i}^\dagger \otimes_{\cR_{\uf_i}}\cR_\upi)\widehat\otimes T_{\underline{\sigma}}^\dagger \otimes F^+T_{\underline{\sigma}^c}^\dagger \hookrightarrow T_{3,i}^\dagger\,.
\end{align*}

\end{itemize}
Moreover,  using the homomorphism $\mathrm{tr}^* \colon T_{3,i} \longrightarrow M_{3,i}$, we obtain two induced Greenberg local conditions on $M_{3,i}$, denoted by $\mathrm{tr}^*\Delta_{\underline{\sigma}}$ and $\mathrm{tr}^*\Delta_{\mathrm{bal}}$.

\begin{proposition}\label{prop:decomp_yohida_M3}
Suppose that $\mathscr{F}^\bullet = \mathscr{F}^\bullet_1$. 
    \begin{align*}
    \xymatrix@C=6pt@R=0pt{
&\mathbf{R}\Gamma_{\fin}(G_{\mathbb Q, \Sigma}, M_3^{\dagger}, \tr^*\Delta_a)  &=& \mathbf{R}\Gamma_{\fin}(G_{\mathbb Q, \Sigma}, M_{3,1}^{\dagger}, \tr^*\Delta_{\underline{\sigma}})  &\oplus &
\mathbf{R}\Gamma_{\fin}(G_{\mathbb Q, \Sigma}, M_{3,2}^{\dagger}, \tr^*\Delta_{\underline{\sigma}}), 
\\
&\mathbf{R}\Gamma_{\fin}(G_{\mathbb Q, \Sigma}, M_3^{\dagger}, \tr^*\Delta_b)  &= &\mathbf{R}\Gamma_{\fin}(G_{\mathbb Q, \Sigma}, M_{3,1}^{\dagger}, \tr^*\Delta_{\rm bal})  &\oplus &
\mathbf{R}\Gamma_{\fin}(G_{\mathbb Q, \Sigma}, M_{3,2}^{\dagger}, \tr^*\Delta_{\underline{\sigma}}), 
\\
&\mathbf{R}\Gamma_{\fin}(G_{\mathbb Q, \Sigma}, M_3^{\dagger}, \tr^*\Delta_c)  &= &\mathbf{R}\Gamma_{\fin}(G_{\mathbb Q, \Sigma}, M_{3,1}^{\dagger}, \tr^*\Delta_{\rm bal}) &\oplus &\mathbf{R}\Gamma_{\fin}(G_{\mathbb Q, \Sigma}, M_{3,2}^{\dagger}, \tr^*\Delta_{\rm bal}), 
\\
&\mathbf{R}\Gamma_{\fin}(G_{\mathbb Q, \Sigma}, M_3^{\dagger}, \tr^*\Delta_b^\perp)  &=& \mathbf{R}\Gamma_{\fin}(G_{\mathbb Q, \Sigma}, M_{3,1}^{\dagger}, \tr^*\Delta_{\rm bal}) &\oplus &\mathbf{R}\Gamma_{\fin}(G_{\mathbb Q, \Sigma}, M_{3,2}^{\dagger}, \tr^*\Delta_{\underline{\sigma}}^\perp)\,. 
\\
&\mathbf{R}\Gamma_{\fin}(G_{\mathbb Q, \Sigma}, M_3^{\dagger}, \tr^*\Delta_a^\perp)  &=  &\mathbf{R}\Gamma_{\fin}(G_{\mathbb Q, \Sigma}, M_{3,1}^{\dagger}, \tr^*\Delta_{\underline{\sigma}}^\perp)& \oplus &\mathbf{R}\Gamma_{\fin}(G_{\mathbb Q, \Sigma}, M_{3,2}^{\dagger}, \tr^*\Delta_{\underline{\sigma}}^\perp)\,.
}
\end{align*} 
In the case of $\mathscr{F}^\bullet_2$, the roles of $M_{3,1}^{\dagger}$ and $M_{3,2}^{\dagger}$ are reversed. 
\end{proposition}

Arguing in the same way as in the proof of Corollary \ref{cor:decomp_delta_yohida_Tupi}, we obtain the following: 

\begin{corollary}\label{cor:decomp_delta_yohida_M}
We have 
 \begin{equation}
 \label{eqn_2025_09_23_1733}
     \delta(M_3^{\dagger}, \tr^*\Delta_a) =  \delta(M_{3,1}^{\dagger}, \tr^*\Delta_{\underline{\sigma}}) \otimes  \delta(M_{3,2}^{\dagger}, \tr^*\Delta_{\underline{\sigma}}). 
 \end{equation}
\end{corollary}

\begin{proposition}\label{prop:basis_yohida_Tupi}
Let $\{f_0, f_1, f_2, f_3\}$ be the basis of $H^1(G_p, T_{\upi}^\dagger)$ given as in \S\ref{subsubsec_2025_07_18_1211}. 
    \item[i)] 
    If $\mathscr{F}^\bullet = \mathscr{F}^\bullet_1$, then 
    we have $\cR_\upi f_0 = H^1(G_p, F^+T_{\uf_1}^\dagger)$ and $\cR_\upi f_1 = H^1(G_p, F^+T_{\uf_2}^\dagger)$. Moreover,  $f_2 \in H^1(G_p, T_{\uf_2}^\dagger)$ and $f_3 \in H^1(G_p, T_{\uf_1}^\dagger)$. 
    \item[ii)] If $\mathscr{F}^\bullet = \mathscr{F}^\bullet_2$, then we have  $\cR_\upi f_0 = H^1(G_p, F^+T_{\uf_2}^\dagger)$ and $\cR_\upi f_1 = H^1(G_p, F^+T_{\uf_1}^\dagger)$. Moreover,  $f_2 \in H^1(G_p, T_{\uf_1}^\dagger)$ and $f_3 \in H^1(G_p, T_{\uf_2}^\dagger)$. 
\end{proposition}

\begin{corollary} 
In the setting of Proposition \ref{prop:basis_yohida_Tupi}, if $\mathscr{F}^\bullet = \mathscr{F}^\bullet_1$, then  $c_\upi = 0$ and $d_\upi$ is a unit. If $\mathscr{F}^\bullet = \mathscr{F}^\bullet_2$, then $c_\upi$ is unit and $d_\upi = 0$. 
\end{corollary}
\begin{proof}
The saturation of the image of $\mathbf{R}^1\Gamma_{\fin}(G_{\mathbb Q, \Sigma}, T_{\upi}^\dagger, \Delta_{\mathscr{F}^2})$ equals $H^1(G_p, T_{\uf_2}^\dagger)_{\cR_\upi}$. This corollary therefore follows from Proposition \ref{prop:basis_yohida_Tupi}. 
\end{proof}

\subsubsection{} We recall that we are working under the assumption that $\varepsilon(T_{\uf_1}^\dagger)=1 =-\varepsilon(T_{\uf_2}^\dagger)$. Let us consider the following analogues of the hypotheses listed in \S\ref{subsubsec_2024_07_15_1105} (see also Remark~\ref{rem_2025_09_23_1946} below):
\begin{itemize}
\item[\mylabel{item_AJ_f}{$\mathbf{AJ}_{p}^{\uf_i}$})] $\res_p\,\delta(T^\dagger_{\uf_i},\Delta_{\emptyset}) \neq 0 $.
\item[\mylabel{item_parity_f}{$\mathbf{P}^-_{\uf_2}$})] $\delta(T^\dagger_{\uf_2}, \Delta_{\mathscr{F}^+})=0$.
    \end{itemize}


\begin{proposition}\label{prop:yoshida_H^1(Tupi)} 
Suppose that $\delta(T^\dagger_{\uf_1}, \Delta_{F^+}) \neq 0$. Assume also that 
\eqref{item_AJ_f}  and \eqref{item_parity_f} hold true $(i=1,2)$.
       \item[i)] If $\mathscr{F}^\bullet = \mathscr{F}^\bullet_1$, then we have 
\begin{align*}
&\mathbf{R}^1\Gamma_{\fin}(G_{\mathbb Q, \Sigma}, T_{\upi}^\dagger, \Delta_{\mathscr{F}^0})  = \mathbf{R}^1\Gamma_{\fin}(G_{\mathbb Q, \Sigma}, T_{\uf_1}^\dagger, \Delta_{\emptyset})_{\cR_\upi}  \oplus 
\mathbf{R}^1\Gamma_{\fin}(G_{\mathbb Q, \Sigma}, T_{\uf_2}^\dagger, \Delta_{F^+})_{\cR_\upi}\,, 
\\
&\mathbf{R}^1\Gamma_{\fin}(G_{\mathbb Q, \Sigma}, T_{\upi}^\dagger, \Delta_{\mathscr{F}^1})  =
\mathbf{R}^1\Gamma_{\fin}(G_{\mathbb Q, \Sigma}, T_{\upi}^\dagger, \Delta_{\mathscr{F}^2})  =  \mathbf{R}^1\Gamma_{\fin}(G_{\mathbb Q, \Sigma}, T_{\uf_2}^\dagger, \Delta_{F^+})_{\cR_\upi}\,, 
\\
&\mathbf{R}^1\Gamma_{\fin}(G_{\mathbb Q, \Sigma}, T_{\upi}^\dagger, \Delta_{\mathscr{F}^3}) = \mathbf{R}^1\Gamma_{\fin}(G_{\mathbb Q, \Sigma}, T_{\upi}^\dagger, \Delta_{\mathscr{F}^4}) = 0\,.  
\end{align*}
\item[ii)] If $\mathscr{F}^\bullet = \mathscr{F}^\bullet_2$, then we have 
\begin{align*}
&\mathbf{R}^1\Gamma_{\fin}(G_{\mathbb Q, \Sigma}, T_{\upi}^\dagger, \Delta_{\mathscr{F}^0})  = 
\mathbf{R}^1\Gamma_{\fin}(G_{\mathbb Q, \Sigma}, T_{\upi}^\dagger, \Delta_{\mathscr{F}^1}) = \mathbf{R}^1\Gamma_{\fin}(G_{\mathbb Q, \Sigma}, T_{\uf_1}^\dagger, \Delta_{\emptyset})_{\cR_\upi}  \oplus 
\mathbf{R}^1\Gamma_{\fin}(G_{\mathbb Q, \Sigma}, T_{\uf_2}^\dagger, \Delta_{F^+})_{\cR_\upi}\,, 
\\
&\mathbf{R}^1\Gamma_{\fin}(G_{\mathbb Q, \Sigma}, T_{\upi}^\dagger, \Delta_{\mathscr{F}^2})  = \mathbf{R}^1\Gamma_{\fin}(G_{\mathbb Q, \Sigma}, T_{\upi}^\dagger, \Delta_{\mathscr{F}^3}) =   \mathbf{R}^1\Gamma_{\fin}(G_{\mathbb Q, \Sigma}, T_{\uf_2}^\dagger, \Delta_{F^+})_{\cR_\upi}\,, 
\\
&\mathbf{R}^1\Gamma_{\fin}(G_{\mathbb Q, \Sigma}, T_{\upi}^\dagger, \Delta_{\mathscr{F}^4}) = 0.  
\end{align*}
\end{proposition}
\begin{proof}
This proposition follows from Proposition \ref{prop:decomp_yohida_Tupi} and \cite[Proposition 6.2]{BS_triple_factor}. 
\end{proof}

\begin{corollary}
\label{cor_2025_08_06_1325}
\eqref{item_AJ_Pi} is false in the situation of Proposition~\ref{prop:yoshida_H^1(Tupi)}$(ii)$. 
\end{corollary}
\begin{proof}
This follows from Lemma \ref{lemma_2025_02_20_1200} and Proposition \ref{prop:yoshida_H^1(Tupi)}(ii). 
\end{proof}

\subsubsection{} Since $\varepsilon(T_{\uf_1}^\dagger)=1 =-\varepsilon(T_{\uf_2}^\dagger)$ by assumption, we have $\varepsilon^{\rm bal}(M_{3,1})=-1=-\varepsilon^{\rm bal}(M_{3,2})$. Let us consider the following hypotheses:
\begin{itemize}
\item[\mylabel{item_AJ_Mi}{$\mathbf{AJ}_{p}^{M_{3,i}}$})] $\res_p\,\delta(M_{3,i}^{\dagger}, \tr^*\Delta_{\underline{\sigma}}) \neq 0 $.
\item[\mylabel{item_parity_Mi}{$\mathbf{P}^-_{M_{3,1}}$})] $\delta(M_{3,1}^{\dagger}, \tr^*\Delta_{\rm bal}) =0$. 
    \end{itemize}
In the same manner as Proposition \ref{prop:yoshida_H^1(Tupi)}, one may also prove the following proposition:

\begin{proposition}\label{prop:yoshida_H^1(M)}
Suppose $\delta(M_{3,2}^{\dagger}, \tr^*\Delta_{\rm bal}) \neq 0$. Assume also that \eqref{item_AJ_Mi} and \eqref{item_parity_Mi} hold $(i=1,2)$.
       \item[i)] If $\mathscr{F}^\bullet = \mathscr{F}^\bullet_1$, then we have 
\begin{align*}
&\mathbf{R}^1\Gamma_{\fin}(G_{\mathbb Q, \Sigma}, M_3^{\dagger}, \tr^*\Delta_a)  = \mathbf{R}^1\Gamma_{\fin}(G_{\mathbb Q, \Sigma}, M_3^{\dagger}, \tr^*\Delta_b) = \mathbf{R}^1\Gamma_{\fin}(G_{\mathbb Q, \Sigma}, M_{3,1}^{\dagger}, \tr^*\Delta_{\rm bal})  \oplus 
\mathbf{R}^1\Gamma_{\fin}(G_{\mathbb Q, \Sigma}, M_{3,2}^{\dagger}, \tr^*\Delta_{\underline{\sigma}}), 
\\
&\mathbf{R}^1\Gamma_{\fin}(G_{\mathbb Q, \Sigma}, M_3^{\dagger}, \tr^*\Delta_c)  = 
\mathbf{R}^1\Gamma_{\fin}(G_{\mathbb Q, \Sigma}, M_3^{\dagger}, \tr^*\Delta_b^\perp)  = \mathbf{R}^1\Gamma_{\fin}(G_{\mathbb Q, \Sigma}, M_{3,1}^{\dagger}, \tr^*\Delta_{\rm bal}), 
\\
&\mathbf{R}^1\Gamma_{\fin}(G_{\mathbb Q, \Sigma}, M_3^{\dagger}, \tr^*\Delta_a^\perp)  = 0. 
\end{align*}
\item[ii)] If $\mathscr{F}^\bullet = \mathscr{F}^\bullet_2$, then we have 
\begin{align*}
&\mathbf{R}^1\Gamma_{\fin}(G_{\mathbb Q, \Sigma}, M_3^{\dagger}, \tr^*\Delta_a)  =  \mathbf{R}^1\Gamma_{\fin}(G_{\mathbb Q, \Sigma}, M_{3,1}^{\dagger}, \tr^*\Delta_{\rm bal})  \oplus 
\mathbf{R}^1\Gamma_{\fin}(G_{\mathbb Q, \Sigma}, M_{3,2}^{\dagger}, \tr^*\Delta_{\underline{\sigma}}), 
\\
&\mathbf{R}^1\Gamma_{\fin}(G_{\mathbb Q, \Sigma}, M_3^{\dagger}, \tr^*\Delta_b) = \mathbf{R}^1\Gamma_{\fin}(G_{\mathbb Q, \Sigma}, M_3^{\dagger}, \tr^*\Delta_c)  = \mathbf{R}^1\Gamma_{\fin}(G_{\mathbb Q, \Sigma}, M_{3,1}^{\dagger}, \tr^*\Delta_{\rm bal}), 
\\
&\mathbf{R}^1\Gamma_{\fin}(G_{\mathbb Q, \Sigma}, M_3^{\dagger}, \tr^*\Delta_b^\perp)   = 
\mathbf{R}^1\Gamma_{\fin}(G_{\mathbb Q, \Sigma}, M_3^{\dagger}, \tr^*\Delta_a^\perp)  = 0. 
\end{align*}
\end{proposition}


\begin{corollary}
\label{cor_2025_08_07_0919}
Suppose that the hypotheses of Proposition~\ref{prop:yoshida_H^1(M)} are valid. If $\mathscr{F}^\bullet = \mathscr{F}^\bullet_1$ $($resp. if $\mathscr{F}^\bullet = \mathscr{F}^\bullet_2$$)$, then $c_M$ is unit and $d_M = 0$ $($resp.  $c_M = 0$ and $d_M$ is a unit$)$. In particular,
      \eqref{item_AJ_M} is false when $\mathscr{F}^\bullet = \mathscr{F}^\bullet_1$. 
\end{corollary}

\begin{remark}
    \label{rem_2025_09_23_1946}
    Under our assumptions on the global root numbers, the hypotheses of Propositions~\ref{prop:yoshida_H^1(Tupi)} and \ref{prop:yoshida_H^1(M)} are consequences of a natural extension of Greenberg's conjectures. We also note that  \eqref{item_nVdeltaT3} is equivalent to four assumptions \eqref{item_AJ_f} and \eqref{item_AJ_Mi} for $i=1,2$.
\end{remark}

\subsubsection{} 
\label{subsubsec_2025_08_06_1543}
Recall that we have a canonical trivialisation (see \cite[(6.8)]{BS_triple_factor}) 
\[
\Log_{\omega_{\uf_i}} \colon H^1(G_p,F^+T_{\uf_i}^\dagger)_{\cR_\upi} \stackrel{\sim}{\longrightarrow} \cR_{\upi}. 
\]
Moreover, on setting $F^+_{\rm bal}M_{3,i}^\dagger := \mathrm{tr}^*(F_{\rm bal}^+T_{3,i}^\dagger )$, we also obtain a surjection
\[
\Log_{\omega_{M_{3,i}}}\,:\, H^1(G_p, F^+_{\rm bal}M_{3,i}^\dagger) \longrightarrow 
H^1(G_p, F^+T_{\uf_i}^\dagger \,\widehat\otimes\, 
{\mathscr Gr}^1 \ad^0(T_{\underline{\sigma}}))_{\cR_\upi}
\xrightarrow{\Log_{\omega_{\uf_i}}}
\cR_{3}\,.
\]

\begin{theorem}\label{thm:delta_decomp_yoshida}
Assume \eqref{item_parity_f} and \eqref{item_parity_Mi}.
     Then,
\[
\delta(T_3^\dagger, \Delta_{a})  =  \Log_{\omega_{M_{3,1}}} (
\delta(M_{3,1}^\dagger, \tr^*\Delta_{\underline{\sigma}}))\times \delta(M_{3,2}^\dagger, \tr^*\Delta_{\rm bal}) \times 
\delta(T_{\uf_1}^\dagger, \Delta_{F^+})_{\cR_3} \times \Log_{\omega_{\uf_2}}\,\delta(T_{\uf_2}^\dagger, \Delta_{\emptyset})_{\cR_3}\,.
\]    
\end{theorem}

We remark that under our running assumptions \eqref{item_parity_f} and \eqref{item_parity_Mi}, if one of \eqref{item_AJ_f} or \eqref{item_AJ_Mi} fails for some $i\in \{1,2\}$ (which is equivalent to say that $\eqref{item_nVdeltaT3}$ fails), then the asserted equality reduces to $0 = 0$. 

\begin{proof}[Proof of Theorem~\ref{thm:delta_decomp_yoshida}]
Proposition~\ref{prop:delta_a-formula2} together with Corollaries \ref{cor:decomp_delta_yohida_Tupi} and \ref{cor:decomp_delta_yohida_M} imply that 
\[
\delta(T_3^\dagger, \Delta_{a})  = \Log^{(4)}_{\mathscr{F}^0}(\delta(M_{3,1}^{\dagger}, \tr^*\Delta_{\underline{\sigma}}) \otimes  \delta(M_{3,2}^{\dagger}, \tr^*\Delta_{\underline{\sigma}}) \otimes  
\delta(T_{\uf_1}^\dagger, \Delta_{\emptyset})_{\cR_3} \otimes  \delta(T_{\uf_2}^\dagger, \Delta_{\emptyset})_{\cR_3} )\,,
\]
where $\Log^{(4)}_{\mathscr{F}^0}$ is as in  \S\ref{sec:delta_a_decomp}. 
Since we have the decomposition $T_3^\dagger = T_{3,1}^\dagger \oplus T_{3,2}^\dagger$, we obtain 
\[
\delta(T_3^\dagger, \Delta_{a}) = 
\Log^{(2)}_{1}( \delta(M_{3,1}^{\dagger}, \tr^*\Delta_{\underline{\sigma}}) \otimes \delta(T_{\uf_1}^\dagger, \Delta_{\emptyset})_{\cR_3}) \times 
   \Log^{(2)}_2( \delta(M_{3,2}^{\dagger}, \tr^*\Delta_{\underline{\sigma}})  \otimes \delta(T_{\uf_2}^\dagger, \Delta_{\emptyset})_{\cR_3} ). 
\]
Here, 
\begin{align*}
\Log^{(2)}_{i} \colon \mathbf{R}^1\Gamma_{\fin}(G_{\mathbb Q, \Sigma}, M_{3,i}^{\dagger}, \Delta_{\underline{\sigma}}) &\otimes \mathbf{R}^1\Gamma_{\fin}(G_{\mathbb Q, \Sigma}, T_{\uf_i}^\dagger, \Delta_{\emptyset})_{\cR_3}   
\\
&\stackrel{\sim}{\longrightarrow} {\bigcap}^2_{\cR_3} (\mathbf{R}^1\Gamma_{\fin}(G_{\mathbb Q, \Sigma}, M_{3,i}^{\dagger}, \Delta_{\underline{\sigma}}) \oplus \mathbf{R}^1\Gamma_{\fin}(G_{\mathbb Q, \Sigma}, T_{\uf_i}^\dagger, \Delta_{\emptyset})_{\cR_3}) 
\\
&\longrightarrow {\bigcap}^2_{\cR_3} H^1(G_{p}, T_{\uf_i}^{\dagger})_{\cR_3}  \cong \cR_3.     
\end{align*}
We next define the morphism 
$$\Log_{F^-} \,\colon\, H^1(G_p, T_{\uf_1}^\dagger)_{\cR_3} \longrightarrow H^1(G_p,F^-T_{\uf_1}^\dagger)_{\cR_3} \stackrel{\sim}{\longrightarrow} \cR_3\,.$$ 
Since $\delta(M_{3,1}^{\dagger}, \tr^*\Delta_{\rm bal}) = 0$ by the assumption \eqref{item_parity_Mi}, it follows from Theorem \ref{thm_2025_02_19_2217} that 
\[
 \delta(M_{3,1}^{\dagger}, \tr^*\Delta_{\underline{\sigma}})  \in 
 \mathbf{R}^1\Gamma_{\fin}(G_{\mathbb Q, \Sigma}, M_{3,1}^{\dagger}, \tr^*\Delta_{\rm bal}), 
\]
and hence 
\[
\Log_{F^-}(\mathrm{res}_p(\delta(M_{3,1}^{\dagger}, \tr^*\Delta_{\underline{\sigma}}))) = 0.   
\]
The definition of $\Log^{(2)}_1$, together with Theorem \ref{thm_2025_02_19_2217}, shows that 
\begin{align*}
 \Log^{(2)}_{1}( \delta(M_{3,1}^{\dagger}, \tr^*\Delta_{\underline{\sigma}}) \otimes \delta(T_{\uf_1}^\dagger, \Delta_{\emptyset})_{\cR_3})   &= \Log_{\omega_{M_{3,1}}}(\delta(M_{3,1}^\dagger, {\rm tr}^*\Delta_{\underline{\sigma}}))  \times        
    \Log_{F^-} (\mathrm{res}_p(\delta(T^\dagger_{\uf_1},\Delta_{\emptyset})_{\cR_3}) )
    \\
    &= \Log_{\omega_{M_{3,1}}}(\delta(M_{3,1}^\dagger, {\rm tr}^*\Delta_{\underline{\sigma}}))  \times        
    \delta(T^\dagger_{\uf_1},\Delta_{F^+})_{\cR_3}. 
\end{align*}

On the other hand, since $\delta(T^\dagger_{\uf_2}, \Delta_{F^+}) = 0$ by the assumption \eqref{item_parity_f}, it follows from Theorem \ref{thm_2025_02_19_2217} that $\Log_{F^-}\circ \mathrm{res}_p(\delta(T_{\uf_2}^\dagger, \Delta_{\emptyset})_{\cR_3}) = 0$.   
Therefore, by the same argument, we have
    \[
     \Log^{(2)}_2( \delta(M_{3,2}^{\dagger}, \tr^*\Delta_{\underline{\sigma}})  \otimes \delta(T_{\uf_2}^\dagger, \Delta_{\emptyset})_{\cR_3} ) = \delta(M_{3,2}^\dagger, \tr^*\Delta_{\rm bal}) \times 
 \Log_{\omega_{\uf_2}}\delta(T_{\uf_2}^\dagger, \Delta_{\emptyset})_{\cR_3}
    \]
    (see also \cite{BS_triple_factor}, Theorem 6.7).  
\end{proof}

\begin{remark}\label{rem:yoshida_log}
    Suppose that $\mathscr{F} = \mathscr{F}_1^{\bullet}$. 
    Then by Proposition \ref{prop:yoshida_H^1(Tupi)}(i), we have 
    $\Log_{\omega_{\uf_2}} = \Log_{\rm c}$. 
    On the other hand, Proposition \ref{prop:yoshida_H^1(M)}(i) implies  that 
    \[
    \Log_{M_{3,1}} = \Log_{\mathscr{F}^3} \neq \Log_{\rm c}. 
    \]
    Here $\Log_{\mathscr{F}^3}$ is the homomorphism induced by the isomorphism $H^1(G_p, \mathscr{F}^3T_\upi^\dagger) \stackrel{\sim}{\longrightarrow} \cR_\upi$. In other words, in the endoscopic cases, the trivialisations of the modules of the relevant leading terms dwell on Eichler--Shimura isomorphisms in different degrees of cohomology.

In the supplementary case when $\mathscr{F} = \mathscr{F}_{2}^{\bullet}$, the situation is reversed; namely, we have $\Log_{M_{3,1}} = \Log_{\mathrm{c}}$, 
and $\Log_{\omega_{\underline{\sigma}_{2}}} = \Log_{\mathscr{F}^3} \neq \Log_{\mathrm{c}}$.   
\hfill$\blacksquare$
\end{remark}

Let us put
\begin{align*}
    \Delta_{M_3^\dagger}^{(c), {\rm alg}} := 
\delta(M_{3,1}^\dagger, \tr^*\Delta_{\underline{\sigma}}) \times \delta(M_{3,2}^\dagger, \tr^*\Delta_{\rm bal})\,,\qquad 
    \Delta_\upi^{\dagger,  {\rm alg}} := \delta(T_{\uf_1}^\dagger, \Delta_{F^+}) _{\cR_\upi} \times \delta(T_{\uf_2}^\dagger, \Delta_{\emptyset})_{\cR_\upi}. 
\end{align*}
Moreover, we define the homomorphisms $\Log_{\omega_{M_3}}$ and $\Log_{\omega_{\upi}}$ according to the choice of the filtration $\mathscr{F}$ as follows: 
\begin{align*}
    \Log_{\omega_{M_3}} := \begin{cases}
        \Log_{\mathscr{F}^3}  & \textrm{ if } \mathscr{F} = \mathscr{F}_1^{\bullet}, 
        \\
        \Log_c & \textrm{ if } \mathscr{F} = \mathscr{F}_2^{\bullet}, 
    \end{cases} 
\quad \quad
    \Log_{\omega_{\upi}}  := \begin{cases}
        \Log_c & \textrm{ if } \mathscr{F} = \mathscr{F}_1^{\bullet}, 
        \\
        \Log_{\mathscr{F}^3}  & \textrm{ if } \mathscr{F} = \mathscr{F}_2^{\bullet}. 
    \end{cases} 
\end{align*}


Then, as a corollary of Theorem \ref{thm:delta_decomp_yoshida} (see also Remark \ref{rem:yoshida_log}), we obtain the following result.

\begin{corollary}
\label{cor_2025_08_14_1551}
In the situation of Theorem~\ref{thm:delta_decomp_yoshida}, we have
\[
\delta(T_3^\dagger, \Delta_{a})  =   \Log_{\omega_{M_3}}( \Delta_{M_3^\dagger}^{(c), {\rm alg}}) \times  \Log_{\omega_{\upi}}( \Delta_\upi^{\dagger,  {\rm alg}})
\,.
\]    
\end{corollary}

\begin{remark}
    When $\varepsilon(T_{\uf_1}^\dagger)= -1 =-\varepsilon(T_{\uf_2}^\dagger)$, the contents of this subsection remain valid if we exchange the roles of $\uf_1$ and $\uf_2$. 
\end{remark}



\subsection{Yoshida lifts (bis)}
\label{subsec_2025_08_06_1504_bis}
We continue to work in the setting of \S\ref{subsec_2025_08_06_1504}, but now to supplement \S\ref{subsec_2025_08_06_1504} to revisit the factorisation problem over the region (d) in this endoscopic scenario and explain that it reduces to \cite{BS_triple_factor,bcpv}. We have the decompositions \eqref{eqn_2025_08_19_0754_1}, \eqref{eqn_2025_08_19_0754} and \eqref{eqn_2025_08_19_0757} as before. Moreover,
\begin{itemize}
\item[-] region (d) corresponds to the $\underline{\sigma}^c$-dominant region for $T_{4,2}^\dagger$ and $\uf_1$-dominant region for $T_{4,1}^\dagger$\,;
\item[-] region (e) corresponds to the balanced region for $T_{4,2}^\dagger$ and  $\uf_1$-dominant region for $T_{?,1}^\dagger$\,. 
\end{itemize}

Based on this observation, \eqref{eqn_2025_08_19_0754_1}, and \eqref{eqn_2025_08_19_0757}, we set 
$$ L_p^{(d)}(T_?^\dagger):= L_p^{(\uf_1)}(T_{?,1}^\dagger)\times L_p^{(\underline{\sigma}^c)}(T_{?,2}^\dagger)\,,\quad L_p^{(d)}(M_3^\dagger):= L_p^{(\uf_1)}(M_{3,1}^\dagger)\times L_p^{\rm bal}(M_{3,2}^\dagger)\,,\qquad ?=3,4\,.$$

\subsubsection{} As we work under the assumption that the global root number of the specialisations of $T_4^\dagger$ in region (e) equals $-1$, we have
\begin{align*}
    -1=\varepsilon^{(e)}(T_4^\dagger)=\varepsilon^{\uf_1}(T_{4,1}^\dagger)\,\cdot\,\varepsilon^{\rm bal}(T_{4,2}^\dagger)\,=\varepsilon^{\uf_1}(M_{3,1})\,\varepsilon(T_{\uf_1})\,\cdot\,\varepsilon^{\rm bal}(M_{3,2})\,\varepsilon(T_{\uf_2})\,.
\end{align*}
That means 
$$-\varepsilon^{\uf_1}(T_{4,1}^\dagger)=\varepsilon^{\rm bal}(T_{4,1}^\dagger)=\varepsilon^{\rm bal}(T_{4,2}^\dagger)=\pm 1$$
and we assume that it is $-1$. In view of \eqref{eqn_2025_08_19_0840}, we then have
$$\varepsilon^{\uf_1}(M_{3,1})=1=\varepsilon^{\rm bal}(M_{3,2})\,.$$

\begin{remark}
Unlike in the non-endoscopic cases, classical specialisations of $M_3^\dagger$ in the region (d) are critical in the sense of Deligne: As the discussion above shows, if $\Pi$ is a Yoshida lift of $(f_1,f_2)$ as above, with the propery that $\Pi\times \sigma\times\sigma^c$ has weights in the region (d), then the $\GL_2 \times \GL_3$-summand $f_1\times \ad^0\sigma$ is $f_1$-dominant, whereas $f_2\times \ad^0\sigma$ is balanced. 
\end{remark}

\begin{theorem}
    \label{thm_2025_08_21_1108}
    In the setting of the present subsection, let us assume that \cite[Conjecture 5.1]{bcpv} (extension of Gross--Kudla conjecture) as well as \eqref{item_non_anom} hold true. Then,
\begin{align}
\begin{aligned}
    L_p^{(d)}(T_?^\dagger)^2&:=L_p^{(\uf_1)}(T_{3,1}^\dagger)^2\times L_p^{(\underline{\sigma}^c)}(T_{3,2}^\dagger)^2\\
    &=\left(L_p^{(\uf_1)}(M_{3,1}^\dagger)\times L_p(\uf_1)\right)\times\left( L_p^{\rm bal}(M_{3,2}^\dagger)\times \Log_{\omega_{\uf_2}} (\Delta_{\uf_2}^\dagger)^2\right) \qquad \pmod{\cR_\upi[1/p]^\times}\\
    &=\underbrace{\left(L_p^{(\uf_1)}(M_{3,1}^\dagger)\times L_p^{\rm bal}(M_{3,2}^\dagger) \right)}_{L_p^{(d)}(M_3^\dagger)} \,\times\,\underbrace{\left(L_p(\uf_1)\times \Log_{\omega_{\uf_2}} (\Delta_{\uf_2}^\dagger)^2 \right)}_{\Log_{\omega_\upi}(\Delta_\upi^\dagger)^2} \qquad \pmod{\cR_\upi[1/p]^\times}\,,
\end{aligned}
\end{align}
where $\Delta_{\uf_2}^\dagger$ and $\Delta_\upi^\dagger$ are as in the statement of Conjecture~\ref{conj_2025_08_06_1401}.
\end{theorem}
In particular, a slight variant of Conjecture~\ref{conj_2025_08_21_1552} holds true in this scenario under the hypothesis of our theorem. In fact, slightly more is true: The ambiguity (which is a factor in $\cR_\upi[1/p]^\times$) specialises at classical points to an explicit algebraic number.
\begin{proof}
   The second equality follows from a direct comparison of interpolation formulae of $L_p^{(f_1)}(T_{3,1}^\dagger)$ and $L_p^{(f_1)}(M_{3,1}^\dagger)\times L_p(\uf_1)$, combined with the main result of \cite{bcpv} (Theorem~7.3 combined with Proposition 6.1 in op. cit.).
\end{proof}


\subsection{Saito--Kurokawa lifts}
\label{subsec_2025_08_25_1211}
In this subsection, we consider the case when $\upi$ is a (one-parameter) family of Saito--Kurokawa lifts of a Hida family $\uf$ of ordinary modular forms. We remark that the assumptions of \S\ref{subsec:families} exclude the case of families of Saito--Kurokawa lifts, and our definition of a family of degree-16 $p$-adic $L$-function in this scenario, as in the case of \S\ref{subsec_2025_08_06_1504} and \S\ref{subsec_2025_08_06_1504_bis}, is somewhat ad hoc. (cf. \S\ref{subsubsec_2025_08_21_1640}). 

\subsubsection{} Let us put $T_\upi^\dagger=T_{\uf}^\dagger \oplus \mathbb Z_p \oplus \mathbb Z_p(1)$, where $\uf$ is a Hida family. 
In this section, we work under the assumption that 
$$ \varepsilon(\uf) = -1= \varepsilon^{\text{bal}}(\uf \otimes \underline{\sigma} \otimes \underline{\sigma}^c)\,, $$ 
where $\varepsilon(\uf)$ is the common global root number of the family $\uf$, and $\varepsilon^{\text{bal}}(\uf \otimes \underline{\sigma} \otimes \underline{\sigma}^c)$ is the same for the family $\uf \otimes \underline{\sigma} \otimes \underline{\sigma}^c$ at those weights that are balanced (i.e., the sum of any two of them is greater than the third one).

\subsubsection{} The Galois representation $T_2^\dagger$ (resp. $T_3^\dagger$), which is free of rank $16$ over $\cR_2:=\cR_\uf\widehat{\otimes}\cR_{\underline{\sigma}}$ (resp. over $\cR_3:=\cR_\uf\widehat{\otimes}\cR_{\underline{\sigma}}\widehat{\otimes}\cR_{\underline{\sigma}}$), admits the following decomposition:
\begin{equation}
\label{eq-decomposition_1}
T_?^\dagger\simeq \underbrace{\left(T_{\underline{\mathbf f}}^\dagger\widehat \otimes T_{\underline{\sigma}} \widehat\otimes T_{\underline{\sigma}}^*\right)}_{T_{?,1}^\dagger} \oplus \underbrace{\left(T_{\underline{\sigma}} \widehat\otimes T_{\underline{\sigma}}^* \right)\widehat\otimes_{\cR_{\underline{\sigma}}} \cR_?}_{T_{?,2}^\dagger} \oplus \underbrace{\left( T_{\underline{\sigma}} \widehat\otimes T_{\underline{\sigma}}^*(1) \right)\widehat\otimes_{\cR_{\underline{\sigma}}} \cR_?}_{T_{?,3}^\dagger} \,,\qquad ?=2,3.
\end{equation}
If $?=2$, we can further decompose each summand above as
\begin{equation}
\label{eq-decomposition}
T_2^\dagger\simeq \underbrace{\left((T_{\underline{\mathbf f}}^\dagger\,\widehat\otimes\, {\rm ad}^0 T_{\underline{\sigma}})\,\oplus \,(T_{\underline{\mathbf f}}^\dagger\otimes_{\cR_\uf}\cR_2)\right)}_{T_{2,1}^\dagger}\oplus \underbrace{\left( {\rm ad}^0 T_{\underline{\sigma}} \oplus \mathbb Q_p \right)\otimes_{\cR_{\underline{\sigma}}} \cR_2}_{T_{2,2}^\dagger} \oplus \underbrace{\left( {\rm ad}^0 T_{\underline{\sigma}}(1) \oplus \mathbb Q_p(1) \right)\otimes_{\cR_{\underline{\sigma}}} \cR_2}_{T_{2,3}^\dagger} \,.
\end{equation}
Observe that the degree--4 summands $T^\dagger_{?,2}$ and $T^\dagger_{?,3}$ are independent of variation in $\uf$.
\subsubsection{} 
\label{subsubsec_2025_08_21_1637}
We remark that:
\begin{itemize}
\item[-] Region (a) corresponds to the $\underline{\sigma}^c$-dominant region for $T_{3,1}^\dagger$. We note that we have $k_1=k=k_2$ in the notation of \S\ref{subsubsec_2025_08_19_0801}.

\item[-] Region (b) is given by the condition $2 \leq c_2 - c_1 \leq 2k - 4$ (since $k_1 = k_2$), together with $2k \leq c_1 + c_2$. This is equivalent to $c_1<c_2< c_1 + (2k - 2)$, and $2k - 2 < c_1 + c_2$;  hence (since $k\geq 2$), it corresponds to the balanced region for $T_{3,1}^\dagger$.

\item[-] Region (c) in this case contains only points with $c_1 = c_2$ (since $k_1 = k_2$). In this situation, setting $c:=c_1=c_2$, the region is defined by the constraint $c \geq k$, and therefore coincides with the balanced region for $T_{3,1}^\dagger$.
\end{itemize}

\subsubsection{}
\label{subsubsec_2025_08_21_1638}
Let $L_p({\rm ad}^0 \underline{\sigma})(\lambda,s)$ denote the adjoint $p$-adic $L$-function  in two variables, constructed by Schmidt \cite{schmidt} and Hida \cite{hida90}, where $s$ stands for the cyclotomic variable and $\lambda$ for the parameter of the Hida family $\underline{\sigma}$. In view of \cite[Theorem 2]{Das}, we have
\begin{equation}
\label{eqn_2025_09_01_1355}
   L_p(\underline{\sigma} \otimes \underline{\sigma}^c)(\lambda,\lambda,\wt(\lambda)) = L_p'({\rm ad}^0 \underline{\sigma})(\lambda,1)\times \res_{s=1}\zeta_p(s)=L_p'({\rm ad}^0 \underline{\sigma})(\lambda,1)\times (1-p^{-1})\,, 
\end{equation}
where $\zeta_p(s)$ is the $p$-adic Riemann zeta function, which has a pole at $s=1$ with residue $1-p^{-1}$.

\subsubsection{} 
\label{subsubsec_2025_08_21_1651}
According to \cite[Theorem 9.4]{Das}, the two variable $p$-adic $L$-function 
$$L_p(\ad\,\underline{\sigma})(\lambda,s):=L_p(\underline{\sigma} \otimes \underline{\sigma}^c)(\lambda,\lambda,\wt(\lambda)-1+s)$$
verifies a functional equation 
\begin{equation}
\label{eqn_2025_09_01_0950}
    L_p(\ad\,\underline{\sigma})(\lambda,s) = \epsilon(\ad\underline{\sigma}_\lambda)(s) \cdot L_p(\ad\,\underline{\sigma}(1))(\lambda,-s)\,, 
\end{equation} 
where $L_p(\ad\,\underline{\sigma}(1))(\lambda,s):=L_p(\ad\,\underline{\sigma})(\lambda,1+s)$, and $\epsilon(\ad\underline{\sigma}_\lambda)(s)$ is the epsilon factor described in \cite[Theorem 9.4]{Das}. More precisely, for any crystalline specialisation $\lambda$ of $\underline{\sigma}$, we have $\epsilon(\ad\underline{\sigma}_\lambda)(s) = A_\lambda \cdot B_\lambda^{\wt(\lambda)-s}$, where $A_\lambda \in \overline{\mathbb Q}^{\times}$ and $B_\lambda$ is the conductor of the Gelbart--Jacquet lift of $\underline{\sigma}_\lambda$. 

Let us put $\epsilon(\ad\underline{\sigma}_\lambda) := \epsilon(\ad\underline{\sigma}_\lambda)(0)$. Substituting $s=0$ in \eqref{eqn_2025_09_01_0950}, we obtain
\begin{equation}
\label{equation_2025_08_27_1626_pre}
L_p(\ad\,\underline{\sigma})_{\vert s=0}(\lambda) = \epsilon(\ad\underline{\sigma}_\lambda) \cdot L_p(\ad\,\underline{\sigma}(1))_{\vert s=0}(\lambda)=L_p(\ad\,\underline{\sigma})_{\vert s=1}(\lambda)\,.
\end{equation}

\subsubsection{}  We assume until the end of our paper that any one of the following equivalent conditions holds true:
\begin{itemize}
\item[i)] The family $\underline \sigma$ is non-CM. 
    \item[ii)] $\underline{\sigma}_\lambda$ is non-CM for some (any) classical specialisation $\lambda$ of $\sigma$ of weight at least $2$.
    \item[iii)] If there exists a Dirichlet character $\chi=\chi_\lambda$ such that for some (any) classical specialisation $\lambda$ of $\sigma$ of weight at least $2$ we have $\underline{\sigma}_\lambda\simeq \underline{\sigma}_\lambda\otimes\chi$, then $\chi=\mathds{1}$ is the trivial character.
     \item[iv)] $\ad^0\rho_\lambda$ is irreducible for for some (any) classical specialisation $\lambda$ of $\sigma$ of weight at least $2$.
\end{itemize} 
The equivalence of (i) and (ii) follows from \cite[Proposition 8]{ghatevatsal}, whereas the equivalence of (ii) and (iii) is well-known; cf. \cite{ribet1977twists,ribet1980twists}. Finally, the equivalence of (iii) and (iv) is an immediate consequence of Schur's lemma.

In this case, the epsilon factor (at $s=0$) $\epsilon(\ad^0 \underline{\sigma}_\lambda)$ interpolates to an analytic function $\epsilon(\ad^0\underline{\sigma})$ of $\lambda$, cf. \cite[Theorem C]{mundy}. As we have $\epsilon(\ad\underline{\sigma}_\lambda)=\epsilon(\ad^0\underline{\sigma}_\lambda)$ (since the epsilon factor of the Riemann zeta function at $s=0$ equals to $1$), it follows that $\epsilon(\ad\underline{\sigma}_\lambda)$ interpolates to an analytic function $\epsilon(\ad\underline{\sigma})=\epsilon(\ad^0\underline{\sigma})$ of $\lambda$. We may then rewrite \eqref{equation_2025_08_27_1626_pre} as 
\begin{equation}
\label{equation_2025_08_27_1626}
L_p(\ad\,\underline{\sigma})_{\vert s=0} = \epsilon(\ad\underline{\sigma}) \cdot L_p(\ad\,\underline{\sigma})_{\vert s=1}\,.
\end{equation}


\subsubsection{} 
\label{subsubsec_2025_08_21_1640}
Based on the observations in \S\ref{subsubsec_2025_08_21_1637}--\S\ref{subsubsec_2025_08_21_1651}, together with \eqref{eq-decomposition_1} and \eqref{eq-decomposition}, we set 
$$ L_p^{(a)}(T_?^\dagger)^2:= L_p^{(\underline{\sigma}^c)}(T_{?,1}^\dagger)^2\times L_p^{(\underline{\sigma}^c)}(T_{?,2}^\dagger)\times L_p^{(\underline{\sigma}^c)}(T_{?,3}^\dagger) \,,\qquad ?=2,3\,.$$

\subsubsection{} We can finally concluse with the following factorisation statement. Note that $L_p^{(\underline{\sigma}^c)}(T_{3,2}^\dagger)$ is the $p$-adic $L$-function $L_p(\ad\,\underline{\sigma})_{\vert s=0}$ (of Hida) and $L_p^{(\underline{\sigma}^c)}(T_{3,3}^\dagger)$ equals $L_p(\ad\,\underline{\sigma}(1))_{\vert s=0}=L_p(\ad\,\underline{\sigma})_{\vert s=1}$.

\begin{proposition}
\label{prop_2025_08_14_1443}
In the setting of the present subsection, assume that \cite[Conjecture 5.1]{bcpv} (extension of Gross--Kudla conjecture) holds true. Then,
$$ \begin{aligned} L_p^{(a)}(T_3^\dagger)^2 & := L_p^{(\underline{\sigma}^c)}(T_{3,1}^\dagger)^2\times L_p^{(\underline{\sigma}^c)}(T_{3,2}^\dagger) \times L_p^{(\underline{\sigma}^c)}(T_{3,3}^\dagger) 
\\ & =  L_p^{\rm bal}(M_{3,1}^\dagger)\times  L_p({\rm ad} \, \underline{\sigma})_{\vert s=0} \times L_p({\rm ad} \, \underline{\sigma}(1))_{\vert s=0} \times \Log (\Delta_{\uf}^\dagger)^2 \qquad \pmod{\cR_{\uf}[1/p]^\times}
\\ & =  \underbrace{\left(\epsilon(\ad\underline{\sigma})\cdot L_p^{\rm bal}(M_{3,1}^\dagger)\times  L_p'({\rm ad}^0 \underline{\sigma})_{\vert s=1}^2 \right)}_{\Log_{\omega_{M_3}}(\delta_{M_3^\dagger}^{(c)})^2} \times \underbrace{\left( \Log (\Delta_{\uf}^\dagger)^2 \times (1-p^{-1})^2 \right)}_{\Log_{\omega_\upi}(\Delta_\upi^\dagger)^2}  \qquad \pmod{\cR_{\uf}[1/p]^\times} \,. \end{aligned}
$$
\end{proposition}

As in the case of Theorem~\ref{thm_2025_08_21_1108}, slightly more is true: The ambiguity (which is a factor in $\cR_{\uf}[1/p]^\times$) specialises at classical points to an explicit algebraic number.

\begin{proof}
The factorisation 
$$L_p^{(\underline{\sigma}^c)}(T_{3,1}^\dagger)=\mathcal C \cdot L_p^{\rm bal}(M_{3,1}^\dagger)\times \Log (\Delta_{\underline{\mathbf f}}^\dagger)$$
is proved in \cite{BS_triple_factor,bcpv} under the running hypotheses, whereas
\begin{equation}
\label{eqn_2025_08_26_1204}
    L_p^{(\underline{\sigma}^c)}(T_{3,3}^\dagger)=L_p(\ad\,\underline{\sigma}(1))_{\vert s=0}=L_p'({\rm ad}^0 \underline{\sigma}(1))_{\vert_{s=0}} \times (1-p^{-1})=L_p'({\rm ad}^0 \underline{\sigma})_{\vert_{s=1}} \times (1-p^{-1})
\end{equation}
is a consequence of Dasgupta's results \cite{Das} that we have reviewed in \S\ref{subsubsec_2025_08_21_1638} and \S\ref{subsubsec_2025_08_21_1651}. Finally, 
$$ L_p^{(\underline{\sigma}^c)}(T_{3,2}^\dagger)=L_p(\ad\,\underline{\sigma})_{\vert s=0}=\epsilon(\ad\underline{\sigma}) \cdot L_p(\ad\,\underline{\sigma})_{\vert s=1} = \epsilon(\ad\underline{\sigma}) \cdot L_p'({\rm ad}^0 \underline{\sigma})_{\vert s=1} \times (1-p^{-1})\,, $$ 
where the second equality follows from \eqref{equation_2025_08_27_1626}, and the third from \eqref{eqn_2025_08_26_1204}.
\end{proof}

\subsubsection{}
We compare the factorisation statement in Proposition~\ref{prop_2025_08_14_1443} with that in Conjecture~\ref{conj_2025_08_22_0031}  in the generic scenario, explaining the meaning of $\Log_{\omega_{M_3}}(\delta_{M_3^\dagger}^{(c)})$ and $\Log_{\omega_\upi}(\Delta_\upi^\dagger)$. 

In Proposition~\ref{prop_2025_08_14_1443}, we have set $\Delta_{\upi}^\dagger:=(1-p^{-1})\,\Delta_{\uf}^\dagger$, where $\Delta_{\uf}^\dagger$ is the family of Heegner cycles as before.  The cohomology class $\delta_{M_3^\dagger}^{(c)}$ is expected to arise as an ``improved Beilinson--Flach class'' $\BF_{\ad \, \underline{\sigma}(1)}^*$ scaled by the square-root of $\epsilon(\ad\underline{\sigma})\cdot L_p^{\rm bal}(M_{3,1}^\dagger)$. We remark that the Beilinson--Flach elemement associated to $\ad^0 T_{\underline{\sigma}}(1)$ vanishes identically (due to the presence of an exceptional zero), and one expects an explicit reciprocity law linking $\BF_{\ad \, \underline{\sigma}(1)}^*$ to 
$$L_p(\ad \, \underline{\sigma})_{\vert s=0} = \epsilon(\ad\underline{\sigma})\cdot(1-p^{-1}) \cdot L_p'(\ad^0 \underline{\sigma})_{\vert s=1}\,,$$
where the equality above follows from \eqref{eqn_2025_09_01_1355} combined with \eqref{equation_2025_08_27_1626}. The construction of the improved class $\BF_{\ad \, \underline{\sigma}(1)}^*$ as well as the reciprocity law alluded to above will be addressed in future work.

\subsubsection{} Observe that in this case, the algebraic counterpart of Proposition~\ref{prop_2025_08_14_1443} reduces to a combination of Palvannan's main results and \cite{palvannan18} and \cite[Theorem 6.7]{BS_triple_factor}.

\subsubsection{}
The factorisation problem for $L_p(16d)$ in the setting of \S\ref{subsec_2025_08_25_1211} reduces to the tautological equality $0=0$ as a consequence of our assumptions on $\varepsilon$-factors.

\bibliographystyle{amsalpha}
\bibliography{BRS-refs}

\end{document}